\documentclass[11pt,a4paper]{article}
\usepackage{lmodern}
\usepackage[utf8]{inputenc}
\usepackage[english]{babel}
\usepackage[T1]{fontenc}
\usepackage{amsmath, amsfonts, amssymb, amsthm}
\usepackage{fullpage}
\usepackage{microtype}
\usepackage{enumitem}
\setenumerate{label=$(\roman*)$}
\usepackage{mathtools} 
\usepackage{tikz}
\usetikzlibrary{decorations.markings}
\usepackage[smalltableaux]{ytableau}
\usepackage[colorlinks=true, linkcolor=blue]{hyperref}

\numberwithin{equation}{section}
\numberwithin{figure}{section}

\newtheorem{theorem}[equation]{Theorem}
\newtheorem{theoremintro}{Theorem}

\newtheorem{proposition}[equation]{Proposition}

\newtheorem{corollary}[equation]{Corollary}
\newtheorem{corollaryintro}[theoremintro]{Corollary}

\newtheorem{lemma}[equation]{Lemma}

\theoremstyle{definition}
\newtheorem{definition}[equation]{Definition}

\theoremstyle{remark}
\newtheorem{remark}[equation]{Remark}
\newtheorem{example}[equation]{Example}

\newcommand\partsn[1][n]{\mathcal{P}_{#1}}
\newcommand\Y{\mathcal{Y}}
\newcommand\nodes{\mathbb{Z}_{\geq 1}^2}
\newcommand\ladnode[2][e]{\mathrm{lad}_{#1}(#2)}

\newcommand\reg[1][e]{\mathrm{reg}_{#1}}

\newcommand\tomega{\widetilde{\omega}}
\newcommand\Pl[1][n]{\mathrm{Pl}_{#1}}
\newcommand\R{\mathbb{R}}
\newcommand\lambdareg{\mu}

\newcommand\myepsilon\eta

\newcommand\hp[1][k]{h_{#1}^+}
\newcommand\hm[1][k]{h_{#1}^-}
\newcommand\myLk[2]{L_e^{#1,(#2)}}
\newcommand\myL[1]{\myLk{#1}{k}}
\newcommand\LM[2][\alpha]{L_{#1}^{#2}}
\newcommand\Lx[2][\alpha]{L_{#1}^{#2}}

\DeclareMathOperator\Shop{Sh}
\DeclareMathOperator\shop{sh}
\newcommand\Sh[1][\alpha]{\Shop_{#1}}
\newcommand\Shp[1][\alpha]{\Shop'_{#1}}
\newcommand\sh[1][e]{\shop_{#1}}
\newcommand\dire[1][e]{\alpha_e}
\newcommand\gr{\mathbb{Y}}
\newcommand\epi{\mathrm{epi}}

\newcommand\suppfalpha[1][\alpha]{s_{#1}}

\newcommand\Omegae[1][e]{\Omega_{#1}} 
\newcommand\Omegaeps[2][\myepsilon]{\Omega^{#2 #1}}
\newcommand\Omegaepse[2][\myepsilon]{\Omega^{#2 #1}_{e}}
\newcommand\Omegaplot{2/pi*(\x*rad(asin(\x/2))+sqrt(4-\x*\x))}
\newcommand\Sigmaplot{(\x*\x+1)/2}
\newcommand\functalpha[2]{\Sh[#1](#2)}
\newcommand\falpha[1][\alpha]
	{\functalpha{#1}{f}}
\newcommand\falphapsprim[3][\myepsilon]{#2^{#3 #1}}
\newcommand\falphaps[2][\myepsilon]{f^{#2 #1}}
\newcommand\applatitk[3][+]{\rho^{#1,(#2)}_{#3}}
\newcommand\applatit[2][+]{\rho^{#1}_{#2}}
\newcommand\tapplatit[2][+]{\widetilde{\rho^{#1}_{#2}}}
\newcommand\tapplat[1]{\widetilde{\rho^{#1}_\mu}}
\newcommand\sizecorner{.15}
\newcommand\myclass[1][2]{\mathfrak{C}_{#1}}
\newcommand\partclass{\mathfrak{Y}}
\newcommand\mydeltaprim{\delta}
\newcommand\mydelta[1][x]{\mydeltaprim_{#1}}
\newcommand\mydeltafull[2]{\mydeltaprim_{#1}^{#2}}
\newcommand\Falpha[2][f]{#1_{[#2]}}
\newcommand\xalphap[1][\alpha]{x_{#1}^+}
\newcommand\xalphapf[2][f]{\xalphap[\alpha](#1)}
\newcommand\xalpham[1][\alpha]{x_{#1}^-}
\newcommand\xalphamf[2][f]{\xalpham[\alpha](#1)}
\newcommand\xalphapp{x_\alpha'^+}
\newcommand\xalphamp{x_\alpha'^-}
\newcommand\coinc{u_\alpha}
\newcommand\compfull[2]{\phi_{#1}^{#2}}
\newcommand\comp[1][\alpha]{\phi_{#1}}
\newcommand\invdelta\tau
\newcommand\su[1][]{a_{#1}} 
\newcommand\sd[1][]{b_{#1}}
\newcommand\biss{\mathcal{L}}
\newcommand\bissp{\biss'}

\newcommand\mywidth{.8\textwidth}
\newcommand\colshake{orange!70}
\definecolor{greensage}{rgb}{0, 0.5019607843137255, 0}

\author{Salim \textsc{Rostam}\thanks{Institut Denis Poisson, CNRS UMR 7013, Université de Tours, 37200 Tours, France}}
\title{Limit shape for regularisation of large partitions under the Plancherel measure}
\date{}

\begin{document}

\maketitle

\abstract{A celebrated result of Kerov--Vershik and Logan--Shepp gives an asymptotic shape for large partitions under the Plancherel measure. We prove that when we consider $e$-regularisations of such partitions we still have a limit shape, which is given by a shaking of the Kerov-Vershik-Logan-Shepp curve. We deduce an explicit form for the first asymptotics of the length of the first row and the first column for the $e$-regularisation.}

\section{Introduction}

Partitions of a given integer $n$ are the different ways to decompose $n$ as an unordered sum of positive integers. In other words, a \textbf{partition} of $n \in \mathbb{Z}_{\geq 0}$ is a non-increasing sequence $\lambda = (\lambda_1 \geq \dots )$ of non-negative integers with sum $n$. This mathematical object appears for instance in the study of the symmetric group $\mathfrak{S}_n$ of permutations of $\{1,\dots,n\}$, since partitions of $n$ index the conjugacy classes (via the cycle decomposition).

A \emph{representation}  of $\mathfrak{S}_n$ of dimension $N$ over a field $k$ is a group homomorphism $\rho : \mathfrak{S}_n \to \mathrm{GL}_N(k)$. We focus for the moment at the case where $k$ is  the field $\mathbb{C}$ of complex numbers, in which case we say that we have a \emph{complex representation}. As any integer decomposes into a product of primes, any complex representation decomposes into a sum of \emph{irreducible} complex representations. It turns out that the fact that the set $\partsn$ of partitions of $n$ index the conjugacy classes implies that $\partsn$ also index the set of irreducible complex representations. If we denote by $\rho_\lambda$ the irreducible complex representation associated with $\lambda \in \partsn$, a standard result of complex representation theory shows that:
\begin{equation}
\label{equation:n!=sumdim}
\#\mathfrak{S}_n = n! = \sum_{\lambda \in \partsn} (\dim \rho_\lambda)^2.
\end{equation}
A remarkable fact is that we are able to explicitly compute the numbers $\dim \rho_\lambda$ (namely with the famous \emph{hook length formula}). We refer for instance to Sagan~\cite{sagan} or James--Kerber~\cite{james-kerber} for more details on the complex representations of $\mathfrak{S}_n$.

\medskip
Now if we want to study the representations of $\mathfrak{S}_n$ over the field $k = \mathbb{F}_p$ with $p$ elements (with $p$ a prime number; we will say that we have a \emph{$p$-representation}), almost everything that is known for complex representations falls apart. A fundamental difference is that some representations may \emph{not} decompose into a sum of irreducible ones. Nevertheless, the study of irreducible representations is still interesting since any representation can always be decomposed into irreducible constituents via a \emph{composition} (or \emph{Jordan--Hölder}) \emph{series}.

A general theorem of Brauer says that the irreducible $p$-representations are in bijection with the $p$-regular conjugation classes, that is, with conjugation classes formed by elements whose order does not divide $p$. For the symmetric group, one can see that the set of $p$-regular conjugation classes is in one-to-one correspondence with the set of \textbf{$\boldsymbol{p}$-regular} partitions, that is, partitions with no $p$ (or more) consecutive equal parts. Note that the dual notion of \emph{$p$-restricted} partitions is also present in the literature, where the difference between two consecutive parts of a $p$-restricted partition it at most $p-1$.  If $\lambda$ is a $p$-regular partition, we will denote by $\rho^p_\lambda$ the associated irreducible $p$-representation. Contrary to the complex case, there is no formula expressing the dimension of $\rho^p_\lambda$ (yet). We refer for instance to James--Kerber~\cite{james-kerber} for more details on the $p$-representations of $\mathfrak{S}_n$.

A way to understand $p$-irreducible representations is to study the $p$-irreducible representations that appear in the decomposition series of an irreducible complex representation reduced modulo $p$. More explicitly, it can be shown that any irreducible complex representation can be realised over $\mathbb{Z}$, that is, we can assume that $\rho_\lambda : \mathfrak{S}_n \to \mathrm{GL}_N(\mathbb{Z})$. The reduction modulo $p$, that we denote by $\overline{\rho}_\lambda : \mathfrak{S}_n \to \mathrm{GL}_N(\mathbb{F}_p)$, is then simply the reduction modulo $p$ of the matrix entries. Now if we return to the problem of determining which irreducible $p$-representations appear in a decomposition series of $\overline{\rho}_\lambda$,  James~\cite{james} gave an explicit combinatorial construction of a $p$-regular partition $\reg(\lambda)$ such that $\rho^p_\lambda$ appears in a decomposition series of $\overline{\rho}_\lambda$. This partition $\reg(\lambda)$ is the \textbf{$\boldsymbol{p}$-regularisation} of $\lambda$. The $p$-regularisation operation has in fact a meaning also when $p$ is not prime, in the context of the \emph{Iwahori--Hecke algebra} of $\mathfrak{S}_n$ (see, for instance, Mathas~\cite{mathas}). In particular, for an integer $e \geq 2$ we will also use the terms $e$-regular and $e$-regularisation. Note that the $e$-regularisation map was recently generalised by Millan Berdasco~\cite{millan} to an $(e,i)$-regularisation map on partitions: it would be interesting to see whether the results we present in this paper generalise to this  setting.

\bigskip
We now go back to~\eqref{equation:n!=sumdim}. This equation shows that $\Pl(\lambda) \coloneqq \frac{(\dim \rho_\lambda)^2}{n!}$ is a probability measure on the set $\partsn$ of partitions of $n$, called the \emph{Plancherel measure}. Via some calculations involving the \emph{hook integral} (defined in the spirit of the hook length formula), Kerov--Vershik~\cite{kerov-vershik:asymptotics} and Logan--Shepp~\cite{logan-shepp:variational}  independently proved that there is a curve~$\Omega$ so that the upper rim $\tomega_\lambda$ of the Young diagram (in the Russian convention) of a large partition $\lambda$ converges uniformly in probability to $\Omega$. An illustration of this convergence is given in Figure~\ref{figure:lgn} (note\footnote{All the computations were made using \textsc{SageMath}~\cite{sage}.}). This limit shape theorem allowed to determine the first asymptotics of the length of the first row (and first column) of a Young diagram, taken under the Plancherel measure. Note that, via the Robinson--Schensted correspondence (which provides a bijective proof of~\eqref{equation:n!=sumdim}), this provides a solution to the \emph{Ulam problem} on the length of a longest increasing subsequence of a word in $\mathfrak{S}_n$ chosen uniformly. We refer for instance to Romik~\cite{romik} for more details on the Plancherel measure and related asymptotics.

In the previous context of regularisation of partitions, the following question is thus natural: for an integer $e \geq 2$, what can be said about the partition $\reg(\lambda)$ when $\lambda$ is a large partition taken under the Plancherel measure? The aim of this paper is to give a first answer to this question.

\begin{figure}[ht]
\begin{center}
\includegraphics[width=\mywidth]{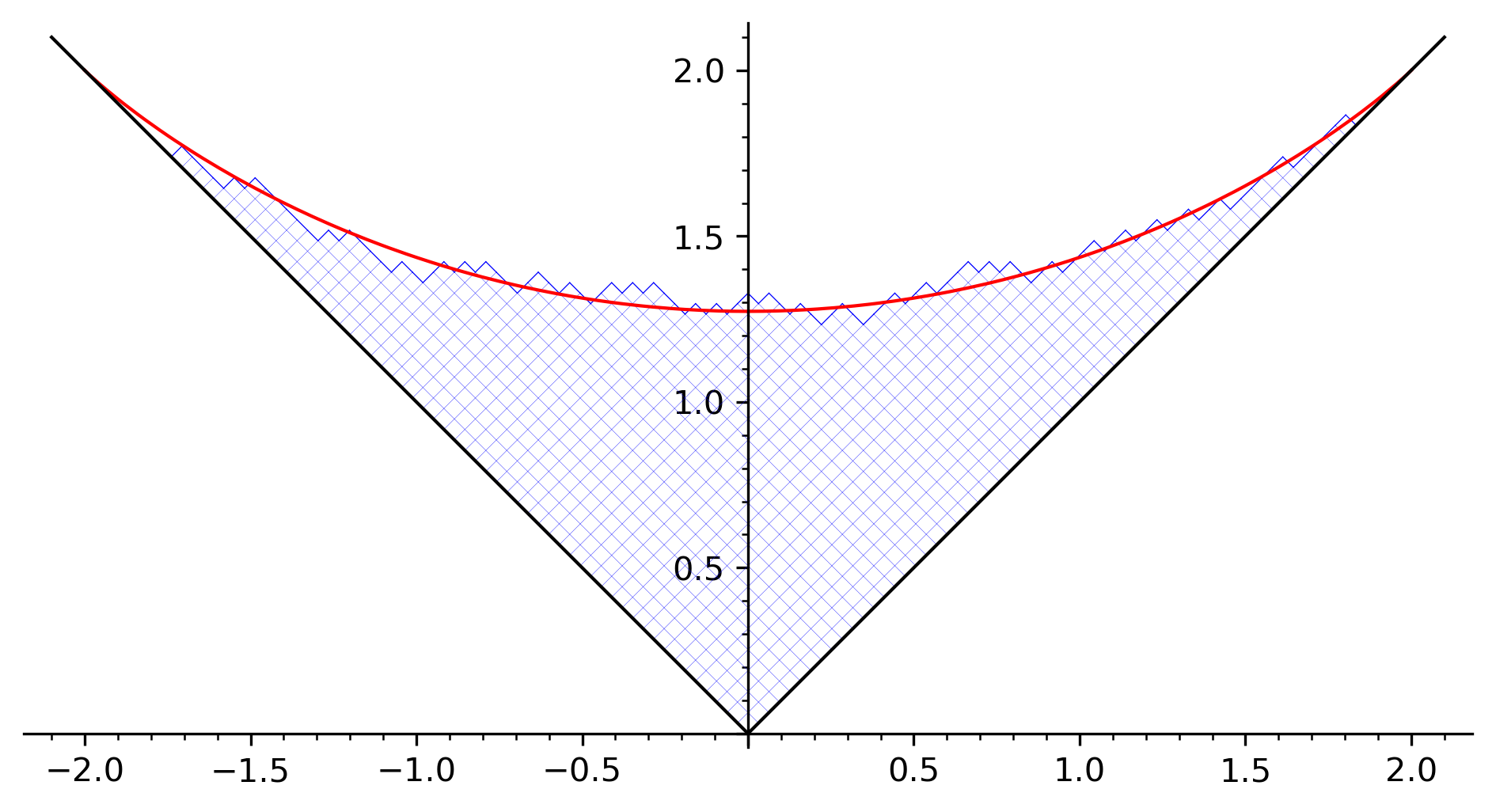}
\end{center}
\caption{The limit shape theorem for large partitions under the Plancherel measure (in red the shape $\Omega$)}
\label{figure:lgn}
\end{figure}

\bigskip
Steiner symmetrisation is a useful tool to study (namely) the isoperimetric problem. More precisely, if $K$ is a compact convex subset of $\R^2$ then its \emph{symmetrisation} with respect to the line $L$ is the set $S_L(K)$ that we obtain after sliding the different slices of $K$ (with respect to the orthogonal direction of $L$) until the midpoint of the slice is on $L$. An example of symmetrisation with respect to a vertical line in given in Figure~\ref{figure:symm_sh}. The link with the isoperimetric problem is that $S_L(K)$ has the same area as $K$ but a smaller perimeter. Note the following result, known as the \emph{sphericity theorem} of Gross: from any compact convex subset of $\R^2$ one can obtain the unit disk after a  succession of (possibly infinite number of) symmetrisations. Moreover, this concept of symmetrisation can also be extended to maps, with for example Schwarz \emph{symmetric rearrangement}, which is a powerful tool to study functional inequalities. We refer for instance to Gruber~\cite{gruber} or Krantz--Parks~\cite{krantz} for more details on the Steiner symmetrisation.

 A variation of Steiner symmetrisation is the notion of \textbf{shaking}. This notion was first introduced by Blaschke~\cite{blaschke} to solve Sylvester's ``four points problem''. The difference with Steiner symmetrisation is that we slide the slices until we \emph{meet} $L$. An example is given in Figure~\ref{figure:symm_sh}. Shaking and Steiner symmetrisation share many properties, for instance, Gross theorem holds for the shaking operation with ``unit disk'' replaced by ``simplex''. We refer to~\cite{shaking_compact} for more details on the shaking operation. We will in fact use the shaking operation in a context where the slices are not orthogonal to the line $L$; such an operation is for instance used in~\cite{freyer}.
 
\begin{figure}
\begin{center}
\begin{tikzpicture}[decoration={markings,mark=between positions .1 and .9 step .2  with {\arrow{>}}}]
\draw (4,3) --  (4,-2);
\draw (4,-2) -- (4,-4) node[below]{$L$};

\fill[blue] (0,2) -- (1,1) -- (0,0) -- cycle;
\draw[fill, blue!30] (4,2) -- (4.5,1) -- (4,0) -- (3.5,1) -- cycle;

\fill[olive] (0,-1) -- (1,-2) -- (0,-3) -- cycle;
\draw[fill, olive!30] (4,-1) -- (4,-3) -- (3,-2) -- cycle;

\draw[red, very thick] (0,1.5) -- (.5,1.5);
\draw[red,dashed,postaction={decorate}] (.5,1.5) -- (3.75,1.5) node[midway,above]{symmetrisation};
\draw[red,very thick] (3.75,1.5) -- (4.24,1.5);

\draw[red,very thick] (0,1) -- (1,1);
\draw[red, dashed,,postaction={decorate}] (1,1) -- (3.5,1);
\draw[red,very thick] (3.5,1) -- (4.5,1);

\draw[orange, very thick] (0,-2) -- (1,-2);
\draw[orange,dashed,,postaction={decorate}] (1,-2) -- (3,-2);
\draw[orange, very thick] (3,-2) -- (4,-2);

\draw[orange,very thick] (0,-2.5) -- (.5,-2.5);
\draw[orange,dashed,postaction={decorate}] (.5,-2.5) -- (3.5,-2.5) node[midway,below]{shaking};
\draw[orange,very thick] (3.5,-2.5) -- (4,-2.5);
\end{tikzpicture}
\end{center}
\caption{Examples of symmetrisation and shaking with respect to a line $L$}
\label{figure:symm_sh}
\end{figure}
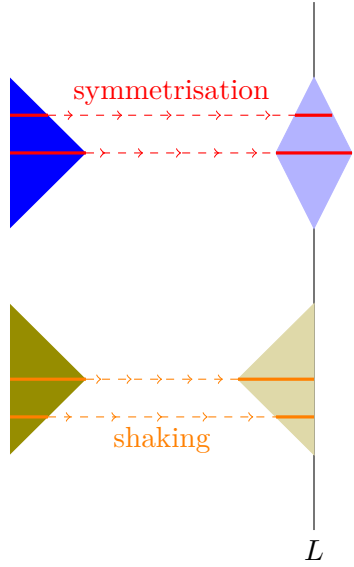

\bigskip

We can now state our main result (Theorems~\ref{theorem:sh_gr_f} and~\ref{theorem:limit_shape_regularisation}, together with~\eqref{equation:suppfalpha_Omega} and Corollary~\ref{corollary:Omegae_convex}).

\begin{theoremintro}
\label{theoremintro:main}
Let $e \geq 2$. Under the Plancherel measure $\Pl$, the upper rim $\tomega_{\reg(\lambda)}$ of the Young diagram (in the Russian convention) of $\reg(\lambda)$ converges uniformly in probability to the convex shape $\Omegae$ as $n \to +\infty$, in other words, for any $\myepsilon > 0$ we have
\[
\Pl\left(\sup_{\R} \,\bigl\lvert\tomega_{\reg(\lambda)}-\Omegae\bigr\rvert > \myepsilon\right) \xrightarrow{n\to+\infty} 0.
\]
The shape $\Omegae$ is obtained by shaking the part $\gr(\Omega)$ of the graph  of $\Omega$ that is above the graph of the absolute value, with respect to the line of equation $y = -x$ and angle $\alpha \coloneqq 1-2e^{-1}$, and is given by:
\begin{align*}
\Omegae(x) &= \Omega(x), &&\text{for all } x \leq x_\alpha^-,
\\
\Omegae(x+\delta_x) &= \Omega(x) + \alpha \delta_x, &&\text{for all } x \in (x_\alpha^-,x_\alpha^+),
\\
\Omegae(x) &= x, &&\text{for all } x \geq \frac{2e}{\pi}\sin\frac{\pi}{e},
\end{align*}
where:
\begin{itemize}
\item $x_\alpha^- \in [-a,x_\alpha^+)$ (is the unique point that) satisfies $\Falpha[\Omega]{\alpha}(x_\alpha^-) = (1-\alpha)a$,
\item $x_\alpha^+ = \Omega'^{-1}(\alpha) \in [0,a)$,
\item $\delta_x = (1-\alpha)^{-1} \Falpha[\Omega]{\alpha}(x) - \Falpha[\Omega]{\alpha}^{-1}\bigl(\Falpha[\Omega]{\alpha}(x)\bigr)$, where $\Falpha[\Omega]{\alpha} : x \mapsto \Omega(x) - \alpha x$ and $\Falpha[\Omega]{\alpha}^{-1}$ denotes its inverse on $[x_\alpha^+,+\infty)$.
\end{itemize}
\end{theoremintro}

 We illustrate the convergence for $e = 2$ (respectively, $e = 3$) in Figure~\ref{figure:limit_shape_2_diagram} (resp. Figure~\ref{figure:limit_shape_3_diagram}) and compare the limit shapes $\Omegae[2]$ and $\Omega$ (resp. $\Omegae[3]$).  In the particular case $e = 2$, the statement of Theorem~\ref{theoremintro:main} becomes more explicit (Corollary~\ref{corollary:limit_shape_alpha0}). In particular, the shape $\Omegae[2]$ is obtained via the horizontal shaking of $\gr(\Omega)$ (with respect to the line of equation $y = -x$), fact that we highlight in Figure~\ref{figure:Omegae_alpha0}.

\begin{corollaryintro}
\label{corollaryintro}
 Under the Plancherel measure $\Pl$, the upper rim $\tomega_{\reg[2](\lambda)}$ of the Young diagram (in the Russian convention) of $\reg[2](\lambda)$ converges uniformly in probability to the convex shape $\Omegae[2]$ given by:
\begin{align*}
\Omegae[2](x) &= \Omega(x),  &&\text{for all } x \leq -2,
\\
\Omegae[2]\bigl(2x+ \Omega(x)\bigr) &= \Omega(x), &&\text{for all }  x \in (-2, 0),
\\
\Omegae[2](x) &= x, &&\text{for all } x \geq \frac{4}{\pi}.
\end{align*}
\end{corollaryintro}

\begin{figure}[ht]
\begin{center}
\includegraphics[width=\mywidth]{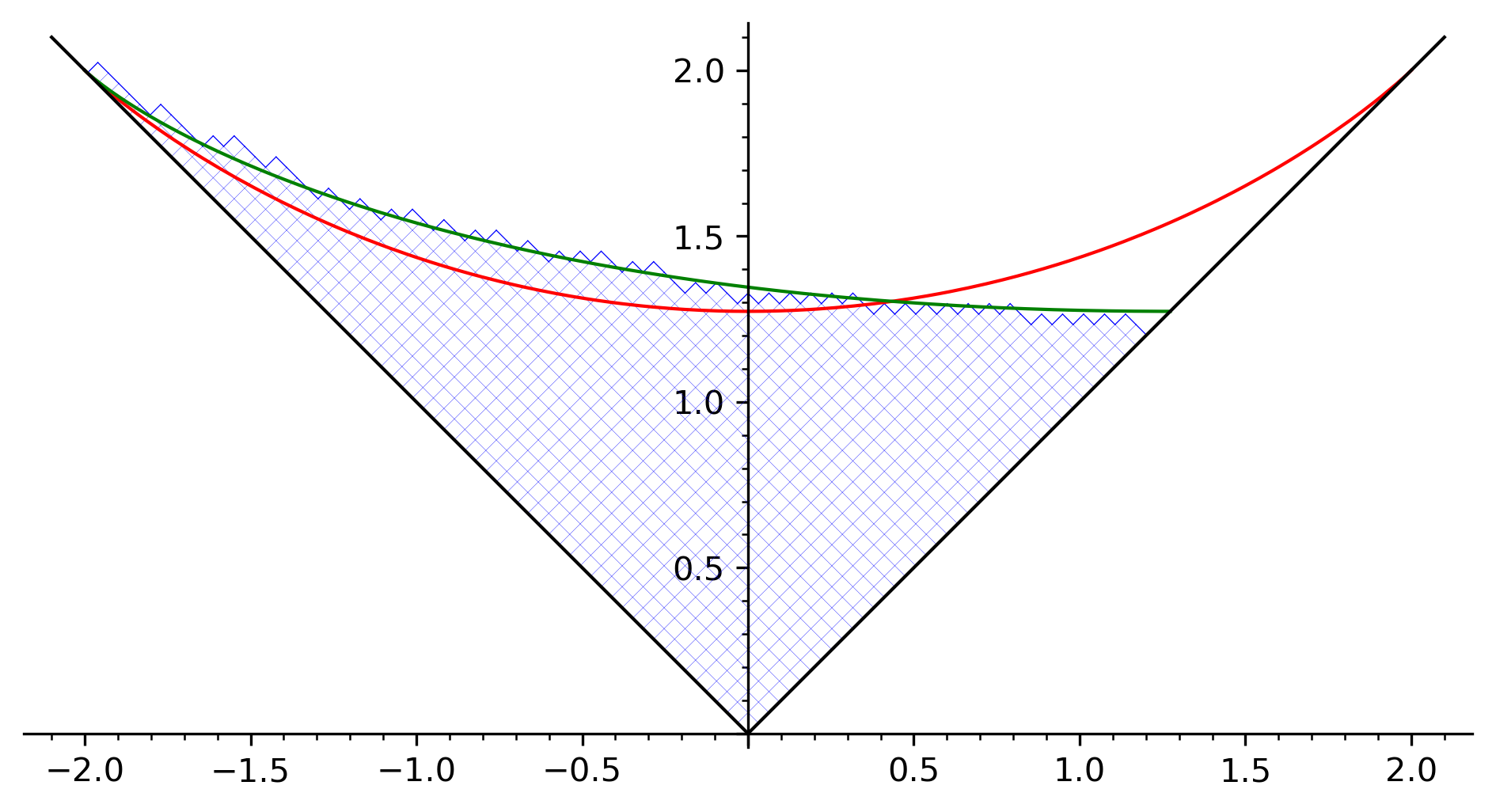}
\end{center}
\caption{Example of $2$-regularisation of a large partition taken under $\Pl[1000]$, with in green the limit shape $\Omegae[2]$ of Corollary~\ref{corollaryintro} and in red the limit shape $\Omega$ of Kerov-Vershik-Logan-Shepp}
\label{figure:limit_shape_2_diagram}
\end{figure}

\begin{figure}[ht]
\begin{center}
\includegraphics[width=\mywidth]{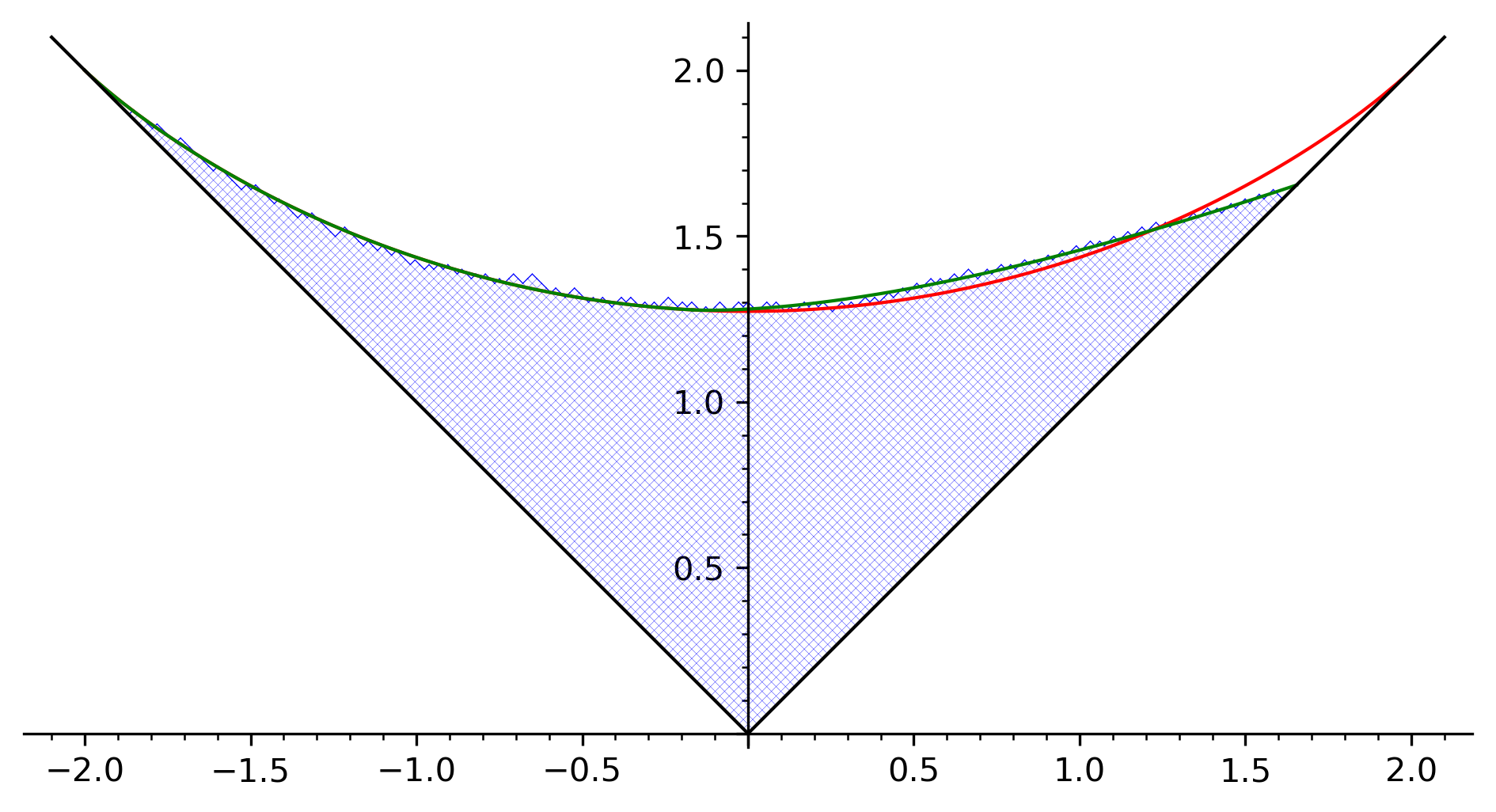}
\end{center}
\caption{Example of $3$-regularisation of a large partition taken under $\Pl[5000]$, with in green the limit shape $\Omegae[3]$ of Theorem~\ref{theoremintro:main} and in red the limit shape $\Omega$ of Kerov-Vershik-Logan-Shepp}
\label{figure:limit_shape_3_diagram}
\end{figure}

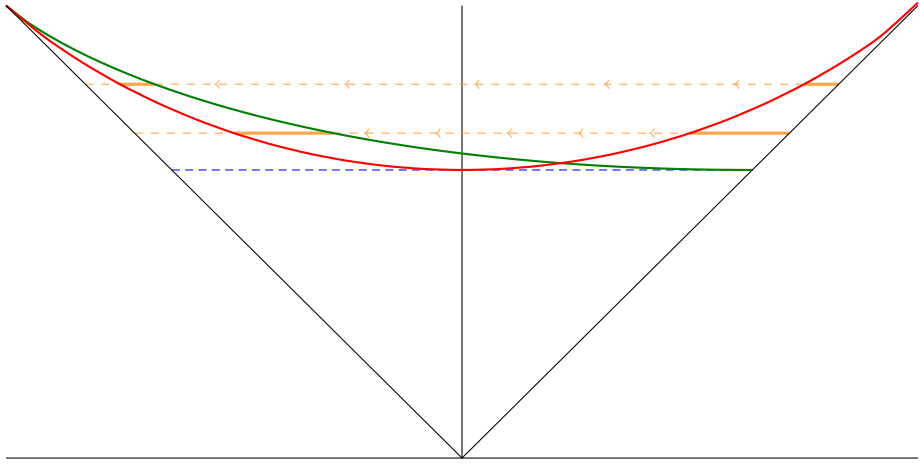
\begin{figure}
\begin{center}
\begin{tikzpicture}[scale=3,decoration={markings,mark=between positions .1 and .9 step .2  with {\arrow{<}}}]

\draw[blue,densely dashed] (-1.2732,1.2732) -- (1.2732,1.2732);

\draw (0,2) -- (0,0);
\draw (-2,0) -- (2,0);

\draw[\colshake,dashed,postaction={decorate}]  (-.564,1.43599) -- (1,1.43599);
\draw[\colshake,dashed] (-1.43599,1.43599) -- (-1,1.43599);
\draw[very thick,\colshake] (-1,1.43599) -- (-.564,1.43599);
\draw[very thick,\colshake] (1,1.43599) -- (1.43599,1.43599);

\draw[dashed,\colshake,postaction={decorate}]  (-1.34799,1.652) -- (1.5,1.652);
\draw[dashed,\colshake] (-1.652,1.652) -- (-1.5,1.652);
\draw[very thick,\colshake] (-1.5,1.652) -- (-1.34799,1.652);
\draw[very thick,\colshake] (1.5,1.652) -- (1.652,1.652);

\draw[domain=-2:0, smooth,greensage,thick] plot ({2*\x+\Omegaplot},{\Omegaplot});
\draw[domain=-2:2,smooth,red,thick] plot (\x,{\Omegaplot});

\draw (-2,2) -- (0,0) -- (2,2);
\end{tikzpicture}
\end{center}
\caption{The maps $\Omegae[2]$ (in green) and $\Omega$ (in red). In blue the tangent for $\Omega$ at $0$. The left orange segments are what is ``added'' to the red curve, and comes from the right orange segments by shaking $\gr(\Omega)$.}
\label{figure:Omegae_alpha0}
\end{figure}

We are also able to prove that the convergence of Theorem~\ref{theoremintro:main} holds for the supports (Theorem~\ref{theorem:cv_supports}). Despite the fact that the limit shape is not fully explicit, we are able to deduce the asymptotic behaviour of the length of the first line and of the first column of $\reg(\lambda)$ (Corollary~\ref{corollary:length_first_column}).

\begin{theoremintro}
Let $e \geq 2$. Under the Plancherel measure $\Pl$:
\begin{enumerate}
\item  the rescaled size $\frac{1}{\sqrt n}\reg(\lambda)_1$ of the first row of $\reg(\lambda)$ converges as $n \to +\infty$ in probability to $2$;
\item the rescaled size $\frac{1}{\sqrt n}\reg(\lambda)'_1$ of the first column of $\reg(\lambda)$ converges as $n \to +\infty$ in probability to $\frac{2e}{\pi}\sin\frac{\pi}{e}$.
\end{enumerate}
\end{theoremintro}

We now give the outline of the paper. 
In Section~\ref{section:background} we recall the definition of the $e$-regularisation map on the set of partitions of $n$ and we state the limit shape result of Kerov--Vershik and Logan--Shepp for large partitions taken under the Plancherel measure. We also recall the (non-orthogonal) shaking operation on compact subsets of $\R^2$. Namely, we explain in Proposition~\ref{proposition:reg_direction} why the direction $\alpha = 1-2e^{-1}$ will be important and we recall in Proposition~\ref{proposition:shaking_subset} that the shaking operation preserves the inclusions.

In Section~\ref{section:shaking_limit_shape} we introduce a shaking operations $f \mapsto \falpha$ on a certain class of functions. In~\textsection\ref{subsection:about_chords} we establish some preliminary results to be able to give Definition~\ref{definition:falpha}, which define this shaking operation $\Sh$ on functions. Then in~\textsection\ref{subsection:compatibility_shifts} and~\textsection\ref{subsection:relative_position} we study some properties of the shaked functions $\falpha$.

In Section~\ref{section:shaking_graphs}  we give our first main result, Theorem~\ref{theorem:sh_gr_f}, proving that the previous two shaking operations are indeed related: the graph of the shaked function $\falpha$ is obtained by shaking the graph of the original function $f$. Still using results from shaking theory, we prove in~\textsection\ref{subsection:convexity} that the function $\falpha$ is convex.

In Section~\ref{section:shaking_partitions} we introduce two approximations $\applatit[\pm]{\lambda}$ of the upper rim $\omega_\lambda$ of the Young diagram (in the Russian convention) of a partition $\lambda$, in order to study the shaking of~$\omega_\lambda$. More precisely, in~\textsection\ref{subsection:outer_regularisation} we define the \emph{outer flattening} $\applatit{\lambda}$ of an $e$-regular partition $\lambda$ (Definition~\ref{definition:applatit+}) and in Proposition~\ref{proposition:applatit+} we prove that $\applatit{\lambda}$ is  close to $\omega_\lambda$ and  that its graph is stable under the shaking operation. A similar construction is made in~\textsection\ref{subsection:inner_regularisation} for the \emph{inner flattening} $\applatit[-]{\lambda}$.

Finally, Section~\ref{section:limit_shape_regularisation} is devoted to the proof of the second main result, Theorem~\ref{theorem:limit_shape_regularisation}, which states that the (rescaled) upper rim $\tomega_{\reg(\lambda)}$ of the Young diagram (in the Russian convention) of the $e$-regularisation of a partition $\lambda$ taken under the Plancherel measure converges uniformly in probability to the shape $\Omegae \coloneqq \functalpha{\alpha}{\Omega}$, where $\alpha = 1-2e^{-1}$. The idea of the proof is to sandwich $\tomega_{\reg(\lambda)}$ between a flattening $\applatit[\pm]{\lambda}$ and a rescaling of the shape $\Omegae$.
We also prove the convergence of the support of $\tomega_\lambda - |\cdot|$ (Theorem~\ref{theorem:cv_supports}), and we deduce the asymptotic length of the first row and the first column of $\reg(\lambda)$ (Corollary~\ref{corollary:length_first_column}).

\paragraph{Acknowledgements} The author would like to thank Vincent Beck, François Bolley and Andrew Elvey Price for useful discussions. The author also thanks the Centre Henri Lebesgue ANR-11-LABX-0020-0. This research was funded, in whole or in part, by the Agence Nationale de la Recherche funding ANR CORTIPOM 21-CE40-001. 
A CC-BY public copyright license has been applied by the
author to the present document and will be applied to all subsequent
versions up to the Author Accepted Manuscript arising from this
submission, in accordance with the grant's open access conditions.

\section{Background}
\label{section:background}

We recall in~\textsection\ref{subsection:regularisation} the combinatorial notion of $e$-regularisation of a partition. In~\textsection\ref{subsection:limit_shape} we recall the limit shape result of Kerov--Vershik and Logan--Shepp for large partitions under the Plancherel measure. Finally, in~\textsection\ref{subsection:shakings}  we present the notion of shaking for compact sets in $\R^2$.

\subsection{Regularisation}
\label{subsection:regularisation}

A \emph{partition} is a finite non-increasing sequence $\lambda = (\lambda_1 \geq \dots \geq \lambda_h > 0)$ of positive integers. If $|\lambda| \coloneqq \sum_{i = 1}^h \lambda_i = n$ then we say that $\lambda$ is a partition of $n$. We denote by $\partsn$ the set of partitions of $n$.
The \emph{Young diagram} of a partition $\lambda = (\lambda_1 \geq \dots \geq \lambda_h > 0)$ is the subset of $\mathbb{Z}_{\geq 1}^2$ given by:
\[
\Y(\lambda) = \bigl\{(a,b) \in \nodes : 1 \leq a \leq h \text{ and } 1 \leq b \leq \lambda_a\bigr\}.
\]
Note that we will consider that the $a$-coordinate in the Young diagrams goes downwards. A \emph{node} is an element of $\nodes$.

\begin{example}
The Young diagram of the partition $\lambda = (4,4,1)$ is \ydiagram{4,4,1}*[*(cyan)]{1+1}\,. The node in blue has coordinates $(1,2)$.
\end{example}

\begin{definition}
\label{definition:regular_partitions}
Let $e \geq 2$. A partition $\lambda = (\lambda_1 \geq \dots\geq\lambda_h>0)$ is \emph{$e$-regular} if no parts repeat $e$ times or more, that is,  if $\lambda_i > \lambda_{i+e-1}$ for any $i \in \{1,\dots, h-e+1\}$.
\end{definition}

As we mentioned in the introduction, the notion of $e$-regular partitions appears for instance in the context of the representation theory of the \emph{Iwahori--Hecke algebra} $\mathcal{H}_q(\mathfrak{S}_n)$ of $\mathfrak{S}_n$ (a certain deformation of $\mathbb{C}\mathfrak{S}_n$), with $q \in \mathbb{C}^\times$ having order $e$, see, for instance,~\cite{mathas} (note that in~\cite{mathas} the dual notion of \emph{$e$-restricted} partition  is used). The next two definitions are due to James~\cite{james}.

\begin{definition}
Let $e \geq 2$.
\begin{itemize}
\item  The \emph{$e$-ladder number} (or simply \emph{ladder number}) of a node $\gamma = (a,b)\in \nodes$ is:
\[
\ladnode{\gamma} \coloneqq a + (e-1)(b-1) \in \mathbb{Z}_{\geq 1}.\]

\item Let $\ell \geq 1$. The \emph{$(e,\ell)$-th ladder} (or simply \emph{$\ell$-th ladder}) is the (finite) set of all nodes of $\nodes$ with $e$-ladder $\ell$. The \emph{$e$-ladder} of a node $\gamma$ is the $(e,\ladnode{\gamma})$-th ladder.
\end{itemize}
\end{definition}

\begin{example}
\label{example:ladder_numbers}
In the Young diagram of $\lambda = (4,4,3,3,3,3,3,1)$, in each node we write the corresponding $4$-ladder numbers:
\[
\ytableausetup{nosmalltableaux}
\ytableaushort{147{10},258{11},369,47{10},58{11},69{12},7{10}{13},8}\,.
\]
\end{example}

\begin{definition}
Let $e \geq 2$ and let $\lambda$ be a partition. The \emph{$e$-regularisation} of $\lambda$ is the partition $\reg(\lambda)$ that we obtain after moving each node of $\Y(\lambda)$ as high as possible in its $e$-ladder.
\end{definition}

Note that $\reg(\lambda)$ is an $e$-regular partition, and if $\lambda$ is $e$-regular then $\reg(\lambda) = \lambda$. As we have mentioned in the introduction, the $e$-regularisation map has a significance in terms of modular representation theory of the symmetric group (or its associated Iwahori--Hecke algebra).

\begin{example}
The $4$-regularisation of the partition of Example~\ref{example:ladder_numbers} is $(5,4,4,3,3,2,2,1)$. The $4$-ladders of the added (respectively, deleted) nodes are in green (resp. red).
\[
\ytableaushort{147{10}{\color{green}13},258{11},369{\color{green}12},47{10},58{11},69{\color{red}12},7{10}{\color{red}13},8}\,
\]
\end{example}

\subsection{Limit shape}
\label{subsection:limit_shape}

\paragraph{Russian convention}
Rotating the Young diagram of $\lambda = (\lambda_1\geq \dots \geq \lambda_h>0) \in \partsn$ by an angle of $\frac{3\pi}{4}$ and embedding it inside $\mathbb{R}\times\mathbb{R}_{\geq 0}$ so that the box $(1,1)$ has bottom vertex at $(0,0)$ and each box has area $2$ (\textit{i.e.} semi-diagonal length $1$) gives the \emph{Russian convention} for the Young diagram of $\lambda$. Note that the node $(a,b) \in \mathcal{Y}(\lambda)$ corresponds to the (square) box with top vertex $(a-b,a+b)$ in $\mathbb{R}\times \mathbb{R}_{\geq 0}$. 
 We denote by $\omega_\lambda : \mathbb{R} \to \mathbb{R}$ the upper rim of the resulting diagram, extending $\omega_\lambda$ by $\omega_\lambda(x) \coloneqq |x|$ outside the diagram.
Then $\omega_\lambda$ is a  continuous piecewise linear function such that:
\begin{itemize}
\item for each $k \in \mathbb Z$ we have
\begin{equation}
\label{equation:omega'}
\omega'_\lambda|_{(k,k+1)} = \pm 1,
\end{equation}
\item we have $\omega_\lambda(x) = |x|$ for $|x| \gg 0$ (more precisely, for $x \leq -\lambda_1$ or $x \geq h$),
\item we have $\omega_\lambda(x) \geq |x|$ for all $x \in \R$ and $\int_{\mathbb{R}} \bigl[\omega_\lambda(x) - |x|\bigr] dx = 2n$.
\end{itemize}
An illustration of the construction of $\omega_\lambda$ is given in Figure~\ref{figure:omegalambda}. (We warn the reader that in the literature the convention is sometimes reversed, that is, our $\omega_\lambda$ is sometimes reflected with respect to the axis $\{0\}\times\R$.)  We will use a particular rescaling $\tomega_\lambda : \mathbb R \to \mathbb R$ of $\omega_\lambda$, given by:
\begin{equation}
\label{equation:rescaling}
\tomega_\lambda(s) \coloneqq \frac{1}{\sqrt n} \omega_\lambda\bigl(s\sqrt n\bigr),
\end{equation}
for all $s \in \R$. Note that the area between the curves of $\tomega_\lambda$ and $|\cdot|$ is $2$.

\begin{figure}
\begin{center}
\begin{tikzpicture}[scale=.6]
\draw[-stealth] (-8,0) -- (8,0) node[below]{$x$};
\draw[-stealth] (0,-1) -- (0,9) node[right]{$y$};
\foreach \i in {-7,-6,-5,-4,-3,-2,-1,1,2,3,4,5,6,7}
	{\draw (\i,-.1) --++ (0,.2);}
\foreach \j in {1,2,...,8}
	\draw (-.1,\j) --++ (.2,0);

\draw[very thick] (-7,7) -- (-4,4) -- (-2,6) -- (0,4) -- (1,5) -- (2,4) -- (3,5) -- (4,4) -- (7,7);
\draw[dotted, very thick] (-7.7,7.7) -- (-7,7);
\draw[dotted, very thick] (7,7) -- (7.7,7.7) node[below right]{$y=\omega_\lambda(x)$};

\draw[dashed] (-4,4) -- (0,0) -- (4,4);
\draw[dashed] (-3,5) -- (0,2) -- (2,4) --++ (1,-1);
\draw[dashed] (-3,3) --++ (2,2);
\draw[dashed] (-2,2) --++ (2,2) --++ (2,-2);
\draw[dashed] (-1,1) --++ (1,1) --++ (1,-1);
\end{tikzpicture}
\end{center}
\caption{Russian convention for the Young diagram of $\lambda = (4,4,2,1)$}
\label{figure:omegalambda}
\end{figure}
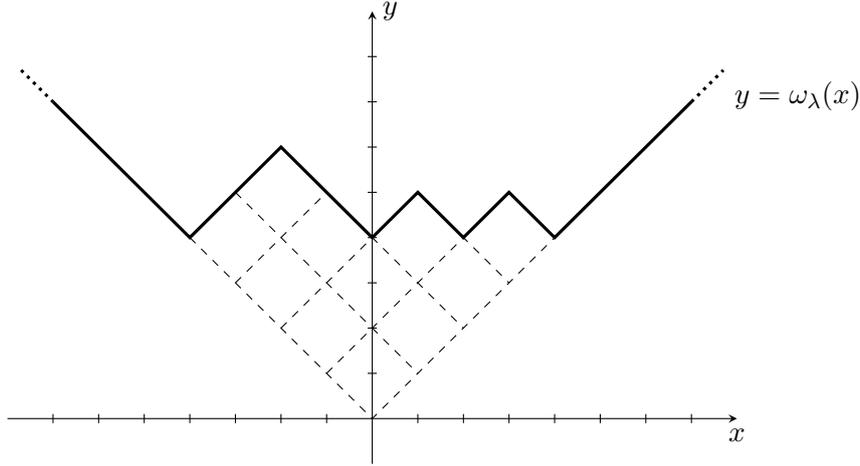

\paragraph{Plancherel measure} Let $\lambda \in \partsn$.
 A \emph{standard tableau of shape $\lambda$} is a bijection $\mathfrak{t} : \mathcal{Y}(\lambda) \to \{1,\dots,n\}$ such that $\mathfrak{t}$ increases along the rows and down the columns, in other words for $(a,b) \in \mathcal{Y}(\lambda)$ we have $\mathfrak{t}(a,b) < \mathfrak{t}(a+1,b)$ if $(a+1,b) \in \mathcal{Y} (\lambda)$ and $\mathfrak{t}(a,b) < \mathfrak{t}(a,b+1)$ if $(a,b+1) \in \mathcal{Y}(\lambda)$. We denote by $\mathrm{Std}(\lambda)$ the set of standard Young tableaux of shape $\lambda$. We have the following standard identity:
\[
n! = \sum_{\lambda \in \partsn} \#\mathrm{Std}(\lambda)^2.
\]

\begin{definition}
The \emph{Plancherel measure} on the set $\partsn$ of the partitions of $n$ is given by
\[
\Pl(\lambda) \coloneqq \frac{\#\mathrm{Std}(\lambda)^2}{n!},
\]
for all $\lambda \in \partsn$.
\end{definition}

The next result describes the Plancherel measure $\Pl$ for large $n$.

\begin{theorem}[\protect{\cite{logan-shepp:variational}, \cite{kerov-vershik:asymptotics}, \cite[Theorem 1.26]{romik}}]
\label{theorem:LLN_partitions}
Let $\Omega : \mathbb R \to \mathbb R$ be defined by
\[
\Omega(s) \coloneqq \begin{dcases}
\frac{2}{\pi}\left(s\arcsin\frac{s}{2}+\sqrt{4-s^2}\right),&\text{if } \lvert s \rvert \leq 2,
\\
\lvert s \rvert, &\text{otherwise}.
\end{dcases}
\]
Then, under the Plancherel measure $\Pl$, the function $\tomega_\lambda$ converges uniformly in probability to $\Omega$ as $n \to +\infty$. In other words, for any $\myepsilon > 0$ we have
\[
\Pl\left(\sup_{\mathbb R}\, \bigl\lvert\tomega_\lambda-\Omega\bigr\rvert > \myepsilon\right) \xrightarrow{n\to+\infty} 0.
\]
Moreover, we also have convergence of the supports, that is:
\begin{align*}
\inf\bigl\{s \in \mathbb R : \tomega_\lambda(s) \neq |s|\bigr\} &\longrightarrow -2,
\\
\intertext{and:}
\sup\bigl\{s \in \mathbb R : \tomega_\lambda(s) \neq |s|\bigr\} &\longrightarrow 2,
\end{align*}
in probability under $\Pl$ as $n \to +\infty$.
\end{theorem}

An illustration of Theorem~\ref{theorem:LLN_partitions} for $n = 1000$ is given in Figure~\ref{figure:lgn}.
Note that the limit shape $\Omega$ has a much simpler form after derivation.

\begin{lemma}
\label{lemma:derivative}
The map $\Omega$ is an antiderivative on $\R$ of:
\[
s \mapsto \begin{dcases}
\frac{2}{\pi}\arcsin\frac{s}{2}, &\text{if } |s| \leq 2,
\\
\mathrm{sgn}(s), &\text{if } |s| > 2.
\end{dcases}
\]
\end{lemma}

%
%

\subsection{Shakings}
\label{subsection:shakings}

The aim of this paper is to put together the notions of~\textsection\ref{subsection:regularisation} and~\textsection\ref{subsection:limit_shape}. As a first step, we show in Figure~\ref{figure:ladders_russian} what do the $4$-ladders of Example~\ref{example:ladder_numbers} look like in the Russian convention. 

\begin{proposition}
\label{proposition:reg_direction}
In the Russian convention, the $e$-regularisation makes the nodes going as left as possible in the direction of $y = \bigl(1-2e^{-1}\bigr)x$.
\end{proposition}

\begin{proof}
In the Young diagram we know that the node $(e,1) \in \nodes$ goes to the node $(1,2)$ during the $e$-regularisation. In the Russian convention it means that the box with top vertex $(e-1,e+1)$ goes to the box with top vertex $(-1,3)$. The corresponding slope is thus $\frac{e+1-3}{e-1-(-1)} = \frac{e-2}{e} = 1-2e^{-1}$ as announced. Note that the slope do not change when both axes are rescaled by a same constant and that all boxes follow the same direction.
\end{proof}

\begin{figure}[ht]
\begin{center}
\begin{tikzpicture}[scale=.6]
\draw[-stealth] (-5,0) -- (9,0) node[below]{$x$};
\draw[-stealth] (0,-1) -- (0,11) node[right]{$y$};
\foreach \i in {-4,-3,...,-1,1,2,...,8}
	{\draw (\i,-.1) --++ (0,.2);}
\foreach \j in {1,2,...,10}
	\draw (-.1,\j) --++ (.2,0);
	
\draw[dashed] (-4,4) --++ (1,1) --++ (4,-4) --++ (-1,-1) -- cycle;
\draw[dashed] (-3,5) --++ (1,1) --++ (4,-4) --++ (-1,-1);
\foreach \i in {-1,0,...,3}
	\draw[dashed] (\i,\i+6) --++ (1,1) --++ (3,-3) --++ (-1,-1);
\draw[dashed] (6,8) --++ (1,1) --++ (1,-1) --++ (-1,-1);

\draw[dashed] (-3,3) --++ (2,2);
\foreach \i in {-2,-1}
	\draw[dashed] (\i,-\i) --++ (7,7);
	
\foreach \i in {1,2,...,8}
	\draw (\i-1,\i) node[fill=white]{$\i$};

\foreach \i in {4,5,...,10}
	\draw (\i-5,\i-2) node[fill=white]{$\i$};
	
\foreach \i in {7,8,...,13}
	\draw (\i-9,\i-4) node[fill=white]{$\i$};
	
\foreach \i in {10,11}
	\draw (\i-13,\i-6) node[fill=white]{$\i$};
\end{tikzpicture}
\caption{The $4$-ladders in the Russian convention for $\lambda = (4,4,3,3,3,3,3,1)$}
\label{figure:ladders_russian}
\end{center}
\end{figure}
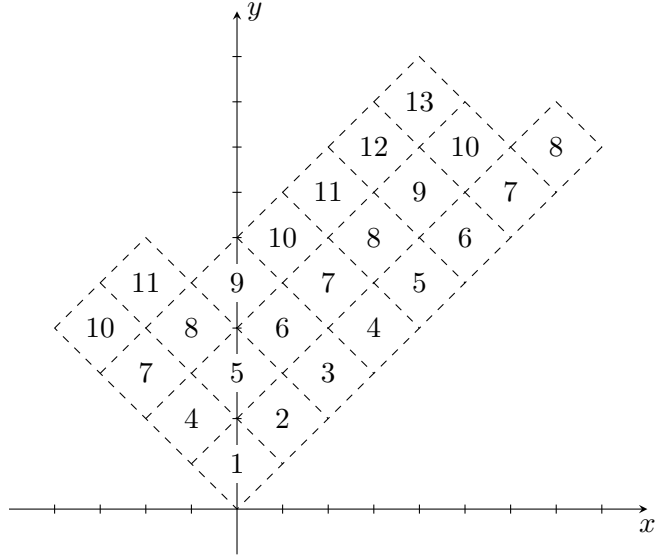

We now define the shaking operation that will be of interest for us. As we mentioned in the introduction, this is a variation of Steiner symmetrisation first introduced by Blaschke~\cite{blaschke}. Let $|\cdot|$ be the one-dimensional Lebesgue measure in $\R$ and let $\biss$ be the line of equation $y = -x$.

\begin{definition}
\label{definition:shaking}
Let $\alpha \in \mathbb{R}_{\geq 0}$ and let $v_\alpha$ be the unit vector positively collinear to $(1,\alpha)^\top$. If $K$ is a compact subset of $\mathbb{R}^2$, its \emph{shaking} with direction $\alpha$ (against $y = -x$) is the compact subset $\Sh(K)$ of $\mathbb{R}^2$ that we obtain by sliding the set $K$ along the direction of the line $(L_\alpha) : y = \alpha x$ until we meet the line $\biss$, that is:
\[
\Sh(K) \coloneqq \bigsqcup_{x \in L} K_\alpha^x,
\]
where $K_\alpha^x$ is:
\begin{itemize}
\item empty if $K\cap(x + L_\alpha) = \emptyset$,
\item the segment with extreme points $x$ and $x+\bigl| K \cap (x+L_\alpha)\bigr|v_\alpha$ otherwise.
\end{itemize}
\end{definition}

In the usual definition of shaking the direction $L_\alpha$ is orthogonal to the  line $\biss$ against which we shake. This non-orthogonal variation is for instance considered in~\cite[\textsection 5]{freyer}.

\begin{example}
Take $\alpha = 0$ and let $K$ be the unit square with bottom left corner at $(1,1)$. Then, as shown on Figure~\ref{figure:Sh_example}, the set $\Sh(K)$ is the parallelogram with vertices $(-2,2),(-1,2),(0,1)$ and $(-1,1)$.
\end{example}

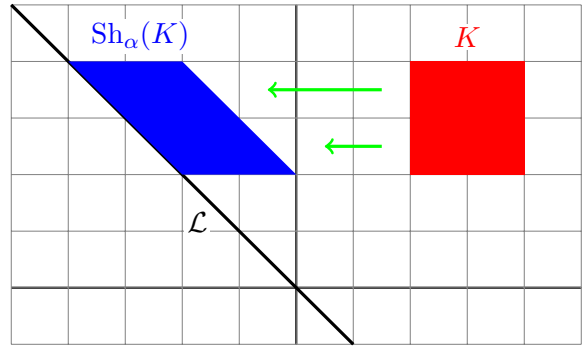
\begin{figure}
\begin{center}
\begin{tikzpicture}[scale=1.5]
\draw[thick] (0,-.5) -- (0,2.5);
\draw[thick] (-2.5,0) -- (2.5,0);
\draw[gray, very thin] (-2.5,-.5) grid[step=.5] (2.5,2.5);

\draw[very thick] (-2.5,2.5) -- (.5,-.5);

\fill[red] (1,1) rectangle (2,2);
\draw[red] (1.5,2.1) node[above,inner sep=1pt,fill=white]{$K$};

\fill[blue] (-2,2) -- node[above, midway,fill=white,inner sep=1pt,shift={(.2,.1)}] {$\Sh(K)$} (-1,2) -- (0,1) -- (-1,1) -- cycle;

\draw[very thick,green,->] (.75,1.75) -- (-.25,1.75);
\draw[very thick,green,->] (.75,1.25) -- (.25,1.25);

\draw (-.75,.7) node[below left,fill=white,inner sep = 1pt]{$\biss$};
\end{tikzpicture}
\end{center}
\caption{Example of the map $\Sh$ with $\alpha = 0$}
\label{figure:Sh_example}
\end{figure}

The following property is a mere reformulation of the definition. If $\alpha \in \R_{\geq 0}$ and $M \in \R^2$, we denote by:
\begin{itemize}
\item $\LM{M}$ the line with slope $\alpha$ containing $M$,
\item $d_\alpha(M, \biss)$ the  distance between $M$ and $\biss$ along the direction $\alpha$, that is, the distance between $M$ and its image under the  projection on $\biss$ along $\LM{M}$.
\end{itemize}

\begin{lemma}
\label{lemma:shaking_distance}
Let $\alpha \in \R_{\geq 0}$ and let $K \subseteq \R^2$ be compact. Let $M \in \R^2$ be at the right of the line $\biss$. Then:
\[
M \in \Sh(K) \iff d_\alpha(M,\biss) \leq \bigl|\LM{M} \cap K\bigr| \text{ and } \LM{M} \cap K \neq \emptyset.
\]
\end{lemma}

The next property is a standard property of Steiner symmetrisations and shakings. The proof is immediate from the definition.

\begin{proposition}
\label{proposition:shaking_subset}
Let $\alpha \in \R_{\geq 0}$. If $K \subseteq K'$ are compact subsets of $\R^2$ then $\Sh(K) \subseteq \Sh(K')$.
\end{proposition}

\begin{proof}
See~\cite[Lemma 1.1]{shaking_compact} and \cite[Lemma 5.3]{freyer}. Note that the proof is in fact the same as for the usual (orthogonal) shakings.
\end{proof}


Recalling Proposition~\ref{proposition:reg_direction}, we are interested in special cases of shakings.

\begin{definition}
\label{definition:sh}
For any $e \geq 2$, we denote by $\sh$ the shaking $\Sh$ with direction $\dire \coloneqq 1-2e^{-1}$.
\end{definition}

\section{Shaking functions}
\label{section:shaking_limit_shape}

Our aim is here to introduce the shaking  for functions of a certain kind. The connection with the shaking operation of~\textsection\ref{subsection:shakings} will be made in Section~\ref{section:shaking_graphs}.


\subsection{A class of functions}

Let $a \in \R_{> 0}$. We define $\myclass[a]$ to be the set of all (convex) functions $f : \R \to \R$ of class $\mathcal{C}^1$ such that:
\begin{subequations}
\label{subequations:assumptions_f}
\begin{gather}
\label{equation:f_even}
f \text{ is even,}
\\
\label{equation:f(s)=|s|}
\text{for any }|s| \geq a \text{ we have }f(s) = |s|,
\\
f \text{ is strictly convex on }[-a,a].
\end{gather}
\end{subequations}

\begin{example}
\label{example:Omega_convex}
By Lemma~\ref{lemma:derivative} we have $\Omega \in \myclass[2]$.
\end{example}

\begin{example}
\label{example:carre}
Let $\Sigma : \R \to \R$ be given by $\Sigma(x) = |x|$ if $|x| \geq 1$ and $\Sigma(x) = \frac{1}{2}(x^2+1)$ otherwise. Then $\Sigma \in \myclass[1]$.
\end{example}

The next map, defined between some sets $\myclass[a]$ via double scaling, will be useful when using the convergence of Theorem~\ref{theorem:LLN_partitions}.

\begin{definition}
\label{definition:feps}
Let $f : \R \to \R$ be any map and let $\myepsilon \in (0,1)$. We define the two maps $\falphaps{\pm} : \R \to \R$ by, for any $s \in \R$,
\[
\falphaps{\pm}(s) \coloneqq (1\pm \tfrac{\myepsilon}{2})\,f\left(\frac{s}{1\pm \tfrac{\myepsilon}{2}}\right).
\]
\end{definition}

For instance, the maps $\Omegaeps[1]{\pm}$ are depicted in Figure~\ref{figure:Omegaeps}. We now gather some informations about the map $\falphaps{\pm}$ for $f \in \myclass[a]$.

\begin{figure}[ht]
\begin{center}
\begin{tikzpicture}[scale=1.5]
\draw (-3,0) -- (3,0);
\draw (0,-.5) -- (0,3);

\draw[thick] (-3,3) -- (0,0) -- (3,3);

\draw[red,domain=-2:2] plot (\x,{\Omegaplot});
\draw[red] (2,2) node[below right]{$\Omega$};

\draw[blue, domain=-2:2] plot({1.5*\x},{1.5*\Omegaplot});
\draw[blue] (3,3) node[below right]{$\Omegaeps[1]{+}$};

\draw[green, domain=-2:2] plot({.5*\x},{.5*\Omegaplot});
\draw[green] (1,1)  node[below right]{$\Omegaeps[1]{-}$};
\end{tikzpicture}
\end{center}
\caption{The maps $\Omega$ and $\Omegaeps[1]{\pm}$}
\label{figure:Omegaeps}
\end{figure}
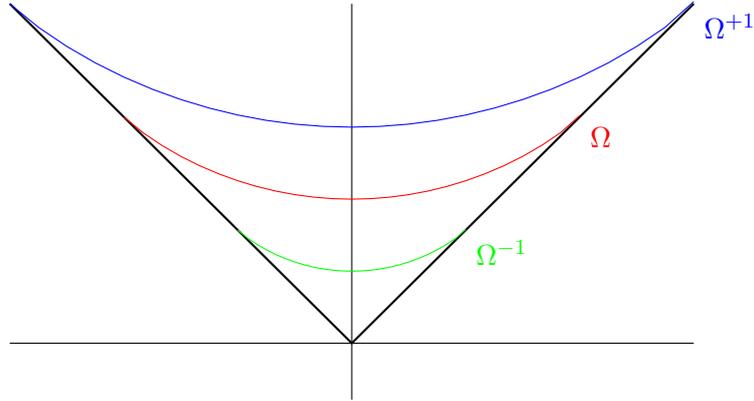

\begin{lemma}
\label{lemma:feps}
Let $a \in \R_{> 0}$ and $\myepsilon \in (0,1)$.  Let $f \in \myclass[a]$.
\begin{enumerate}
\item We have $\falphaps{\pm} \in \myclass[(2\pm\myepsilon)\frac{a}{2}]$, in particular for $a = 2$ we have $\falphaps{\pm} \in \myclass[2\pm \myepsilon]$.
\item
\label{item:falpha'}
For any $s \in \R$ we have $(\falphaps{\pm})'(s) = f'\left(\frac{s}{1\pm\tfrac{\myepsilon}{2}}\right)$.
\item The maps $\falphaps{+} - f$ and $f - \falphaps{-}$ are decreasing on $\R_{\geq 0}$.
\label{item:f+_-_f}
\item For any $s \in \R$ we have $\falphaps{-}(s) \leq f(s) \leq \falphaps{+}(s)$, and:
\begin{enumerate}[label=$(\alph*)$]
\item we have $\falphaps{-}(s) = f(s)$ if and only if $|s| \geq a$,
\item we have $\falphaps{+}(s) = f(s)$ if and only if $|s| \geq (2+\myepsilon)\tfrac{a}{2}$,
\end{enumerate}
\item We have $\|f - \falphaps{\pm}\|_\infty = \tfrac{\myepsilon}{2} f(0)$.
\end{enumerate}
\end{lemma}

\begin{proof}
\begin{enumerate}
\item Clear.
\item Clear.
\item Follows from the previous point and from the fact that $f'$ is increasing.
\item Since the functions are even, it suffices to prove the inequality on $\R_{\geq  0}$. We prove only the relations for $\falphaps{+}$, the ones for $\falphaps{-}$ being similar. By the previous point, we know that $g \coloneqq \falphaps{+} - f$  is decreasing on $\R_{\geq 0}$, thus we deduce that $g \geq 0$ since $g(s) = |s| - |s| = 0$ for $s \gg 0$. For the equality case, first note that the necessary condition holds. For the sufficient condition, since $f$ is strictly convex on $[-a,a]$ we have in fact $g'(s) < 0$ for all $s \in (0,a]$ thus by what precedes we have $g(s) > 0$ for all $s \in [0,a]$. Now for $s \in \bigl[a,(2+\myepsilon)\tfrac{a}{2}\bigr)$ we have $g'(s) = f'\bigl(\frac{s}{1+\tfrac{\myepsilon}{2}}\bigr) - 1   < 0$ since $1 = f'(a)$ and $f'$ is strictly increasing on $[0,a]$ and we conclude  the proof as before.
\item Again we only prove the result for $\falphaps{+}$. With the previous notation, we know that $g = \falphaps{+} - f$ is non-negative and decreases on $\R_{\geq 0}$ thus its maximum is reached at $s = 0$.
\end{enumerate}
\end{proof}

\subsection{About chords}
\label{subsection:about_chords}

In this whole part we fix $a \in \R_{> 0}$ and $\alpha \in [0,1)$, together with $f \in \myclass[a]$.
We will here study the chords of the curve of $f$ that have slope $\alpha$.
The last two equations in~\eqref{subequations:assumptions_f} imply that for all $|s| < a$ we have:
\begin{subequations}
\begin{equation}
\label{equation:f(s)>|s|}
f(s) > |s|,
\end{equation}
since $f$ is above its tangent at $s = a$, and:
\begin{equation}
\label{equation:f'<1}
|f'(s)| < 1,
\end{equation}
\end{subequations}
for all $s \in (-a,a)$, since $f'$ is strictly increasing on $[-a,a]$. In particular $f'$ induces a bijection $[-a,a] \to [-1,1]$, allowing us to make the following definition.

\begin{definition}
\label{definition:xalpha+}
We define the real number
\[\xalphapf{\alpha} \coloneqq f'^{-1}(\alpha),
\]
where $f'^{-1} : [-1,1] \to [-a,a]$ denotes the inverse function of $f'$, and we define the map $\Falpha{\alpha} : s \mapsto f(s) - \alpha s$.
\end{definition}

In the sequel we will write $\xalphap$ instead of $\xalphapf{\alpha}$ when $f$ is clear from the context. We give an example of $\xalphap$ in Figure~\ref{figure:xalphaprime}, with the map $\Sigma$ of Example~\ref{example:carre}.

\begin{proposition}
\label{proposition:xalpha+}
We have $\xalphap \in [0,a)$. Moreover, the map $\Falpha{\alpha}$ is convex, decreasing on $\bigl(-\infty,\xalphap\bigr)$ and increasing on $\bigl(\xalphap,+\infty\bigr)$, with limit $\pm\infty$ at $\pm \infty$.
\end{proposition}

\begin{proof}
By the above discussion, we know that $f'$ is increasing on $[-a,a]$. Hence, since $f'(0) = 0$ (since $f$ is even by~\eqref{equation:f_even}) and $f'(a) = 1$ (by~\eqref{equation:f(s)=|s|}) we have $\xalphap \in [0,a)$. We have $(\Falpha{\alpha})'(s) = f'(s) - \alpha$, thus recalling that $f$ is strictly convex on $[-a,a]$ we have that $f'$ is increasing thus we obtain both the convexity and the variations of $\Falpha{\alpha}$. 
\end{proof}

Thanks to Proposition~\ref{proposition:xalpha+}, we can make the following definition.

\begin{definition}
\label{definition:phi}
Let $\Falpha{\alpha}^{-1}$ be the inverse function of $\Falpha{\alpha}$ on $\bigl[\xalphap,+\infty\bigr)$. We define the map $\compfull{\alpha}{f}: \R \to \R$ by:
\[
\compfull{\alpha}{f} \coloneqq \Falpha{\alpha}^{-1} \circ \Falpha{\alpha}.
\]
\end{definition}

As before, we will write $\comp$ instead of $\compfull{\alpha}{f}$ when $f$ is clear from the context. Note that $\Falpha{\alpha}^{-1}$ is defined (and is a bijection) from $\Bigl[\Falpha{\alpha}\bigl(\xalphap\bigr),+\infty\Bigr)$ to $\bigl[\xalphap,+\infty\bigr)$.

\begin{lemma}
\label{lemma:phi}
The following properties are satisfied.
\begin{enumerate}
\item
\label{item:falpha_circ_phi}
We have $\Falpha{\alpha} \circ \Falpha{\alpha}^{-1} = \mathrm{id}$, in particular $\Falpha{\alpha} \circ \comp = \Falpha{\alpha}$.
\item
\label{item:phi_s_xalphap}
For any $s \geq \xalphap$ we have $\comp(s)= s$.
\item
\label{item:phi_bij}
 The map $\comp$ is decreasing bijection  $\bigl(-\infty,\xalphap\bigr] \to \bigl[\xalphap,+\infty\bigr)$. In particular, for any $s < \xalphap$ we have $\comp(s) > s$.
\item
\label{item:geometric_interpretation_phi}
For any $s < \xalphap$, the abscissa $\comp(s)$ is the unique element in $\bigl(\xalphap,+\infty\bigr)$ such that the chord in the curve of $f$ between the points of abscissae $s$ and $\comp(s)$  has slope $\alpha$.
\end{enumerate}
\end{lemma}

\begin{proof}
The first two points simply follow from the fact that $\Falpha{\alpha}^{-1}$ is the inverse of $\Falpha{\alpha}$ on $\bigl[\xalphap,+\infty\bigr)$. The third point follows from Proposition~\ref{proposition:xalpha+}, noting that $\Falpha{\alpha}^{-1}$ is increasing as the inverse of an increasing function. For the last point, for any $s < \xalphap$ we have, by the first point, $\Falpha{\alpha}\bigl(\comp(s)\bigr) = \Falpha{\alpha}(s)$, thus:
\[
f\bigl(\comp(s)\bigr)-\alpha\comp(s) = f(s) - \alpha s.
\]
We thus obtain, noting that $\comp(s) \neq s$ by the preceding point:
\[
\alpha = \frac{f\bigl(\comp(s)\bigr) - f(s)}{\comp(s)-s},
\]
whence the assertion on the slope. The uniqueness follows from the strict convexity of $f$.
\end{proof}

An illustration of Lemma~\ref{lemma:phi}\ref{item:geometric_interpretation_phi} is given at Figure~\ref{figure:delta_salpha}.
By Lemma~\ref{lemma:phi}\ref{item:phi_bij}, since by Proposition~\ref{proposition:xalpha+} we have $a > \xalphap$ we know that there exists a unique $\xalphamf{\alpha} \in \bigl(-\infty,\xalphap\bigr)$ such that:
\begin{equation}
\label{equation:def_xalpham}
\comp\bigl(\xalphamf{\alpha}\bigr) = a.
\end{equation}
Again, we will write $\xalpham$ instead of $\xalphamf{\alpha}$ when $f$ is clear from the context.  Note that, by convexity of $f$ and by and Lemma~\ref{lemma:phi}\ref{item:geometric_interpretation_phi}, since $\alpha \geq 0$ we have $a \leq \comp(- a)$ thus:
\begin{equation}
\label{equation:xalpham_-a}
\xalpham \geq -a.
\end{equation}

\begin{definition}
\label{definition:xprime}
For $x \in \R$, we denote by $\Lx{x}$ the line with slope $\alpha$ containing $(x,-x)$. Moreover, we denote by $\xalphamp = \xalphamp(f)$ (resp. $\xalphapp = \xalphapp(f)$) the unique $x \in \R$ such that $\Lx{x}$ contains the point of the curve of $f$ with abscissa $\xalpham$ (resp. $\xalphap$).
\end{definition}

We give an example of $\xalpham, \xalphamp,\xalphapp$ in Figure~\ref{figure:xalphaprime}. The next result follows from~\eqref{subequations:assumptions_f} and from Lemma~\ref{lemma:phi}. Note that the term ``below'' (resp. ``above'') means ``strictly below (resp. above) except maybe at the extremities''.

\begin{figure}
\begin{center}
\begin{tikzpicture}[scale=5]
\draw (-1,0) -- (1.1,0);
\draw (0,-.2) -- (0,1.1);

\draw (-.8,.8) -- (0,0) -- (.8,.8);

{
\color{green}
\draw[dashed] (.25,.53125) -- (.25,0) node[below]{$\xalphap$};
\draw[domain=-1:1.1] plot (\x,{(\x -.25)/4+.53125});
\draw[dashed] (-.375,.375) -- (-.375,0) node[below]{$\xalphapp$};
}
\draw[domain=-1:1, smooth,red] plot (\x,{\Sigmaplot});
\draw[red] (-.85,1) node[right]{$y=f(x)$};

\draw (-1,1) -- (-.8,.8);
\draw (.8,.8) -- (1.1,1.1);

\color{orange}
\draw[domain=-1:1.1] plot (\x,{(\x -1)/4+1});
\draw[dashed] (-.5,.625) -- (-.5,0) node[below]{$\xalpham$};
\draw[dashed](-.6,.6) -- (-.6,0) node[below left]{$\xalphamp$};


\end{tikzpicture}
\end{center}
\caption{Example of $\xalphap,\xalpham,\xalphapp,\xalphamp$. The green and orange lines have slope $\alpha$.}
\label{figure:xalphaprime}
\end{figure}
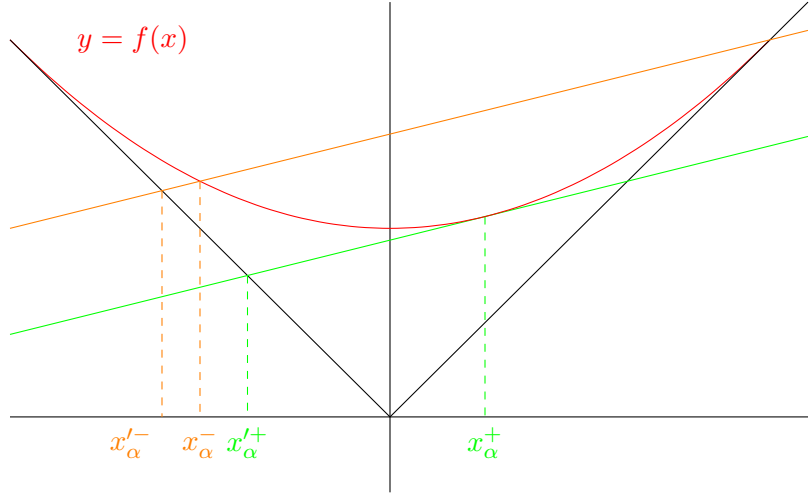

\begin{lemma}
\label{lemma:intersection_f_L}
We have $-a \leq \xalphamp < \xalphapp < 0$. Moreover, if $x \in \R$ then on $[-a,a]$ we have:
\begin{itemize}
\item if $x \leq -a$ then $\Lx{x}$ is above the curve of $f$,
\item if $-a < x \leq \xalphamp$ then $\Lx{x}$ is below and then above the curve of $f$,
\item if $\xalphamp < x < \xalphapp$ then $\Lx{x}$ is below, then above and then below the curve of $f$,
\item if $x \geq \xalphapp$ then $\Lx{x}$ is below the curve of $f$.
\end{itemize}
\end{lemma}

We now study the map $\comp$ in more details.

\begin{lemma}
\label{lemma:phi_precis}
The following properties are satisfied.
\begin{enumerate}
\item
\label{item:phi_decreasing_bij_x-_x+}
The map $\comp$ induces a decreasing bijection from $\bigl[\xalpham,\xalphap\bigr]$ onto $\bigl[\xalphap,a\bigr]$.
\item
\label{item:phi_x_leq_x-}
For any $s \leq \xalpham$ we have $\comp(s) = (1-\alpha)^{-1}\Falpha{\alpha}(s)$.
\end{enumerate}
\end{lemma}

\begin{proof}
The first point follows from Lemma~\ref{lemma:phi} and~\eqref{equation:def_xalpham}. For the second one, by~\eqref{equation:def_xalpham} and Lemma~\ref{lemma:phi}\ref{item:falpha_circ_phi} we have $\Falpha{\alpha}\bigl(\xalpham\bigr) = \Falpha{\alpha}(a)$ thus $\Falpha{\alpha}\bigl(\xalpham\bigr) = (1-\alpha)a$ by~\eqref{equation:f(s)=|s|}. Since $\xalpham < \xalphap$, by Proposition~\ref{proposition:xalpha+} we know that if $s \leq \xalpham$ then $\Falpha{\alpha}(s) \geq \Falpha{\alpha}\bigl(\xalpham\bigr)$ thus, since $\alpha < 1$, there exists $b \geq a$ such that:
\begin{equation}
\label{equation:falpha_b}
\Falpha{\alpha}(s) = (1-\alpha)b.
\end{equation}
Since $\xalphap < a \leq b$, by~\eqref{equation:f(s)=|s|} we have $ (1-\alpha)b = \Falpha{\alpha}(b)$ thus we deduce that:
\[
\Falpha{\alpha}(s) = \Falpha{\alpha}(b),
\]
thus $\comp(s) = b$ by Lemma~\ref{lemma:phi}.
 This concludes the proof since $b= (1-\alpha)^{-1}\Falpha{\alpha}(s)$ by~\eqref{equation:falpha_b}.
\end{proof}

 The next definition will play a crucial role in the statement of our first main result (in Section~\ref{section:shaking_graphs}).

\begin{definition}
\label{definition:delta}
We define the map $\mydeltaprim^f : \R \to \R$ by:
\[
\mydeltaprim^f \coloneqq (1-\alpha)^{-1}\Falpha{\alpha} - \comp = (1-\alpha)^{-1} \Falpha{\alpha} - \Falpha{\alpha}^{-1} \circ \Falpha{\alpha}.
\]
\end{definition}

As usual, we will write $\mydeltaprim$ instead of $\mydeltaprim^f$ when $f$ is clear from the context. We will write $\mydelta$ for the image of $x \in \R$ under $\mydeltaprim$. In the sequel, we will focus on the interval $\bigl(-\infty,\xalphap\bigr]$.

\begin{lemma}
\label{lemma:delta_increasing}
The following properties are satisfied.
\begin{enumerate}
\item
\label{item:delta_0}
For any $x \leq \xalpham$ we have $\mydelta = 0$.
\item The map  $\mydeltaprim$ is increasing on $\bigl[\xalpham,\xalphap\bigr]$.
\end{enumerate}
In particular, the map $x \mapsto x+ \mydelta$ is increasing on $\bigl(-\infty,\xalphap\bigr]$.
\end{lemma}

\begin{proof}
The first point immediately follows from Lemma~\ref{lemma:phi}\ref{item:phi_x_leq_x-}. 
For the second one, consider the map $d : y \mapsto (1-\alpha)^{-1} \Falpha{\alpha}(y) - y$. We have:
\[
(1-\alpha)d'(y) = f'(y) - \alpha -(1-\alpha)= f'(y)-1,
\]
thus $d$ is decreasing on $(-\infty,a]$ by~\eqref{subequations:assumptions_f}. Now by Lemma~\ref{lemma:phi} we have $d \circ \phi = \mydeltaprim$, and Lemma~\ref{lemma:phi}\ref{item:phi_decreasing_bij_x-_x+} concludes the proof.  The last assertion is immediate since we know that $\delta$ is non-decreasing on $\bigl(-\infty,\xalphap\bigr]$.
\end{proof}

We now give the geometric interpretation of $\mydelta$.  For $s \in \R$, let~$L_s$ be the line  of slope $\alpha$ that crosses the curve of $f$ at the point of abscissa $s$. Since $\alpha \neq 1$, this line $L_s$ also crosses the line $\bissp$ of equation $y = x$.

\begin{lemma}
\label{lemma:geometric_interpretation_delta}
Let $s \in \R$. The intersection point between $L_s$ and $\bissp$ has abscissa $(1-\alpha)^{-1}\Falpha{\alpha}(s) = \comp(s) + \mydelta[s]$.
\end{lemma}

\begin{proof}
The line $L_s$ has equation $y = \alpha(x-s)+f(s)$. Thus, if $x$ is the abscissa of the intersection point between $L_s$ and $\bissp$ we obtain $x = \alpha(x-s)+f(s)$ thus $(1-\alpha)x = \Falpha{\alpha}(s)$. We conclude the proof by Definition~\ref{definition:delta}.
\end{proof}

Note that the point of abscissa $\xalphap$ has a particular role, since by Definition~\ref{definition:xalpha+} the line $L_{\xalphap}$ is exactly the tangent of $f$ at the point of abscissa $\xalphap$. In this case, recalling Lemma~\ref{lemma:phi}\ref{item:phi_s_xalphap} we have $\comp\bigl(\xalphap\bigr) = \xalphap$, leading to the following definition:
\[
\suppfalpha(f) \coloneqq \xalphap(f) + \mydeltaprim^f_{\xalphap(f)},
\]
that is:
\begin{equation}
\label{equation:suppfalpha}
\suppfalpha = \xalphap + \mydelta[\xalphap].
\end{equation}
We picture in Figure~\ref{figure:delta_salpha} the result of Lemma~\ref{lemma:geometric_interpretation_delta} and the definition~\eqref{equation:suppfalpha}. Note that by Proposition~\ref{proposition:xalpha+} and Lemma~\ref{lemma:delta_increasing} we have:
\begin{equation}
\label{equation:suppfalpha_ineq}
\suppfalpha > \xalphap \geq 0,
\end{equation}
and the image of $\bigl(-\infty,\xalphap\bigr]$ under the map $x \mapsto x+\mydelta$ is $\bigl(-\infty,\suppfalpha\bigr]$.

\begin{figure}
\begin{center}
\begin{tikzpicture}[scale=5]
\draw (-1,0) -- (1.1,0);
\draw (0,-.2) -- (0,1.1);

\draw (-.8,.8) -- (0,0) -- (.8,.8);

{
\color{green}
\draw[dashed] (.25,.53125) -- (.25,0) node[below]{$x_\alpha^+$};
\draw[domain=-1:1.1] plot (\x,{(\x -.25)/4+.53125});
\draw[dashed] (.625,.625) -- (.625,0) node[above left]{$\suppfalpha$};
}
\draw[domain=-1:1, smooth,red] plot (\x,{\Sigmaplot});
\draw[red] (-.85,1) node[right]{$y=f(x)$};

\draw (-1,1) -- (-.8,.8);
\draw (.8,.8) -- (1.1,1.1) node[right]{$\bissp$};


\color{blue}
\draw[dashed] (-.25,.53125) -- (-.25,0) node[below]{$x$};
\draw[domain=-1:1.1] plot (\x,{(\x+.25)/4+.53125});

\draw[dashed] (.75,.78125) -- (.75,0);
\draw (.68,0) node[below]{$\scriptstyle\comp(x)$};
\draw[dashed] (.79166,.79166) -- (.79166,0);
\draw (.75,0) node[below right]{$\scriptstyle \comp(x)+\mydelta$};

\draw[ultra thick]  (.75,.78125) -- (.79166,.79166);
\end{tikzpicture}
\end{center}
\caption{Geometric interpretation of $\comp$, $\mydelta$ and $\suppfalpha$}
\label{figure:delta_salpha}
\end{figure}

\begin{corollary}
\label{corollary:support_fe}
We have:
\[
\suppfalpha = (1-\alpha)^{-1}\Falpha{\alpha}\bigl(\xalphap\bigr) < a.
\]
\end{corollary}

\begin{proof}
We have recalled that $\comp(\xalphap) = \xalphap$, thus by~\eqref{equation:suppfalpha} the equality immediately follows from Lemma~\ref{lemma:geometric_interpretation_delta}. The inequality follows from the fact that $\Falpha{\alpha}$ is increasing on $[\xalphap,+\infty)$, noting that $(1-\alpha)a = \Falpha{\alpha}(a)$ by~\eqref{equation:f(s)=|s|} and recalling that $\xalphap < a$ by Proposition~\ref{proposition:xalpha+}.
\end{proof}

\subsection{Shaking functions}
\label{subsubsection:shaking_functions}

Again, fix $a \in \R_{> 0}$ and $\alpha \in [0,1)$ and let $f \in \myclass[a]$. Recall the notation $\Falpha{\alpha} = f - \alpha\,\mathrm{id}_{\R}$, recall from Lemma~\ref{lemma:delta_increasing} and~\eqref{equation:suppfalpha} that $x \mapsto x + \mydelta$ is a bijection $\bigl(-\infty,\xalphap\bigr] \to (-\infty,\suppfalpha\bigr]$  and denote by $\invdelta(f) = \invdelta : \bigl(-\infty,\suppfalpha\bigr] \to \bigl(-\infty,\xalphap\bigr]$ its inverse. The next object will be involved in a crucial way in our main result.

\begin{definition}
\label{definition:falpha}
We define  the map $\falpha[\alpha] : \bigl(-\infty,\suppfalpha\bigr] \to \R$ by:
\[
\falpha = \alpha \,\mathrm{id}_\R + \Falpha{\alpha}\circ \tau.
\]
We extend $\falpha$ to $\R$ by setting $\falpha(x) \coloneqq x$ for all $x > \suppfalpha$.
\end{definition}


\begin{proposition}
\label{proposition:falpha_x+mydelta}
The map $\falpha$ on $(-\infty, \suppfalpha]$ is given for any $x \leq \xalphap$ by:
\[
\falpha(x+\mydelta) = f(x) + \alpha\mydelta.
\]
In particular:
\begin{itemize}
\item for any  $x \leq \xalpham$ we have $\falpha(x) = f(x)$, thus $\falpha(x) = |x|$ for all $x \leq -a$,
\item for any $x \in (-a,\suppfalpha)$ we have $\falpha(x) > |x|$,
\item we have $\falpha(\suppfalpha) = \suppfalpha$.
\end{itemize}
\end{proposition}

\begin{proof}
For any $y \leq \suppfalpha$ we have $\falpha(y) = \alpha y + \Falpha{\alpha}\bigl(\tau(y)\bigr)$, thus for any $x \leq \xalphap$ we have:
\begin{align*}
\falpha(x+\mydelta)
&=
\alpha(x+\mydelta) + \Falpha{\alpha}(x)
\\
&=
\alpha x + \alpha\mydelta + f(x)-\alpha x
\\
&=
f(x) + \alpha\mydelta,
\end{align*}
as announced. Now for $x \leq \xalpham$ we have $\mydelta = 0$ by Lemma~\ref{lemma:delta_increasing} thus $\falpha(x) = f(x)$ and we obtained the announced result by~\eqref{equation:f(s)=|s|}. Moreover, by~\eqref{equation:suppfalpha} we have:
\begin{align*}
\falpha(\suppfalpha)
&=
\falpha(\xalphap + \mydelta[\xalphap])
\\
&=
f(\xalphap) + \alpha\mydelta[\xalphap]
\\
&=
f(\xalphap) + \alpha(\suppfalpha - \xalphap)
\\
&=
\Falpha{\alpha}(\xalphap) + \alpha\suppfalpha,
\end{align*}
whence the result by Corollary~\ref{corollary:support_fe}. Now let $x \in (-a,\suppfalpha)$ and let us prove that $\falpha(x)>|x|$. The result is clear if $x \leq \xalpham$ by~\eqref{equation:f(s)>|s|} since then $\falpha(x) = f(x)$, thus we assume $x \in (\xalpham,\suppfalpha)$. By Lemma~\ref{lemma:delta_increasing} we can thus write $x = s + \mydelta[s]$ for $s \in (\xalpham,\xalphap)$. If $x < 0$ then by Lemma~\ref{lemma:delta_increasing} we have $s < 0$ thus $|s| > |x|$ and, using~\eqref{subequations:assumptions_f},
\[
\falpha(x) = \falpha(s+\mydelta[s]) = f(s) + \alpha\mydelta[s] \geq f(s) \geq |s| > |x|,
\]
as desired.
Thus, we now assume that $x \geq 0$ and we want to prove that $f(s)+\alpha\mydelta[s] > s + \mydelta[s]$. Recalling Definition~\ref{definition:delta}, we have:
\begin{align*}
f(s)+\alpha\mydelta[s] - s - \mydelta[s]
&=
f(s)-s - (1-\alpha)\mydelta[s]
\\
&=
f(s) - s - \Falpha{\alpha}(s) + (1-\alpha)\comp(s)
\\
&=
f(s) - s - f(s) + \alpha s + (1-\alpha)\comp(s)
\\
&=
(1-\alpha) \bigl(\comp(s)-s\bigr),
\end{align*}
thus Lemma~\ref{lemma:phi}\ref{item:phi_bij} concludes the proof.
\end{proof}

Note that the quantity $\alpha\mydelta$ in Proposition~\ref{proposition:falpha_x+mydelta} is exactly the ordinates difference of the extremities of the segment joining the points of abscissa $\comp(x)$ and $\comp(x)+\mydelta$  of the line $L_s$ of Lemma~\ref{lemma:geometric_interpretation_delta}.
 Hence, Proposition~\ref{proposition:falpha_x+mydelta} asserts that the value of $\falpha$ at $x+\mydelta$ is exactly $f(x)$ plus this same difference $\alpha\mydelta$. We illustrate this fact in Figure~\ref{figure:falpha}.

\begin{figure}
\begin{center}
\begin{tikzpicture}[scale=6]
\clip (-.5,-.2) rectangle (1.1,1);

\draw (-1,0) -- (1.1,0);
\draw (0,-.2) -- (0,1.1);

\draw (-.8,.8) -- (0,0) -- (.8,.8);

\draw[domain=-1:1, smooth,red] plot (\x,{\Sigmaplot});
\draw[red] (-.35,.65) node{$\scriptstyle y=f(x)$};

\draw (-1,1) -- (-.8,.8);
\draw (.8,.8) -- (1.1,1.1);


\node (A) at (0,.54166){};
\color{red}
\draw[dashed] (-.20833,.54166) -- (0,.54166);
\draw[dashed,<-] (A) -- (.25,.8) node[above]{$\functalpha{\alpha}{f}(x+\mydelta)$};

\color{blue}
\draw[dashed] (-.25,.53125) -- (-.25,0);
\draw (-.27,0) node[below]{$x$};
\draw[domain=-1:1.1] plot (\x,{(\x+.25)/4+.53125});
\draw[dashed] (-.20833,.54166) -- (-.20833,0);
\draw (-.25,0) node[below right]{$ x + \mydelta$};

\draw[dashed] (.75,.78125) -- (.75,0);
\draw (.68,0) node[below]{$\scriptstyle\comp(x)$};
\draw[dashed] (.79166,.79166) -- (.79166,0);
\draw (.75,0) node[below right]{$\scriptstyle \comp(x)+\mydelta$};

\draw[ultra thick]  (.75,.78125) -- (.79166,.79166);
\draw[ultra thick] (-.25,.53125) -- (-.20833,.54166);

\end{tikzpicture}
\end{center}
\caption{Construction of $\falpha$}
\label{figure:falpha}
\end{figure}

\begin{remark}
\label{remark:Shf_|x|gg0}
Recalling~\eqref{equation:f(s)=|s|} and~\eqref{equation:suppfalpha_ineq}, we deduce that $\falpha(x) = |x|$ for $x \leq -a$ and $x \geq \suppfalpha$.
\end{remark}

We have the following particular case of Proposition~\ref{proposition:falpha_x+mydelta} for $\alpha = 0$.

\begin{corollary}
\label{corollary:limit_shape_alpha0}
The map $\falpha[0] : \R \to \R$ is given by:
\begin{align*}
\falpha[0](x) &= f(x),  &&\text{for all } x \leq -a,
\\
\falpha[0]\bigl(2x+ f(x)\bigr) &= f(x), &&\text{for all }  x \in (-a, 0),
\\
\falpha[0](x) &= x, &&\text{for all } x \geq f(0).
\end{align*}
\end{corollary}

\begin{proof}
First note that $\Falpha{0} = f$. Recalling from~\eqref{equation:f_even} that $f$ is even,  we have $f'(0) = 0$ thus $\xalphap[0] = 0$ (recalling Definition~\ref{definition:xalpha+}). Hence, the map $\Falpha{0}^{-1}$ is defined as the inverse map of $f$ on $[0,+\infty)$. Using again the fact that $f$ is even,  for all $x \in (-\infty,0)$ we have $f(x) = f(-x)$ thus since $-x \geq 0$ we obtain:
\[
\comp[0](x) = \Falpha{0}^{-1}\circ f(x) = \Falpha{0}^{-1}\circ f(-x) = -x.
\]
Hence, for all $x \leq 0$ we have:
\begin{equation}
\label{equation:mydelta0}
\mydelta = f(x) - \comp(x) = f(x) + x,
\end{equation}
thus $x + \mydelta = f(x) + 2x$. Hence, we have the announced formula for $\falpha[0]\bigl(2x + f(x)\bigr)$ for all $x \leq \xalphap[0] = 0$ (note that $f(x) = -x$ for $x \leq -a$). Finally, by definition we have $\falpha[0](x) = x$ for $x > \suppfalpha[0]$, and we conclude the proof since $\suppfalpha[0] = \xalphap[0] + \mydelta[{\xalphap[0]}] = 0 + \mydelta[0]$ and by~\eqref{equation:mydelta0} we have $\mydelta[0] = f(0)+0=f(0)$.
\end{proof}


\begin{lemma}
\label{lemma:derivative_falpha}
The map $\falpha$ is derivable on $(-\infty,\suppfalpha)$ and for any $x \in (-\infty,\suppfalpha)$ we have $\falpha'(x) < \alpha$.
\end{lemma}

\begin{proof}
By Definition~\ref{definition:delta} we know that $\mydeltaprim^f$ is derivable thus $\comp$ as well and thus $\falpha$ too. The second statement is clear for $x \leq -a$ since $\falpha(x) = -x$ by Proposition~\ref{proposition:falpha_x+mydelta}.
Using Definition~\ref{definition:falpha}, for all $x \in (-a,\suppfalpha)$ we have:
\begin{equation}
\label{equation:falpha'_expression}
\falpha'(x) = \alpha + \tau'(x) \times  \Falpha{\alpha}'\circ \tau(x),
\end{equation}
 thus it suffices to prove that:
\begin{equation}
\label{equation:falpha'}
\tau' \times \Falpha{\alpha}'\circ \tau < 0, \quad \text{on } (-a,\suppfalpha).
\end{equation}
By definition, the map $\tau : (-\infty,\suppfalpha] \to \bigl(-\infty,\xalphap\bigr]$ is the inverse map of the increasing bijection $x\mapsto x+\mydelta$  thus $\tau$ is increasing as well.  We deduce that~\eqref{equation:falpha'} is satisfied since by Proposition~\ref{proposition:xalpha+} we know that $\Falpha{\alpha}$ is decreasing on $\bigl(-\infty,\xalphap\bigr]$.
\end{proof}

\begin{remark}
By Proposition~\ref{proposition:xalpha+} we have $\Falpha{\alpha}'(\xalphap) = 0$ thus since $\tau(\suppfalpha) = \xalphap$ we obtain from~\eqref{equation:falpha'_expression} that the left derivative of $\falpha$ at $\suppfalpha$ is $\alpha$. Since $\alpha  \neq 1$, we deduce that $\falpha$ is not derivable at $\suppfalpha$ since $\falpha(x) = x$ for $x > \suppfalpha$ (by Proposition~\ref{proposition:falpha_x+mydelta}).
\end{remark}

\begin{remark}
\label{remark:falpha_CVU}
By Lemma~\ref{lemma:phi}\ref{item:geometric_interpretation_phi} and the chordal slope lemma, we have $\xalpham = \comp^{-1}(a) \to a$ as $\alpha \to 1$. Hence, we deduce from Proposition~\ref{proposition:falpha_x+mydelta} that $\bigl(\falpha\bigr)_{\alpha \in (0,1)}$ converges pointwise to $f$ on $[-a,a]$ as $\alpha \to 1$. By Lemma~\ref{lemma:derivative_falpha} and Remark~\ref{remark:Shf_|x|gg0}, we know that each $\falpha$ is $1$-Lipschitz and thus a standard result tells that the preceding convergence is uniform.
\end{remark}

We can now define the shapes that will be of interest for our problem with the Plancherel measure. Recall from Definition~\ref{definition:sh} that we will be interested in the shakings $\sh = \Sh[1-2e^{-1}]$.

\begin{definition}
\label{definition:Omegae}
For any $e \geq 2$, we define $\Omegae \coloneqq  \functalpha{1-2e^{-1}}{\Omega}$.
\end{definition}

 In Figure~\ref{figure:reg_limit_shape}  we represent the curves $\Omegae$ for $e \in \{2,3,4\}$, together with the curve $\Omega$.


\begin{figure}
\begin{center}
\includegraphics[width=\mywidth]{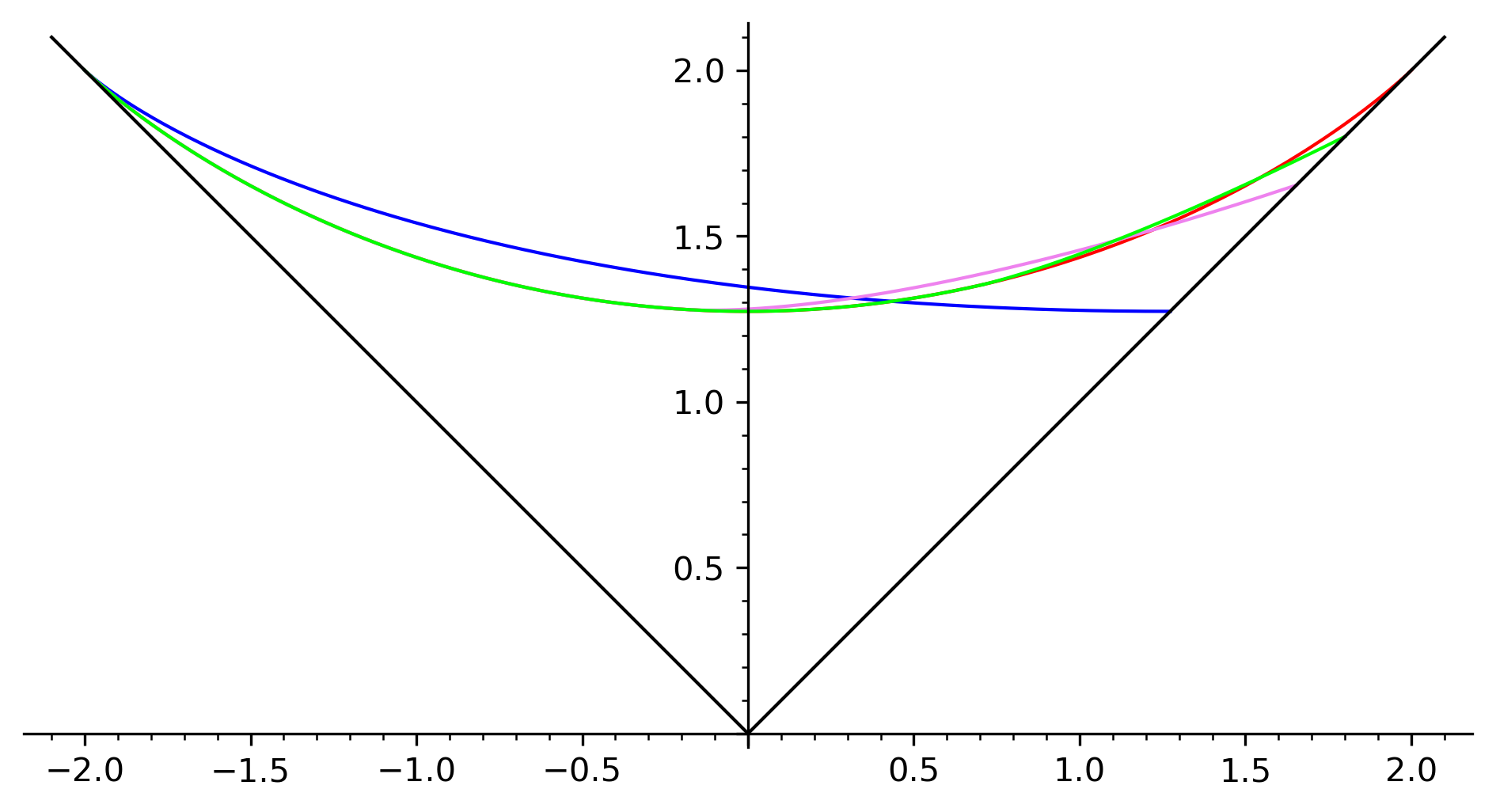}
\end{center}
\caption{Shapes $\Omega$ (in red), $\Omegae[2]$ (in blue), $\Omegae[3]$ (in violet) and $\Omegae[4]$ (in green).}
\label{figure:reg_limit_shape}
\end{figure}

\subsection{Compatibility with the shifts}
\label{subsection:compatibility_shifts}

Let $a \in \R_{> 0}$ and $\alpha \in [0,1)$, together with $f \in \myclass[a]$.
In this part, which involves many but simple calculations, we study the behaviour of $\functalpha{\alpha}{\falphaps{\pm}}$ with respect to $\falpha$. These results will be used in Section~\ref{section:limit_shape_regularisation}.

\begin{lemma}
\label{lemma:salphaeta}
Let $\myepsilon \in (0,1)$. Then:
\begin{itemize}
\item $\suppfalpha\bigl(\falphaps{\pm}\bigr) = \bigl(1\pm \frac{\eta}{2}\bigr) \suppfalpha(f)$,
 \item for all $x \in \R$ we have $\mydeltaprim^{\falphaps{\pm}}_x = \bigl(1\pm \frac{\eta}{2}\bigr) \mydeltaprim^f_{x/(1\pm \frac{\eta}{2})}$, in other words $\mydeltaprim^{(\falphaps{\pm})} = (\mydeltaprim^f)^{\pm \eta}$.
\end{itemize}
\end{lemma}

\begin{proof}
Let $\theta \coloneqq 1\pm\tfrac{\myepsilon}{2}$.
Note that for any $x \in \R$ we have (note that $\Falpha[(\falphaps{\pm})]{\alpha}$ is defined by Definition~\ref{definition:xalpha+}):
\begin{align}
\Falpha[(\falphaps{\pm})]{\alpha}(x)
&=
\falphaps{\pm}(x) - \alpha x
\notag
\\
&=
\theta f\left(\tfrac{x}{\theta}\right)-\theta\alpha\tfrac{x}{\theta}
\notag
\\
&=
\theta\Falpha{\alpha}\left(\tfrac{x}{\theta}\right)
\notag
\\
&=
\falphapsprim{(\Falpha{\alpha})}{\pm}(x),
\notag
\end{align}
thus $\Falpha[(\falphaps{\pm})]{\alpha} = \falphapsprim{(\Falpha{\alpha})}{\pm}$, and without ambiguity we can write this function $\Falpha[\falphaps{\pm}]{\alpha}$. By Lemma~\ref{lemma:feps}, for all $x \in \R$ we have:
\[
\bigl(\Falpha[\falphaps{\pm}]{\alpha}\bigr)'(x) = \Falpha{\alpha}'\left(\tfrac{x}{\theta}\right).
\]
 Recalling from Proposition~\ref{proposition:xalpha+} that $\xalphap(f)$ (resp. $\xalphap\bigl(\falphaps{\pm}\bigr)$) is the unique $x \in \R$ such that $\Falpha{\alpha}'(x) = 0$ (resp. ${(\Falpha[\falphaps{\pm}]{\alpha})}'(x) = 0$), we deduce that:
\[
\xalphap(\falphaps{\pm}) = \theta\xalphap(f).
\]
Now recall from Corollary~\ref{corollary:support_fe} that $(1-\alpha)\suppfalpha(f) = \Falpha{\alpha}\bigl(\xalphap(f)\bigr)$. We thus have:
\begin{align*}
(1-\alpha)\suppfalpha\bigl(\falphaps{\pm}\bigr)
&=
\Falpha[\falphaps{\pm}]{\alpha}\bigl(\xalphap(\falphaps{\pm})\bigr)
\\
&=
\Falpha[\falphaps{\pm}]{\alpha}\left(\theta\xalphap(f)\right)
\\
&=
\theta \Falpha{\alpha}\bigl(\xalphap(f)\bigr)
\\
&=
\theta (1-\alpha)\suppfalpha(f),
\end{align*}
which proves the first item of the Lemma.

Now recalling Definition~\ref{definition:phi}, for any $x,y \in \R$ we have $x = \Falpha{\alpha}^{-1}(y) \iff x \geq \xalphap(f)$ and $y = \Falpha{\alpha}(x)$. We thus have:
\begin{align*}
x = \bigl(\Falpha[\falphaps{\pm}]{\alpha}\bigr)^{-1}(y)
&\iff
x \geq \xalphap\bigl(\falphaps{\pm}\bigr) \text{ and } y = \Falpha[\falphaps{\pm}]{\alpha}(x)
\\
&\iff
x \geq \theta\xalphap(f) \text{ and } y = \theta\Falpha{\alpha}\left(\tfrac{x}{\theta}\right)
\\
&\iff
\tfrac{x}{\theta} = \Falpha{\alpha}^{-1}\left(\tfrac{y}{\theta}\right),
\end{align*}
thus for all $y \geq \xalphap\bigl(\falphaps{\pm}\bigr) = \theta\xalphap(f)$ we have:
\[
\bigl(\Falpha[\falphaps{\pm}]{\alpha}\bigr)^{-1}(y)
=
\theta\Falpha{\alpha}^{-1}\left(\tfrac{y}{\theta}\right).
\]
We deduce that for all $x \in \R$ we have:
\begin{align*}
\bigl(\Falpha[\falphaps{\pm}]{\alpha}\bigr)^{-1} \circ \Falpha[\falphaps{\pm}]{\alpha}(x)
&=
\theta \Falpha{\alpha}^{-1}\left(\frac{\Falpha[\falphaps{\pm}]{\alpha}(x)}{\theta}\right)
\\
&=
\theta \Falpha{\alpha}^{-1} \circ \Falpha{\alpha}\left(\tfrac{x}{\theta}\right),
\end{align*}
and thus:
\begin{align*}
\mydeltaprim^{\falphaps{\pm}}_x
&=
(1-\alpha)^{-1} \Falpha[\falphaps{\pm}]{\alpha}(x) - \bigl(\Falpha[\falphaps{\pm}]{\alpha}\bigr)^{-1} \circ \Falpha[\falphaps{\pm}]{\alpha}(x)
\\
&=
(1-\alpha)^{-1} \theta \Falpha{\alpha}\bigl(\tfrac{x}{\theta}\bigr) - \theta \Falpha{\alpha}^{-1} \circ \Falpha{\alpha}\bigl(\tfrac{x}{\theta}\bigr)
\\
&=
\theta\left[(1-\alpha)^{-1} \Falpha{\alpha}\bigl(\tfrac{x}{\theta}\bigr) -  \Falpha{\alpha}^{-1} \circ \Falpha{\alpha}\bigl(\tfrac{x}{\theta}\bigr)\right]
\\
&=
\theta \mydeltaprim^f_{\frac{x}{\theta}},
\end{align*}
which concludes the proof.
\end{proof}

\begin{proposition}
\label{proposition:falphaeps}
Let  $\myepsilon \in (0,1)$. We have:
\[
\functalpha{\alpha}{\falphaps{\pm}} = \falpha^{\pm \myepsilon}.
\]
\end{proposition}

\begin{proof}
Set $\theta \coloneqq 1 \pm \tfrac{\myepsilon}{2}$.
We first prove the equality on $\bigl(\suppfalpha(\falphaps{\pm}),+\infty)$. On this interval we have $\functalpha{\alpha}{\falphaps{\pm}}(x) = x$, moreover:
\[
\falphapsprim{\falpha}{\pm}(x)
=
\theta \falpha\bigl(\tfrac{x}{\theta}\bigr)
=x,
\]
since:
\[\frac{x}{\theta} > \frac{\suppfalpha(\falphaps{\pm})}{\theta} = \suppfalpha,
\]
where the last equality follows from Lemma~\ref{lemma:salphaeta}.

We now take $x \leq \xalphap(\falphaps{\pm})$, so that $x + \mydelta^{\falphaps{\pm}} \leq \suppfalpha(\falphaps{\pm}) = \tfrac{\suppfalpha}{\theta}$. We deduce from Proposition~\ref{proposition:falpha_x+mydelta} and Lemma~\ref{lemma:salphaeta} that:
\begin{align*}
\falphapsprim{\falpha}{\pm}\bigl(x + \mydeltafull{x}{\falphaps{\pm}}\bigr)
&=
\falphaps{\pm}(x) + \alpha \mydeltafull{x}{\falphaps{\pm}}
\\
&=
\theta f\bigl(\tfrac{x}{\theta}\bigr) + \theta \alpha \mydeltafull{x/\theta}{f}
\\
&=
\theta \falpha \left(\frac{x}{\theta} + \mydeltafull{x/\theta}{f}\right)
\\
&=
\theta \falpha \left(\frac{x}{\theta} + \frac{\mydeltafull{x}{\falphaps{\pm}}}{\theta}\right)
\\
&=
\functalpha{\alpha}{\falphaps{\pm}}\bigl(x  + \mydeltafull{x}{\falphaps{\pm}}\bigr),
\end{align*}
which concludes the proof.
\end{proof}

\begin{proposition}
\label{proposition:diff_fepse_fe}
We have $\mathopen{\|}\falpha[\alpha]-\falpha[\alpha]^{\pm \myepsilon}\|_{\infty,\R} \to 0$ as $\myepsilon \to 0$.
\end{proposition}

\begin{proof}
First, note that $\falpha - \falphapsprim{\falpha}{\pm}$ has compact support $K$ by Remark~\ref{remark:Shf_|x|gg0}. Moreover, by Remark~\ref{remark:falpha_CVU} we know that $\falpha$ is $1$-Lipschitz. Hence, for all $x \in K$ we have:
\begin{align*}
\falphapsprim{\falpha}{\pm}(x) - \falpha(x)
&=
\bigl(1\pm\tfrac{\myepsilon}{2}\bigr)\falpha\left(\frac{x}{1\pm\tfrac{\myepsilon}{2}}\right) - \falpha(x)
\\
&=
\falpha\left(\frac{x}{1\pm\tfrac{\myepsilon}{2}}\right) - \falpha(x) \pm \tfrac{\myepsilon}{2}\falpha\left(\frac{x}{1\pm\tfrac{\myepsilon}{2}}\right).
\end{align*}
We can assume $\myepsilon \in (0,1)$, so that $\frac{x}{1\pm\tfrac{\myepsilon}{2}}$ lies in the compact $K' \coloneqq \{2y : y \in K\}$ whenever $x \in K$. With $\|K\|_{\infty} \coloneq \max_{x \in K} |x|$ we deduce that for any $x \in K$ we have:
\begin{align*}
\bigl|\falphapsprim{\falpha}{\pm}(x) - \falpha(x)\bigr|
&\leq
|x| \left|\frac{1}{1\pm\tfrac{\myepsilon}{2}}-1\right| + \tfrac{\myepsilon}{2}\mathopen{\|}\falpha\|_{\infty,K'}
\\
&\leq
\|K\|_{\infty} \frac{\eta}{2\pm \myepsilon} + \tfrac{\myepsilon}{2}\mathopen{\|}\falpha\|_{\infty,K'},
\end{align*}
whence the result.
\end{proof}

\subsection{Relative position}
\label{subsection:relative_position}

Let $a > 0$ and $\alpha \in [0,1)$, together with $f \in \myclass[a]$.
In this short part, we study the relative positions of the curves of $f$ and $\falpha$. Note that this property will not be used to prove the other results of the paper.

 Recall from Proposition~\ref{proposition:falpha_x+mydelta} that for any $x \leq \suppfalpha$ we have:
\[
\falpha(x) = \Falpha{\alpha}\bigl(\tau(x)\bigr) + \alpha x,
\]
where $\tau : (-\infty,\suppfalpha] \to (-\infty,\xalphap]$ is the inverse function of $x \mapsto x + \mydelta$.

\begin{lemma}
\label{lemma:phi_circ_tau}
There exists a unique $\coinc \in \bigl(\xalpham,\suppfalpha\bigr)$ such that $\comp\circ \tau(\coinc) = \coinc$. Moreover, for $u \in \bigl(\xalpham,\suppfalpha\bigr)$ we have $\comp\circ\tau(u) < u \iff u > \coinc$.
\end{lemma}

\begin{proof}
Recall from Lemma~\ref{lemma:phi_precis} that the continuous function $\comp = \Falpha{\alpha}^{-1}\circ\Falpha{\alpha}$ is decreasing on $\bigl(\xalpham,\xalphap\bigr)$ and from Lemma~\ref{lemma:delta_increasing} that $\tau$ is increasing on $(-\infty,\suppfalpha]$ with $\tau(\suppfalpha) = \xalphap$ and $\tau(\xalpham) = \xalpham$. We deduce that $\comp \circ \tau$ is decreasing on $\bigl(\xalpham,\suppfalpha\bigr)$.  Moreover,  we have:
\[
\comp\circ\tau\bigl(\xalpham\bigr) = \comp\bigl(\xalpham\bigr) > \xalphap > \xalpham,
\]
since $\comp$ is decreasing, and, recalling~\eqref{equation:suppfalpha_ineq},
 \[
 \comp\circ\tau\bigl(\suppfalpha\bigr) = \comp\bigl(\xalphap\bigr) = \xalphap < \suppfalpha.
 \]
By the intermediate value theorem, we thus obtain the existence of the point $\coinc$ of the statement. Its uniqueness and the remaining part of the statement follow from the above fact that $u \mapsto \comp\circ \tau(u) - u$ is decreasing.
\end{proof}

\begin{proposition}
For any $x \in \bigl(\xalpham,a\bigr)$, we have $\falpha(x) < f(x)$ if and only if $x \in \bigl(\coinc,a\bigr)$. 
Moreover,  the point of abscissa $\coinc$ is the unique intersection point between the curves of $\falpha$ and $f$ in $\bigl(\xalpham,a\bigr)$.
\end{proposition}

\begin{proof}
First, note that by Proposition~\ref{proposition:falpha_x+mydelta}, we know that $\falpha$ and $f$ coincide (at least) on $\bigl(-\infty,\xalpham\bigr]\cup\bigl[a,+\infty\bigr)$, moreover $\falpha(x) < f(x)$ for all $x \in \bigl[\suppfalpha,a\bigr)$ by~\eqref{subequations:assumptions_f}. Now for any $x \in \bigl(\xalpham,\suppfalpha\bigr)$ we have:
\begin{align*}
\falpha(x) < f(x)
&\iff
\Falpha{\alpha}\bigl(\tau(x)\bigr) + \alpha x < f(x)
\\
&\iff \Falpha{\alpha}\bigl(\tau(x)\bigr) < f(x) - \alpha x
\\
&\iff
\Falpha{\alpha}\bigl(\tau(x)\bigr) < \Falpha{\alpha}(x).
\end{align*}
For all $y \in (\xalpham,\xalphap]$ we have $\mydelta[y] > 0$  by Lemma~\ref{lemma:delta_increasing} thus $y + \mydelta[y] > y$. Since $\tau$ is increasing we deduce that $x > \tau(x)$ for all  $x \in (\xalpham,\suppfalpha]$.
 Recalling~\eqref{equation:suppfalpha_ineq}, we deduce that for all $x \in (\xalpham,\xalphap]$ we have $\tau(x) < x \leq \xalphap$ thus the above inequality is not satisfied since $\Falpha{\alpha}$ is decreasing on $(-\infty,\xalphap]$ by Proposition~\ref{proposition:xalpha+}. Now for $x \in \bigl(\xalphap,\suppfalpha\bigr)$, by Definition~\ref{definition:phi} we have $\Falpha{\alpha}^{-1}\bigl(\Falpha{\alpha}(x)\bigr) = x$. Since $\Falpha{\alpha}^{-1}$ is increasing on $[\xalphap,+\infty)$ (by Proposition~\ref{proposition:xalpha+}) we deduce that for $x \in (\xalphap,\suppfalpha)$ we have:
\[
\Falpha{\alpha}\bigl(\tau(x)\bigr) < \Falpha{\alpha}(x)
\iff
\comp\circ\tau(x) < x.
\]
Lemma~\ref{lemma:phi_circ_tau} concludes the proof.
\end{proof}

\section{Shaking graphs}
\label{section:shaking_graphs}

The aim of this part is to determine how the quantity $\falpha$ is related with the shaking of $\gr(f)$. We first introduce the objects on which we will use the shaking operation introduced in~\textsection\ref{subsection:shakings}. Let $\partclass$ be the set of all continuous maps $g : \R \to \R$ such that:
\begin{subequations}
\begin{gather}
g(x) \geq |x|, \quad\text{for all }x \in \R,
\label{equation:g(x)>=|x|}
\\
g(x) = |x|, \quad\text{for all } |x| \gg 0,
\label{equation:g(x)=|x|}
\\
g(x_0) \neq |x_0|, \quad \text{for some } x_0 \in \R.
\end{gather}
\end{subequations}
Note that $\myclass[a] \subseteq \partclass$ for $a > 0$.

\begin{definition}
Let $g \in \partclass$ and let $\su[g], \sd[g] \in \R$ be defined by:
\begin{align*}
\su[g] &\coloneqq \inf\{ x \in \R : g(x) \neq |x|\} \in \R,
\\
\sd[g] &\coloneqq \sup\{x \in \R : g(x) \neq |x|\} \in \R.
\end{align*}
 We denote by $\gr(g)$ the part of the graph of $g$ on $[\su[g],\sd[g]]$ that is above the graph of the absolute value, that is:
\[
\gr(g) \coloneqq \{(x,y) \in [\su[g],\sd[g]] \times \R_{\geq 0} : |x| \leq y \leq g(x)\}.
\]
\end{definition}

For instance, we have pictured in Figure~\ref{figure:Y(Omega)} the set $\gr(\Omega)$.

\begin{figure}
\begin{center}
\begin{tikzpicture}[scale=2]
\draw[gray] (-2.1,0) -- (2.1,0);
\draw[gray] (0,-.5) -- (0,2.1);

\draw (-2.1,2.1) -- (0,0) -- (2.1,2.1);

\draw[domain=-2:2,color=red,smooth,thick] plot (\x,{2/pi*(\x*rad(asin(\x/2))+sqrt(4-\x*\x))}) node[left]{$\scriptstyle y = \Omega(x)$};

\fill[blue!20] (0,0) -- (-2,2) -- plot[domain=-2:2,smooth] (\x,{2/pi*(\x*rad(asin(\x/2))+sqrt(4-\x*\x))})  -- cycle;
\end{tikzpicture}
\end{center}
\caption{The set $\gr(\Omega)$ (in blue)}
\label{figure:Y(Omega)}
\end{figure}

\begin{lemma}
\label{lemma:inclusion_gr}
Let $g,h \in \partclass$. We have $g(x) \leq h(x)$ for all $x \in \R$ if and only if $\gr(g) \subseteq \gr(h)$.
\end{lemma}

\begin{proof}
Assume that $g(x) \leq h(x)$ for all $x \in \R$.
By~\eqref{equation:g(x)>=|x|} we have $|x| \leq g(x) \leq h(x)$ for all $x \in \R$, thus if $g(x) \neq |x|$ then $h(x) \neq |x|$. This proves that $\su[g] \geq \su[h]$ and $\sd[g] \leq \sd[h]$. The result follows.

Conversely, assume that $\gr(g) \subseteq \gr(h)$.  Note that by~\eqref{equation:g(x)>=|x|}, it suffices to prove that $g(x) \leq h(x)$ for all $x \in [\su[g],\sd[g]]$. If $x \in [\su[g],\sd[g]]$ then $\bigl(x,g(x)\bigr) \in \gr(g)$ by~\eqref{equation:g(x)>=|x|} thus $\bigl(x,g(x)\bigr) \in \gr(h)$ thus $g(x) \leq h(x)$. This concludes the proof.
\end{proof}

\subsection{A stability criterion}

 Recall that $\biss$ denotes the line with equation $y = -x$.

\begin{lemma}
\label{lemma:derivative_bounded_alpha}
Let $L$ be a line of slope $\alpha \in \R$.
Let $g \in \partclass$ and write $a \coloneqq \su[g]$ and $b \coloneqq \sd[g]$. Assume that $g$ differentiable on $(\su,\sd)$ with  $g'(x) < \alpha$ for all $x \in (\su,\sd)$. Then one and only one of the followings holds.
\begin{enumerate}
\item
The line $L$ is strictly above the curve of $g$ on $[\su,\sd]$.
\item
The line $L$ is strictly below the curve of $g$ on $[\su,\sd]$.
\item
\label{item:line_crosses_two_points}
There is a unique $x \in [\su,\sd]$ such that $\bigl(x,g(x)\bigr) \in L$, moreover the line $L$ is strictly below (resp. above) the curve of $g$ on $[\su,x)$ (resp. $(x,\sd]$). Besides, the point $\bigl(x,g(x)\bigr)$ is the unique point $M \in L$ at the right of $\biss$ such that $d_\alpha(M,\biss) = \bigl|L \cap \gr(g)\bigr|$.
\end{enumerate}
In particular, the intersection $L \cap \gr(g)$ is a segment.
\end{lemma}

\begin{proof}
The cases are obtained via differentiation. Note that for $(iii)$, the point $M$ is clearly unique and by the first part of $(iii)$ we know that the point $\bigl(x,g(x)\bigr)$ satisfies the assumptions. To obtain the last assertion, with $K \coloneqq L \cap \gr(g)$ we have respectively:
\begin{enumerate}
\item $K$ is empty,
\item $K$ is the intersection between $L$ and $[\su,\sd] \times \bigl\{(u,v) \in [\su,\sd] \times \R_{\geq 0} : |u| \leq v\bigr\}$,
\item $K$ is the intersection between $L$ and $[\su,x] \times \bigl\{(u,v) \in [\su,\sd] \times \R_{\geq 0} : |u| \leq v\bigr\}$,
\end{enumerate}
thus $K$ is a segment indeed, by convexity of the absolute value.
\end{proof}

The next result gives a simple condition on $g \in \partclass$ for $\gr(g)$ to be stable under the shaking operation.

\begin{lemma}
\label{lemma:graphe_stable_Sh}
Let $g \in \partclass$  be continuously differentiable on $(\su[g],\sd[g])$. Let $\alpha \in [0,1)$ and assume that for all $x \in (\su[g],\sd[g])$ we have $g'(x) < \alpha$. Then $\gr(g)$ is stable under the shaking operation $\Sh$.
\end{lemma}

\begin{proof}
First, since the set $\gr(g)$ is compact we can take its image by $\Sh$ indeed. Write $\su\coloneqq \su[g]$ and $\sd \coloneqq \sd[g]$.

By Lemma~\ref{lemma:derivative_bounded_alpha}, it suffices to prove that if $L$ is a line of slope $\alpha < 1$ and if $x \in \R$ is the unique  real number such that $(x,-x) \in L$ then $L \cap \gr(g) \neq \emptyset$ if and only if $(x,-x) \in \gr(g)$. Obviously this condition is necessary, thus assume that $L \cap \gr(g) \neq \emptyset$ and let us prove that $(x,-x) \in \gr(g)$. In other words, we have to prove that $x \in [\su,\sd]$ and that $x \leq 0$ (\text{i.e.} $|x| \leq -x$). We proceed by contradiction.

 If $x > 0$ then since $\alpha < 1$ and by~\eqref{equation:g(x)=|x|} we have $L \cap \gr(g) = \emptyset$. Similarly, if $x > \sd$ then since $\alpha > -1$  we have $L \cap \gr(g) = \emptyset$ again.

Note that we  have $\su < 0$. Indeed, if $\su \geq 0$ then we have $0 \leq \su \leq \sd$, thus $g(\su) = \su$ and $g(\sd) = \sd$; by the mean value theorem, this implies that there exists $c \in (\su,\sd)$ with $g'(c) = 1$, but this is a contradiction since $g'(c) < \alpha < 1$ by assumption. Now if $x < \su$, then since $\alpha \geq 0$ the point $(\su,|\su|) = (\su,g(\su))$ is strictly below $L$. Thus, by Lemma~\ref{lemma:derivative_bounded_alpha} the whole curve of $g$ on $[\su,\sd]$ is below $L$. As a result, we have $L \cap \gr (g)= \emptyset$, which is a contradiction. Hence, we have proved both $x \leq 0$ and $a \leq x \leq b$, as announced.
\end{proof}

\subsection{Shaking graphs}

We prove here the first main result of the paper, that is, that the graph of the shaking is given by the shaking of the graph (Theorem~\ref{theorem:sh_gr_f}). Let $\alpha \in [0,1)$, $a \in \mathbb{R}_{> 0}$ and $f \in \myclass[a]$.

\begin{proposition}
\label{proposition:gr_shf_stable}
\begin{itemize}
\item  We can apply Lemma~\ref{lemma:derivative_bounded_alpha} for $\falpha$.
\item The set $\gr\bigl(\falpha\bigr)$ is stable under the shaking operation $\Sh$.
\end{itemize}
\end{proposition}

\begin{proof}
First, note that from~\eqref{equation:f(s)=|s|}, \eqref{equation:f(s)>|s|} and Proposition~\ref{proposition:falpha_x+mydelta}, we know that $\su[\falpha] = -a$ and $\sd[\falpha] = \suppfalpha$. Second, since $f$ is continuously differentiable, by Definition~\ref{definition:falpha} we know that $\falpha$ is continuously differentiable on $(-a,\suppfalpha)$. Hence, by Lemma~\ref{lemma:graphe_stable_Sh} we know that the second item follows from the first one. We conclude the proof since $\falpha$ satisfies the assumptions of Lemma~\ref{lemma:derivative_bounded_alpha}  by Lemma~\ref{lemma:derivative_falpha}.
\end{proof}

Recall from Definition~\ref{definition:xprime} the definition of $\Lx{x}$ and $\xalphapp$.
We now give the analogue of Lemma~\ref{lemma:intersection_f_L} for $\falpha$. 

\begin{lemma}
\label{lemma:intersection_shf_L}
Let  $x \in \R$. On $[-a,a]$ we have:
\begin{itemize}
\item if $x \leq -a$ then $\Lx{x}$ is above the curve of $\falpha$,
\item if $-a < x \leq \xalphapp$ then $\Lx{x}$ is below and then above the curve of $\falpha$,
\item if $x > \xalphapp$ then $\Lx{x}$ is below the curve of $\falpha$.
\end{itemize}
\end{lemma}

\begin{proof}
The three points follow from Propositions~\ref{proposition:falpha_x+mydelta} and~\ref{proposition:gr_shf_stable}, noting that by Lemma~\ref{lemma:geometric_interpretation_delta} and~\eqref{equation:suppfalpha} the points of the curve of $\falpha$ with abscissa $\xalphap$ and $\suppfalpha$ are both on $\Lx{\xalphapp}$.
\end{proof}

We can now give  the main result of this section.  Note that Figure~\ref{figure:Omegae_alpha0} of the introduction illustrates Theorem~\ref{theorem:sh_gr_f} for $\Omegae[2] = \Sh[0](\Omega)$.

\begin{theorem}
\label{theorem:sh_gr_f}
We have:
\[
\Sh\bigl(\gr(f)\bigr) = \gr\bigl(\falpha[\alpha]\bigr).
\]
\end{theorem}

\begin{proof}
 Write $\Lx[]{x}$ instead of $\Lx{x}$ for simplicity. 
It suffices to prove that for any $x \in \R$ we have  $\Lx[]{x} \cap \Sh\bigl(\gr(f)\bigr) = \Lx[]{x} \cap \gr\bigl(\Sh(f)\bigr)$. By Lemmas~\ref{lemma:intersection_f_L} and~\ref{lemma:intersection_shf_L}, it suffices to consider the case $x \in (\xalphamp,\xalphapp)$.

By Lemmas~\ref{lemma:phi} and~\ref{lemma:intersection_f_L}, we know that there exists $s < \xalphap$ such that $\Lx[]{x}$ intersects the curve of $f$ at the points of abscissa $s$ and $\comp(s) > s$, with $\Lx[]{x}$ being above the curve of $f$ exactly between $s$ and $\comp(s)$.
Hence, by Lemma~\ref{lemma:phi}\ref{item:geometric_interpretation_phi} we thus know that the connected components of $\Lx[]{x} \cap \gr(f)$ are:
\begin{align*}
L^x_- &\coloneqq \Lx[]{x} \cap \gr(f) \cap \bigl([-a,s]\times \R\bigr),
\\
L^x_+ &\coloneqq \Lx[]{x} \cap \gr(f) \cap \bigl([\comp(s),a] \times \R),
\end{align*}
and:
\begin{itemize}
\item the segment $S^x \coloneqq \Sh\bigl(\Lx[]{x} \cap \gr(f)\bigr)$ is obtained by appending $L^x_+$ at the end of $L^x_-$ (note that $(x,-x) \in L^x_-$ by Lemma~\ref{lemma:intersection_f_L}),
\item the segment $L^x_-$ ends at the point of $\Lx[]{x}$ with abscissa $s$.
\end{itemize}
By Lemma~\ref{lemma:geometric_interpretation_delta}, we know that $L^x_+$ has abscissa length $\mydelta[s]$. We deduce that $S^x$ ends at the point $P$ of $\Lx[]{x}$ with abscissa $s+\mydelta[s]$. Recalling that $\Lx[]{x}$ has slope $\alpha$, we deduce that $P$ has ordinate $t$ satisfying $\alpha = \frac{t-f(s)}{s+\mydelta[s]-s}$ so that $t = f(s) + \alpha\mydelta[s]$.

Now by Proposition~\ref{proposition:gr_shf_stable}, we know that $\gr\bigl(\falpha\bigr)$ is stable under $\Sh$. In particular, the set $T^x \coloneqq \Lx[]{x} \cap \gr\bigl(\falpha\bigr)$ is a segment. By  Lemma~\ref{lemma:derivative_bounded_alpha} and Proposition~\ref{proposition:gr_shf_stable}, we know that there is a unique $\hat{s} \in [-a,\suppfalpha]$ such that $\bigl(\hat{s},\falpha(\hat{s})\bigr) \in \Lx[]{x}$, and this $\hat{s}$ is the right extremity of $T^x$. By Proposition~\ref{proposition:falpha_x+mydelta} we have $\falpha(s+\mydelta[s]) = f(s) + \alpha\mydelta[s] = t$ and we have seen that $P = (s+\mydelta[s],t) \in \Lx[]{x}$, thus  $\hat{s} = s+\mydelta[s]$ by unicity. Hence, the two segments $S^x$ and $T^x$ start and end at the same point thus are equal.
\end{proof}

\subsection{Convexity}
\label{subsection:convexity}

Let $a \in \R_{> 0}$ and  $\alpha \in [0,1)$, together with $f \in \myclass[a]$. We want to prove that the convexity of $f$ ensures that $\falpha$ is convex as well. Note that this result will not be used in the remaining part of the paper.  The key point is the following result.

\begin{lemma}
\label{lemma:shaking_convex}
Let $K \subseteq \R^2$ be compact. If $K$ is convex then $\Sh(K)$ is also convex.
\end{lemma}

\begin{proof}
See, for instance,~\cite[Proposition 9.1]{gruber} or~\cite[Proposition 7.1.7]{krantz} (see also \cite[Lemma 1.1.(ii)]{shaking_compact}). The idea is to notice that for $M,N \in \Sh(K)$, the convex hull of
\[
\bigl[K \cap (M+L_\alpha)\bigr] \cup \bigl[K \cap (N + L_\alpha)\bigr]
\]
(where $L_\alpha$ is the line with equation $y = \alpha x$) is a trapezoid, whose shaking remains convex, as illustrated in Figure~\ref{figure:Sh_example}.
\end{proof}

We cannot directly apply Lemma~\ref{lemma:shaking_convex} from Theorem~\ref{theorem:sh_gr_f} since $\gr(f)$ is not convex. Instead, for $g \in \partclass$  we consider the following restricted epigraph, recalling the notation $\Lx[]{x} = \Lx{x}$ from Definition~\ref{definition:xprime}:
\[
\epi(g) \coloneqq \bigl\{(x,y) \in [-a,\comp(-a)] \times \R_{\geq 0} : g(x) \leq y \text{ and } (x,y) \text{ is below } \Lx[]{-a} \bigr\}.
\]
 Note that by Lemma~\ref{lemma:phi_precis} we have $\comp(-a) = \frac{1+\alpha}{1-\alpha}a \geq a$ (which does not depend on $f$). Moreover, by Lemma~\ref{lemma:phi} we know that $\comp(-a)$ is exactly the intersection point between $\Lx[]{-a}$ and the line $\bissp$ of equation $y = x$. An example is given at Figure~\ref{figure:epi(Omega)}.

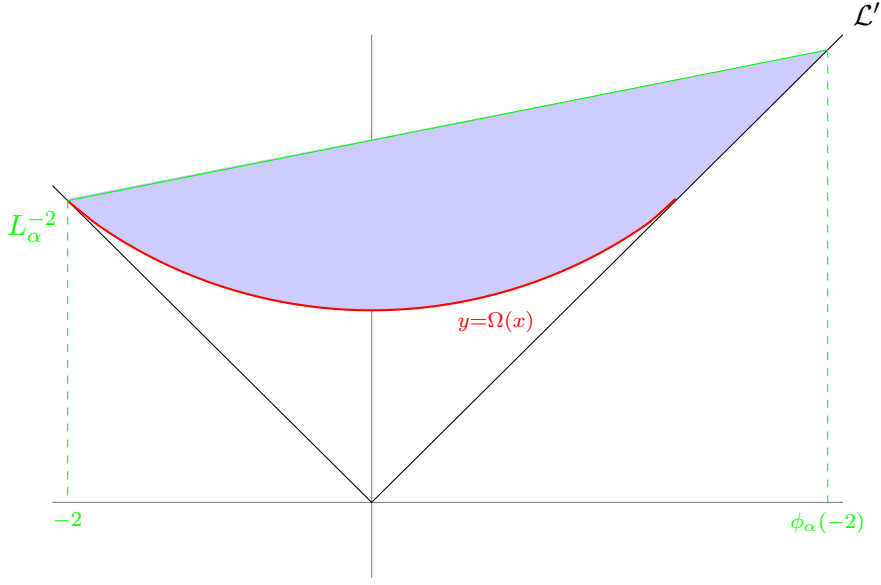
\begin{figure}
\begin{center}
\begin{tikzpicture}[scale=2]
\draw[gray] (-2.1,0) -- (3.1,0);
\draw[gray] (0,-.5) -- (0,3.1);

\fill[blue!20] (3,3) -- (2,2) -- plot[domain=2:-2,smooth] (\x,{2/pi*(\x*rad(asin(\x/2))+sqrt(4-\x*\x))}) -- cycle;

\draw (-2.1,2.1) -- (0,0) -- (3.1,3.1) node[above right]{$\bissp$};

\draw[domain=-2:2,color=red,smooth,thick] plot (\x,{2/pi*(\x*rad(asin(\x/2))+sqrt(4-\x*\x))});
\draw[red] (.5,1.2) node[right] {$\scriptstyle y = \Omega(x)$} ;

{
\color{green}
\draw (-2,2) node[below left]{$\Lx{-2}$} -- (3,3);

\draw[dashed] (-2,2) -- (-2,0)node[below]{$\scriptstyle -2$};
\draw[dashed] (3,3) -- (3,0) node[below]{$\scriptstyle \comp(-2)$};
}
\end{tikzpicture}
\end{center}
\caption{The set $\epi(\Omega)$ (in blue) with $\alpha = \frac{1}{5}$}
\label{figure:epi(Omega)}
\end{figure}

\begin{remark}
\label{remark:epi}
If $[\su[g],\sd[g]] \subseteq [-a,a]$ and if the curve of $g \in \partclass$ is below $\Lx[]{-a}$ on $[-a,a]$ then:
\begin{subequations}
\label{subequations:epi_gr}
\begin{itemize}
\item   we do not loose  information by considering $\epi(g)$ instead of its genuine epigraph, in particular if $\epi(g)$ is convex then $g$ is convex,
\item the intersection 
\begin{equation}
\label{equation:epi_cap_gr}
\epi(g) \cap \gr(g) = \bigl\{(x,g(x)) : x \in [\su[g],\sd[g]]\bigr\},
\end{equation}
is the curve of $g$ on $[\su[g],\sd[g]]$,
\item we have:
\begin{equation}
\label{equation:epi_cup_gr}
\epi(g) \cup \gr(g) = \{(x,y) \in [-a,\comp(-a)]\times\R_{\geq 0} : |x| \leq y \text{ and } (x,y) \text{ is below } \Lx[]{-a}\}.
\end{equation}
\end{itemize}
\end{subequations}
\end{remark}

The next result will be used without further notice.

\begin{lemma}
\label{lemma:f_falpha_satisfy}
Both $f$ and $\falpha$ satisfy the condition of Remark~\ref{remark:epi}.
\end{lemma}

\begin{proof}
The statement is clear for $f$ by~\eqref{subequations:assumptions_f}. For $\falpha$, note that by Proposition~\ref{proposition:falpha_x+mydelta} we have $\su[\falpha] = -a$ and $\sd[\falpha] = \suppfalpha \in [0,a]$ (by Corollary~\ref{corollary:support_fe}). We conclude the proof using Lemma~\ref{lemma:intersection_shf_L}.
\end{proof}

We now consider the shaking $\Shp$ with direction $\alpha$ against the line $\bissp$ with equation $y = x$  but with respect to the unit vector $-v_\alpha$, which is \emph{negatively} collinear to $(1,\alpha)^\top$.  The analogue of Lemma~\ref{lemma:shaking_distance} is the following (with the notation of Lemma~\ref{lemma:shaking_distance}).

\begin{lemma}
\label{lemma:shakingp_distance}
Let $K \subseteq \R^2$ be compact. Let $M \in \R^2$ be on the \emph{left} of the line $\bissp$. Then:
\[
M \in \Shp(K) \iff d_\alpha(M,\bissp) \leq \bigl|\LM{M} \cap K\bigr| \text{ and } \LM{M} \cap K \neq \emptyset.
\]
\end{lemma}

An example of $\Shp$ is given in Figure~\ref{figure:Shp}. We will need to rephrase the interpretation of $\Shp$ of Lemma~\ref{lemma:shakingp_distance} in terms of $\biss$. 

\begin{figure}
\begin{center}
\begin{tikzpicture}[scale=1.5]

\draw[thick] (0,-.5) -- (0,3.5);
\draw[thick] (-3.5,0) -- (3.5,0);
\draw[gray, very thin] (-3.5,-.5) grid[step=.5] (3.5,3.5);

\draw[very thick] (-3.5,3.5) -- (0,0) -- (3.5,3.5);

\fill[red] (-.5,2) rectangle (.5,3);
\draw[red] (0,3.1) node[above,inner sep=1pt,fill=white]{$K$};

\fill[blue] (-3,3) -- node[above, midway,fill=white,inner sep=1pt,shift={(.2,.1)}] {$\Sh(K)$} (-2,3) -- (-1,2) -- (-2,2) -- cycle;
\draw[very thick,green,->] (-.625,2.75) -- (-1.625,2.75);
\draw[very thick,green,->] (-.625,2.25) -- (-1.125,2.25);

\fill[violet] (3,3) -- node[above, pos=.6,fill=white,inner sep=1pt,shift={(.2,.1)}] {$\Shp(K)$} (2,3) -- (1,2) -- (2,2) -- cycle;
\draw[very thick,greensage,->] (.625,2.75) -- (1.625,2.75);
\draw[very thick,greensage,->] (.625,2.25) -- (1.125,2.25);

\draw (-1.25,1) node[below left,fill=white,inner sep = 1pt]{$\biss$};
\draw (1.25,1) node[below right, fill=white, inner sep = 1pt]{$\bissp$};

\end{tikzpicture}
\end{center}
\caption{Shakings $\Sh$ and $\Shp$ with $\alpha = 0$}
\label{figure:Shp}
\end{figure}

\begin{lemma}
\label{lemma:shakingp_distance_reverse}
Let $g \in \partclass$ be satisfying the assumption of Remark~\ref{remark:epi}. Let $M = (x,y) \in \R^2$ be below $\Lx[]{-a}$ such that $|x| \leq y$. Then:
\[
M \in \Shp\bigl(\epi(g)\bigr) \iff d_\alpha(M,\biss) \geq \bigl|\LM{M} \cap \gr(g)\bigr| \text{ and } \LM{M} \cap \epi(g) \neq \emptyset.
\]
\end{lemma}

\begin{proof}
By assumption, the point $M$ is on the left of $\bissp$ thus:
\[
M \in \Shp\bigl(\epi(g)\bigr) \iff d_\alpha(M,\bissp) \leq \bigl|\LM{M} \cap \epi(g)\bigr| \text{ and } \LM{M} \cap \epi(g) \neq \emptyset.
\]
Now the subset $D \coloneqq \LM{M} \cap \{(u,v) \in \R^2 : |u| \leq v\}$ is a segment of $\R^2$ since $\alpha \neq \pm 1$, with left (resp. right) extremity in $\biss$ (resp. $\biss'$). We have $M \in D$ thus:
\[
d_\alpha(M,\bissp) = |D| - d_\alpha(M,\biss).
\]
Moreover, by~\eqref{subequations:epi_gr}, since $M$ is below $\Lx[]{-a}$ and since the curve of $g$ has zero Lebesgue measure, we have:
\[
\bigl|\LM{M} \cap \epi(g)\bigr| = |D| - \bigl|\LM{M} \cap \gr(g)\bigr|.
\]
The result follows.
\end{proof}

\begin{proposition}
\label{proposition:epi(sh)}
We have:
\[
\epi\bigl(\falpha[\alpha]\bigr) = \Shp\bigl(\epi(f)\bigr).
\]
\end{proposition}

\begin{proof}
It suffices to prove that for all $x \in \R$ we have $\Lx[]{x} \cap \epi\bigl(\falpha\bigr) = \Lx[]{x} \cap \Shp\bigl(\epi(f)\bigr)$. By Lemmas~\ref{lemma:intersection_f_L} and~\ref{lemma:intersection_shf_L}, it suffices to consider $x \in (\xalphamp,\xalphapp)$. Let $M_0 = (u_0,v_0)$ be the intersection point between $\Lx[]{x}$ and the curve of $\falpha$ and let $M = (u,v) \in \Lx[]{x}$. In particular, note that $\Lx[]{x}$ is the line of slope $\alpha$ containing $M$. 
We have:
\begin{align*}
M \in \epi\bigl(\falpha[\alpha]\bigr)
&\iff
v \geq \falpha[\alpha](u)
\\
&\iff
v = \falpha[\alpha](u) \text{ or } v > \falpha[\alpha](u)
\\
&\iff
v = \falpha[\alpha](u) \text{ or } M \notin \gr\bigl(\falpha[\alpha]\bigr).
\end{align*}
Thus, by Theorem~\ref{theorem:sh_gr_f} and Lemma~\ref{lemma:shaking_distance}  we have:
\begin{align*}
M \in \epi\bigl(\falpha[\alpha]\bigr)
&\iff
v = \falpha[\alpha](u) \text{ or } M \notin \Sh\bigl(\gr(f)\bigr)
\\
&\iff
v = \falpha[\alpha](u) \text{ or } d_\alpha(M,\biss) > \bigl|\Lx[]{x} \cap \gr(f)\bigr| \text{ or } \Lx[]{x} \cap \gr(f) = \emptyset.
\end{align*}
By Lemma~\ref{lemma:intersection_f_L}, since $x \in (\xalphamp,\xalphapp)$ we have $\Lx[]{x} \cap \gr(f) \neq \emptyset$ thus we deduce that:
\[
M \in \epi\bigl(\falpha[\alpha]\bigr) \iff v = \falpha[\alpha](u) \text{ or } d_\alpha(M,\biss) > \bigl|\Lx[]{x} \cap \gr(f)\bigr|.
\]
Now by definition of $\Sh$ and by  Theorem~\ref{theorem:sh_gr_f} we have:
\[
\bigl|\Lx[]{x}\cap \gr(f)\bigr|
=
\bigl|\Lx[]{x} \cap \Sh\bigl(\gr(f)\bigr)\bigr|
=
\bigl| \Lx[]{x} \cap \gr\bigl(\falpha[\alpha]\bigr)\bigr|.
\]
Using this equality twice and using Lemma~\ref{lemma:derivative_bounded_alpha}\ref{item:line_crosses_two_points} (that we can use by the existence of $M_0$), by Proposition~\ref{proposition:gr_shf_stable} we obtain:
\begin{align*}
M \in \epi\bigl(\falpha[\alpha]\bigr)
&\iff 
v = \falpha[\alpha](u) \text{ or } d_\alpha(M,\biss) > \bigl|\Lx[]{x} \cap \gr\bigl(\falpha[\alpha]\bigr)\bigr|
\\
&\iff 
d_\alpha(M,\biss) \geq \bigl|\Lx[]{x} \cap \gr\bigl(\falpha[\alpha]\bigr)\bigr|
\\
&\iff 
d_\alpha(M,\biss) \geq \bigl|\Lx[]{x} \cap \gr(f)\bigr|.
\end{align*}
Again by Lemma~\ref{lemma:intersection_f_L}, since $x \in (\xalphamp,\xalphapp)$ we have $\Lx[]{x} \cap \epi(f) \neq \emptyset$ thus we obtain, by Lemma~\ref{lemma:shakingp_distance_reverse}:
\begin{align*}
M \in \epi\bigl(\falpha[\alpha]\bigr)
&\iff 
d_\alpha(M,\biss) \geq \bigl|\Lx[]{x} \cap \gr(f)\bigr| \text{ and } \Lx[]{x} \cap \epi(f) \neq \emptyset
\\
&\iff M \in \Shp\bigl(\epi(f)\bigr).
\end{align*}
This concludes the proof.
\end{proof}

\begin{corollary}
\label{corollary:falpha_convex}
The function $\falpha[\alpha]$ is convex.
\end{corollary}

\begin{proof}
By Lemma~\ref{lemma:shaking_convex}, which can also be stated with $\Shp$, using Remark~\ref{remark:epi} and Proposition~\ref{proposition:epi(sh)}   we obtain that $\epi\bigl(\falpha\bigr)$ is convex thus $\falpha$ is convex on $[-a,\comp(-a)]$. By Proposition~\ref{proposition:falpha_x+mydelta}, we deduce that $\falpha$ is convex on $\R$.
\end{proof}

Recall from Definition~\ref{definition:Omegae} that $\Omegae = \Sh[1-2e^{-1}](\Omega)$. Example~\ref{example:Omega_convex} gives the following particular case.

\begin{corollary}
\label{corollary:Omegae_convex}
The shape $\Omegae$ is convex for all $e \geq 2$.
\end{corollary}

\section{Shaking partitions}
\label{section:shaking_partitions}

Let $\lambda$ be a partition and let $e \geq 2$. Recall from Definition~\ref{definition:sh} that $\sh = \Sh$ with $\alpha = \alpha_e\coloneqq 1-2e^{-1}$. The following observation is the starting point of this section.

\begin{lemma}
\label{lemma:sh_gr_reg}
We have $\sh\bigl(\gr(\omega_\lambda)\bigr) = \sh\bigl(\gr(\omega_{\reg(\lambda)})\bigr)$ and $\sh\bigl(\gr(\tomega_\lambda)\bigr) = \sh\bigl(\gr(\tomega_{\reg(\lambda)})\bigr)$.
\end{lemma} 

\begin{proof}
Let $L$ be a line of slope $\alpha_e$. By Proposition~\ref{proposition:reg_direction}, we know that, in the Russian convention, the line $L$ crosses the same number of boxes in the Young diagram of $\lambda$ and in the Young diagram of $\reg(\lambda)$. This implies that $|L \cap \gr(\omega_\lambda)| = |L \cap \gr(\omega_{\reg(\lambda)})|$, whence the first equality. The proof for the second one is the same.
\end{proof}

Hence, to study the shaking operations on partitions it suffices to study the shaking operations on regular partitions. In fact, the shaking of $\gr(\tomega_{\mu})$ for $\mu$ an $e$-regular partition is a bit delicate to determine. Instead, we will bound the latter graph by two close graphs that are stable under the shaking operation.

\subsection{Corners}

Let $\lambda$ be a partition.
We say that $c \in \mathbb{Z}$ is an \emph{outer corner} (resp. \emph{inner corner}) of $\lambda$ if $\omega'_\lambda|_{(c-1,c)} = 1$ (resp. $=-1$) and $\omega'_\lambda|_{(c,c+1)} = -1$ (resp. $=1$).   Note that $\lambda$ has always at least one inner corner, and has exactly one inner corner if and only if $\lambda$ is empty.

\begin{figure}[h]
\begin{center}
\begin{tikzpicture}[scale=.6]
\draw[-stealth] (-8,0) -- (8,0) node[below]{$x$};
\draw[-stealth] (0,-1) -- (0,9) node[right]{$y$};
\foreach \i in {-7,-6,-5,-4,-3,-2,-1,1,2,3,4,5,6,7}
	{\draw (\i,-.1) --++ (0,.2);}
\foreach \j in {1,2,...,8}
	\draw (-.1,\j) --++ (.2,0);

\draw[very thick] (-7,7) -- (-4,4) -- (-2,6) -- (0,4) -- (1,5) -- (2,4) -- (3,5) -- (4,4) -- (7,7);
\draw[dotted, very thick] (-7.7,7.7) -- (-7,7);
\draw[dotted, very thick] (7,7) -- (7.7,7.7) node[below right]{$y=\omega_\lambda(x)$};

\draw[dashed] (-4,4) -- (0,0) -- (4,4);
\draw[dashed] (-3,5) -- (0,2) -- (2,4) --++ (1,-1);
\draw[dashed] (-3,3) --++ (2,2);
\draw[dashed] (-2,2) --++ (2,2) --++ (2,-2);
\draw[dashed] (-1,1) --++ (1,1) --++ (1,-1);

\begin{scope}[densely dotted, thick]

\color{red}
\draw (-2,6) -- (-2,0) node[below]{$c_1$};
\fill (-2,0) circle (\sizecorner);

\foreach \i in {1,3}
	{
	\draw (\i,5) -- (\i,0);
	\fill (\i,0) circle (\sizecorner);
	}
\draw (1,0) node[below]{$c_2$};
\draw (3,0) node[below]{$c_3$};

\color{blue}
\foreach \i in {-4,2,4}
	{
	\draw (\i,4) -- (\i,0);
	\fill (\i,0) circle (\sizecorner);
	}
	
\draw (-4,0) node[below]{$i_1$};
\draw (0,4) -- (0,0);
\fill (0,0) circle (\sizecorner);
\draw (-.5,0) node[below]{$i_2$};
\draw (2,0) node[below]{$i_3$};
\draw (4,0) node[below]{$i_4$};
\end{scope}
\end{tikzpicture}
\end{center}
\caption{Outer corners (in red) and inner corners (in blue) for $\lambda = (4,4,2,1)$}
\label{figure:corners}
\end{figure}
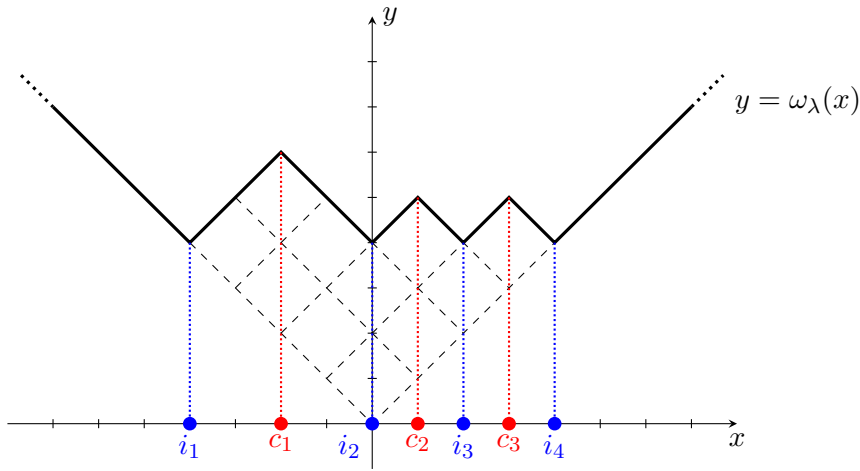

\begin{lemma}
\label{lemma:corners_intertwins}
Let $\lambda$ be a partition and let $\{c_1 < \dots < c_r\}$ (resp. $\{i_1 < \dots < i_s\}$) be its set of outer (resp. inner) corners. Then $s = r+1$ and:
\begin{enumerate}
\item for all $m \in \{1,\dots,r\}$ we have $i_m < c_m < i_{m+1}$ and $\omega_\lambda(c_m) = \omega_\lambda(i_m) + c_m - i_m$,
\item if $\lambda$ is $e$-regular then $i_m \geq c_m - e + 1$ for all $m \in \{1,\dots,r\}$.
\end{enumerate}
\end{lemma}

\begin{proof}
The first item is standard and  follows from the definition of inner and outer corners and from the fact that $\omega_\lambda(x) = |x|$ for $x \gg 0$.  For the second one, since $i_m$ is the inner corner preceding $c_m$ we have $\omega'_\lambda|_{(i_m,c_m)} = 1$. By $e$-regularity, we thus have $c_m - i_m \leq e-1$, whence the result.
\end{proof}

We illustrate Lemma~\ref{lemma:corners_intertwins} in Figure~\ref{figure:corners}. Note that if $\lambda$ is a partition and $\{i_1 < \dots < i_{r+1}\}$ is its set of inner corners, then by definition of a Young diagram we have:
\begin{subequations}
\label{subequations:omega_abs}
\begin{align}
\omega(s) &= |s|, &&\text{for all } s \notin (i_1,i_{r+1}),
\\
\omega(s) &> |s|, &&\text{for all } s \in (i_1,i_{r+1}).
\end{align}
\end{subequations}
In the sequel, fix $e \geq 2$ and let $L_e$ be the line of $\R^2$ with slope $\dire = 1-2e^{-1}$ and containing the origin.

\subsection{Outer flattening}
\label{subsection:outer_regularisation}

Let $\lambda$ be an $e$-regular partition. Let $\{c_1 < \dots < c_r\}$ (resp. $\{i_1 < \dots < i_{r+1}\}$)  be the set of outer (resp. inner) corners of $\lambda$. We define $\applatitk{0}{\lambda} \coloneqq \omega_\lambda$ and, if $r \geq 1$, assuming by induction on $k \in \{1,\dots,r\}$ that we have constructed the piecewise linear function $\applatitk{k-1}{\lambda} : \R \to \R$ we construct the piecewise linear function $\applatitk{k}{\lambda} : \R \to \R$ as follows. Let $\myL{+}$ be the affine line $\bigl(c_k,\omega_\lambda(c_k)\bigr) + L_e$.

\begin{enumerate}
\item For $s \geq c_k$ then $\applatitk{k}{\lambda}(s) \coloneqq \applatitk{k-1}{\lambda}(s)$.
\item
\label{item:Lec_intersection}
At $s = c_k$, we follow the line $\myL{+}$ (in the negative direction) until we meet the curve of $\applatitk{k-1}{\lambda}$, at the point of abscissa $\hp$.
\item For $s \leq \hp$ then $\applatitk{k}{\lambda}(s) \coloneqq \applatitk{k-1}{\lambda}(s)$ again.
\end{enumerate}

An example of construction is given in Figure~\ref{figure:applatitk} (note that the last step will be illustrated later, in Figure~\ref{figure:applatit+}).

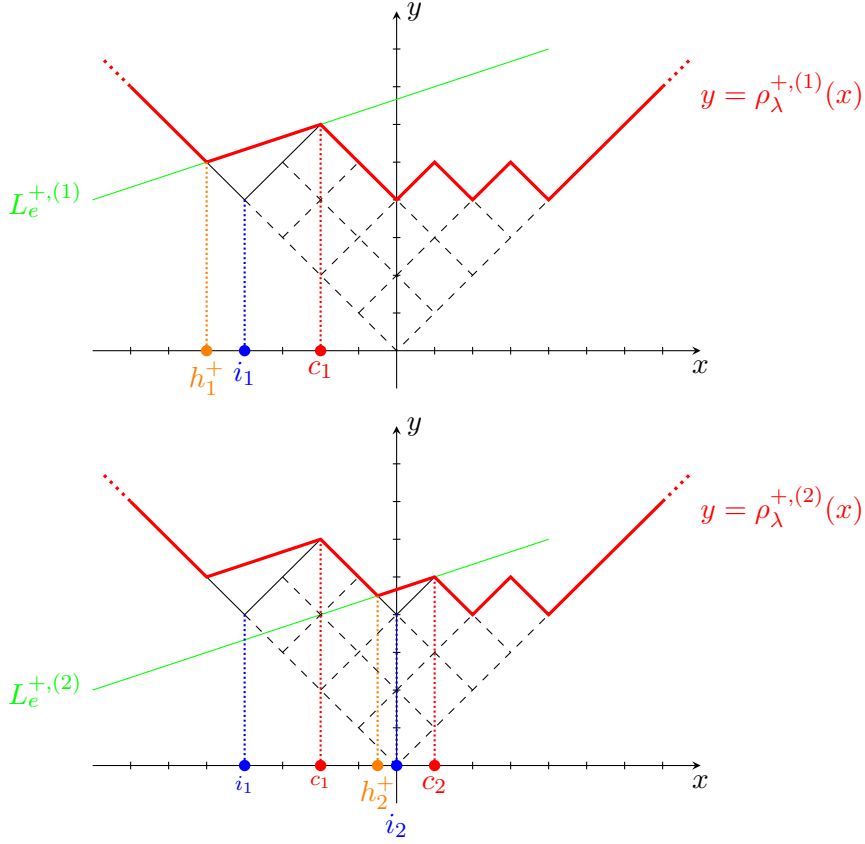
\begin{figure}[h]
\begin{center}
\begin{tikzpicture}[scale=.5]
\begin{scope}[shift={(0,11)}]
\draw[-stealth] (-8,0) -- (8,0) node[below]{$x$};
\draw[-stealth] (0,-1) -- (0,9) node[right]{$y$};
\foreach \i in {-7,-6,-5,-4,-3,-2,-1,1,2,3,4,5,6,7}
	{\draw (\i,-.1) --++ (0,.2);}
\foreach \j in {1,2,...,8}
	\draw (-.1,\j) --++ (.2,0);

\draw (-7,7) -- (-4,4) -- (-2,6) -- (0,4) -- (1,5) -- (2,4) -- (3,5) -- (4,4) -- (7,7);

\draw[green] (-8,4) node[left]{$\myLk{+}{1}$} -- (4,8);

\draw[very thick,red] (-7,7) -- (-5,5) -- (-2,6) -- (0,4) -- (1,5) -- (2,4) -- (3,5) -- (4,4) -- (7,7);
\draw[dotted, very thick, red] (-7.7,7.7) -- (-7,7);
\draw[dotted, very thick, red] (7,7) -- (7.7,7.7) node[below right]{$y = \applatitk{1}{\lambda}(x)$};

\draw[dashed] (-4,4) -- (0,0) -- (4,4);
\draw[dashed] (-3,5) -- (0,2) -- (2,4) --++ (1,-1);
\draw[dashed] (-3,3) --++ (2,2);
\draw[dashed] (-2,2) --++ (2,2) --++ (2,-2);
\draw[dashed] (-1,1) --++ (1,1) --++ (1,-1);

\begin{scope}[densely dotted, thick]

\color{red}
\draw (-2,6) -- (-2,0) node[below]{$ c_1$};
\fill (-2,0) circle (\sizecorner);


\color{blue}
\foreach \i in {-4}
	{
	\draw (\i,4) -- (\i,0);
	\fill (\i,0) circle (\sizecorner);
	}
	
\draw (-4,0) node[below]{$ i_1$};
\color{orange}
\draw (-5,5) -- (-5,0) node[below]{$\hp[1]$};
\fill (-5,0) circle (\sizecorner);


\end{scope}
\end{scope}

\draw[-stealth] (-8,0) -- (8,0) node[below]{$x$};
\draw[-stealth] (0,-1) -- (0,9) node[right]{$y$};
\foreach \i in {-7,-6,-5,-4,-3,-2,-1,1,2,3,4,5,6,7}
	{\draw (\i,-.1) --++ (0,.2);}
\foreach \j in {1,2,...,8}
	\draw (-.1,\j) --++ (.2,0);

\draw (-7,7) -- (-4,4) -- (-2,6) -- (0,4) -- (1,5) -- (2,4) -- (3,5) -- (4,4) -- (7,7);

\draw[green] (-8,2)node[left]{$\myLk{+}{2}$} -- (4,6) ;

\draw[very thick,red] (-7,7) -- (-5,5) -- (-2,6) -- (-.5,4.5) -- (1,5) -- (2,4) -- (3,5) -- (4,4) -- (7,7);
\draw[dotted, very thick, red] (-7.7,7.7) -- (-7,7);
\draw[dotted, very thick, red] (7,7) -- (7.7,7.7) node[below right]{$y = \applatitk{2}{\lambda}(x)$};

\draw[dashed] (-4,4) -- (0,0) -- (4,4);
\draw[dashed] (-3,5) -- (0,2) -- (2,4) --++ (1,-1);
\draw[dashed] (-3,3) --++ (2,2);
\draw[dashed] (-2,2) --++ (2,2) --++ (2,-2);
\draw[dashed] (-1,1) --++ (1,1) --++ (1,-1);

\begin{scope}[densely dotted, thick]

\color{red}
\draw (-2,6) -- (-2,0) node[below]{$\scriptstyle c_1$};
\fill (-2,0) circle (\sizecorner);

\foreach \i in {1}
	{
	\draw (\i,5) -- (\i,0);
	\fill (\i,0) circle (\sizecorner);
	}
\draw (1,0) node[below]{$ c_2$};

\color{blue}
\foreach \i in {-4,0}
	{
	\draw (\i,4) -- (\i,0);
	\fill (\i,0) circle (\sizecorner);
	}
	
\draw (-4,0) node[below]{$\scriptstyle i_1$};
\draw (0,-1) node[below]{$i_2$};
\color{orange}

\draw (-.5,4.5) -- (-.5,0) node[below]{$\hp[2]$};
\fill (-.5,0) circle (\sizecorner);

\end{scope}
\end{tikzpicture}
\end{center}
\caption{The maps $\applatitk{1}{\lambda}$ and $\applatitk{2}{\lambda}$ (in thick red) for the $3$-regular partition $\lambda = (4,4,2,1)$}
\label{figure:applatitk}
\end{figure}

\begin{proposition}
\label{proposition:applatit_corners}
\begin{itemize}
\item We have $\hp \in [i_k-1, i_k)$.
\item The functions $\applatitk{k-1}{\lambda}$ and $\applatitk{k}{\lambda}$ coincide  (at least) outside $(i_k-1,c_k)$, and for any $s \in (i_k-1,c_k)$ we have $0 \leq \applatitk{k}{\lambda}(s) -\applatitk{k-1}{\lambda}(s) \leq e$.
\end{itemize}
\end{proposition}

\begin{proof}
Following the curve of $\omega_\lambda$ to the left from the point $\bigl(c_k,\omega_\lambda(c_k)\bigr)$, since $\dire < 1$ we can cross $\myL{+}$ only after an inner corner, that is, since $i_k < c_k$ is the inner corner preceding $c_k$ (cf. Lemma~\ref{lemma:corners_intertwins}) then $\hp < i_k$. Moreover, by definition of an inner corner we have $\omega_\lambda(i_k-1) = \omega_\lambda(i_k)+1$, thus the line joining $\bigl(c_k,\omega_\lambda(c_k)\bigr)$ and $\bigl(i_k-1,\omega_\lambda(i_k-1)\bigr)$ has slope:
\[
\alpha \coloneqq \frac{\omega_\lambda(c_k)-\omega_\lambda(i_k)-1}{c_k-i_k+1}.
\]
Recalling from Lemma~\ref{lemma:corners_intertwins} that
$\omega_\lambda(i_k) = \omega_\lambda(c_k) - c_k+i_k$, we obtain:
\[
\alpha =  \frac{c_k-i_k-1}{c_k-i_k+1} = 1 - \frac{2}{c_k-i_k+1}.
\]
But now $\lambda$ is $e$-regular thus by Lemma~\ref{lemma:corners_intertwins} we have  $i_k \geq c_k-e+1$. We thus have $c_k-i_k+1 \leq e$ and thus, since $c_k-i_k+1 > 0$ (since $i_k \leq c_k$) we obtain:
\[
\alpha \leq 1-\frac{2}{e} = \dire.
\]
Thus, the line $\myL{+}$ meets the curve of $\omega_\lambda$ somewhere between the points of abscissa $i_k-1$ (included) and $i_k$ (excluded), thus $i_k-1 \leq \hp < i_k$ as announced. 

By construction, the maps $\applatitk{k}{\lambda}$ and $\applatitk{k-1}{\lambda}$ differ only on $(\hp,c_k)$. By the preceding inequality, we have $(i_k-1,c_k) \supseteq (\hp,c_k)$ thus these two functions coincide outside $(i_k-1,c_k)$ indeed.
Moreover, since $\dire < 1$ we have $\applatitk{k}{\lambda}(s) > \applatitk{k-1}{\lambda}(s)$ for any $s \in (\hp,c_k)$ thus we obtain the first member of the announced inequality. To prove the second one, since $\dire > -1$ (resp. $\dire < 1$) we have that $\applatitk{k}{\lambda} - \applatitk{k-1}{\lambda}$ is increasing (resp. decreasing) on $[\hp,i_k]$ (resp. $[i_k,c_k]$) thus  reaches its maximum at $i_k$. We now note the following two points.
\begin{itemize}
\item Since $\dire \geq 0$ and $\hp < i_k < c_k$, we have $\applatitk{k}{\lambda}(i_k) \leq \applatitk{k}{\lambda}(c_k)$ since $\applatitk{k}{\lambda}$ has slope $\dire$ on $(\hp,c_k)$.
\item The maps $\applatitk{k}{\lambda}$ and $\applatitk{k-1}{\lambda}$ coincide on $(c_k,+\infty)$ thus since $c_1 < \dots < c_k$ we obtain by induction that:
\begin{equation}
\label{equation:applatitk_leq_ck}
\applatitk{k}{\lambda}(s) = \omega_\lambda(s), \qquad \text{for all } s \geq c_k.
\end{equation}
Now if $k \geq 2$ we obtain $\applatitk{k-1}{\lambda}(i_k) = \omega_\lambda(i_k)$ since $i_k > c_{k-1}$ by Lemma~\ref{lemma:corners_intertwins}, and this equality remains valid if $k = 1$ since $\applatitk{0}{\lambda} = \omega_\lambda$ by definition.
\end{itemize}
Thus, by Lemma~\ref{lemma:corners_intertwins} we obtain:
\begin{align*}
\applatitk{k}{\lambda}(i_k) - \applatitk{k-1}{\lambda}(i_k)
&\leq 
 \applatitk{k}{\lambda}(c_k) - \omega_\lambda(i_k)
\\
&\leq
\omega_\lambda(c_k) - \omega_\lambda(i_k)
\\
&\leq 
c_k - i_k
\\
&\leq e-1
\end{align*}
whence the result.
\end{proof}

\begin{definition}
\label{definition:applatit+}
We define $\applatit{\lambda} \coloneqq \applatitk{r}{\lambda}$ and we define $\tapplatit{\lambda}$ by $\tapplatit{\lambda}(s) \coloneqq \frac{1}{\sqrt n}\applatit{\lambda}(s\sqrt n)$ for all $s \in \R$.
\end{definition}

Note that the rescaling for $\tapplatit{\lambda}$ is the same as for $\tomega_\lambda$ (see~\eqref{equation:rescaling}). Note also that by Proposition~\ref{proposition:applatit_corners}, the map $\applatit{\lambda}$ does not in fact depend on the order that we chose on  the outer corners.  An example of a map $\applatit{\lambda}$ is given in Figure~\ref{figure:applatit+}. Note that from Lemma~\ref{lemma:corners_intertwins} and Proposition~\ref{proposition:applatit_corners} for $k \geq 2$ we have $\hp \in (c_{k-1},c_k)$.

\begin{figure}[h]
\begin{center}
\begin{tikzpicture}[scale=.6]
\draw[-stealth] (-8,0) -- (8,0) node[below]{$x$};
\draw[-stealth] (0,-1) -- (0,9) node[right]{$y$};
\foreach \i in {-7,-6,-5,-4,-3,-2,-1,1,2,3,4,5,6,7}
	{\draw (\i,-.1) --++ (0,.2);}
\foreach \j in {1,2,...,8}
	\draw (-.1,\j) --++ (.2,0);

\draw (-7,7) -- (-4,4) -- (-2,6) -- (0,4) -- (1,5) -- (2,4) -- (3,5) -- (4,4) -- (7,7);

\draw[green] (-8,4) node[left]{$\myLk{+}{1}$} -- (4,8);
\draw[green] (-8,2)node[left]{$\myLk{+}{2}$} -- (4,6) ;
\draw[green] (-8,1.3333)node[below left]{$\myLk{+}{3}$} -- (4,5.3333) ;

\draw[very thick,red] (-7,7) -- (-5,5) -- (-2,6) -- (-.5,4.5) -- (1,5) -- (1.5,4.5) -- (3,5) -- (4,4) -- (7,7);
\draw[dotted, very thick, red] (-7.7,7.7) -- (-7,7);
\draw[dotted, very thick, red] (7,7) -- (7.7,7.7) node[below right]{$y = \applatit{\lambda}(x)$};

\draw[dashed] (-4,4) -- (0,0) -- (4,4);
\draw[dashed] (-3,5) -- (0,2) -- (2,4) --++ (1,-1);
\draw[dashed] (-3,3) --++ (2,2);
\draw[dashed] (-2,2) --++ (2,2) --++ (2,-2);
\draw[dashed] (-1,1) --++ (1,1) --++ (1,-1);

\begin{scope}[densely dotted, thick]

\color{red}
\draw (-2,6) -- (-2,0) node[below]{$\scriptstyle c_1$};
\fill (-2,0) circle (\sizecorner);

\foreach \i in {1,3}
	{
	\draw (\i,5) -- (\i,0);
	\fill (\i,0) circle (\sizecorner);
	}
\draw (1.2,0) node[below left]{$ \scriptstyle c_2$};
\draw (3,0) node[below]{$\scriptstyle c_3$};

\color{blue}
\foreach \i in {-4,0,2,4}
	{
	\draw (\i,4) -- (\i,0);
	\fill (\i,0) circle (\sizecorner);
	}
	
\draw (-4,0) node[below]{$\scriptstyle i_1$};
\draw (0,-1) node[below]{$\scriptstyle i_2$};
\draw (2.2,0) node[below] {$\scriptstyle i_3$};
\draw (4,0) node[below]{$\scriptstyle i_4$};
\color{orange}
\draw (-5,5) -- (-5,0) node[below]{$\hp[1]$};
\fill (-5,0) circle (\sizecorner);

\draw (-.5,4.5) -- (-.5,0) node[below]{$\hp[2]$};
\fill (-.5,0) circle (\sizecorner);

\draw (1.5,4.5) -- (1.5,0) node[below]{$\hp[3]$};
\fill (1.5,0) circle (\sizecorner);
\end{scope}
\end{tikzpicture}
\end{center}
\caption{The map $\applatit{\lambda}$ (in thick red) for the $3$-regular partition $\lambda = (4,4,2,1)$. In this case we have $\applatit{\lambda} = \applatitk{3}{\lambda}$.}
\label{figure:applatit+}
\end{figure}
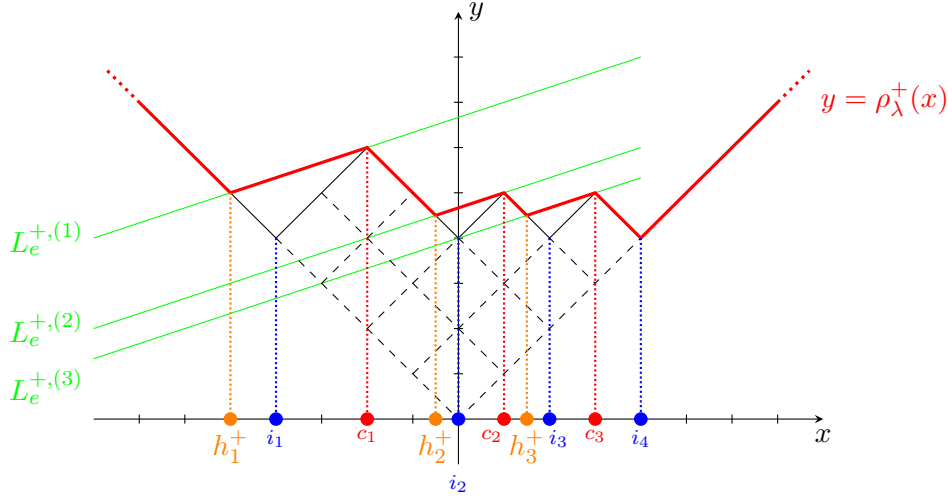

\begin{proposition}
\label{proposition:applatit+}
Recall that $\lambda$ is $e$-regular.
\begin{enumerate}
\item For all $s \in \R$ we have $0 \leq \applatit{\lambda}(s) - \omega_\lambda(s) \leq e$.
\item We have $\applatit{\lambda}(s) = |s|$ outside $(i_1 - 1,i_{r+1})$.
\item  The graph $G_\lambda^+ \coloneqq \gr\bigl(\applatit{\lambda}\bigr)$ is stable under the shaking operation with slope $\dire$, that is, we have $\sh(G_\lambda^+) = G_\lambda^+$.
\end{enumerate}
\end{proposition}

\begin{proof}
The first point follows directly from Proposition~\ref{proposition:applatit_corners}, recalling that $\applatitk{0}{\lambda} = \omega_\lambda$. The second point follows from the inequalities $i_1 - 1 \leq h_1$ and $h_r \leq c_r$ of Proposition~\ref{proposition:applatit_corners},  recalling that $c_r \leq i_{r+1}$ by Lemma~\ref{lemma:corners_intertwins}, together with~\eqref{subequations:omega_abs}.

To prove the last point, note that it follows directly from Definition~\ref{definition:applatit+}  and from the convexity of (the epigraph of) the absolute value that $\applatit{\lambda} \in \partclass$, thus $\gr(\applatit{\lambda})$ is well-defined. The proof of the second point follows from the simple observation that the slopes of $\applatit{\lambda}$ are either $-1$ or $\dire$ (except the part $\applatit{\lambda}(s) = s$ for $s \gg 0$). 
As a consequence,  if a line $L'$ of slope $\dire$ intersects the curve of $\applatit{\lambda}$ at a point of abscissa $s \in [\su[\applatit{\lambda}],\sd[\applatit{\lambda}]]$ then $L'$ remains below the curve $\applatit{\lambda}$ on $(\su[\applatit{\lambda}],s)$. Observing that $L'$ intersects the curve of $|\cdot|$ on $(-\infty,0)$ since $\dire < 1$ gives the result.
\end{proof}

\subsection{Inner flattening}
\label{subsection:inner_regularisation}
We now define a similar construction as $\applatit{\lambda}$ but for inner corners. We give the statements without proofs since they are entirely similar.

Let $\lambda$ be an $e$-regular partition. Let $\{i_1 < \dots < i_{r+1}\}$  be the set of inner corners where $r \geq 0$ and define $\applatitk[-]{r+1}{\lambda} \coloneqq \omega_\lambda$.  If $r \geq 1$, assuming by decreasing induction on $k \in \{1,\dots,r\}$ that we have constructed the piecewise linear function $\applatitk[-]{k+1}{\lambda} : \R \to \R$,  we construct the piecewise linear function $\applatitk[-]{k}{\lambda} : \R \to \R$ as follows. Let $\myL{-}$ be the affine line $\bigl(i_k,\omega_\lambda(i_k)\bigr) + L_e$.
\begin{enumerate}
\item For $s \leq i_k$ then $\applatitk[-]{k}{\lambda}(s) \coloneqq \applatitk[-]{k+1}{\lambda}(s)$.
\item
At $s = i_k$, we follow the line $\myL{-}$ (in the \emph{positive} direction) until we meet the curve of $\applatitk[-]{k+1}{\lambda}$, at the point of abscissa $\hm$.
\item For $s \geq \hm$ then $\applatitk[-]{k}{\lambda}(s) \coloneqq \applatitk[-]{k+1}{\lambda}(s)$ again.
\end{enumerate}

\begin{proposition}
\label{proposition:applatit_inner_corners}
\begin{itemize}
\item For any $k \in \{1,\dots,r\}$ we have $\hm \in (i_k,i_{k+1}]$.
\item The functions $\applatitk[-]{k+1}{\lambda}$ and $\applatitk[-]{k}{\lambda}$ coincide (at least) outside $(i_k,i_{k+1})$, and for any $s \in (i_k,i_{k+1})$ we have $-e \leq \applatitk[-]{k}{\lambda}(s) -\applatitk[-]{k+1}{\lambda}(s) \leq 0$.
\end{itemize}
\end{proposition}

\begin{definition}
\label{definition:applatit-}
We define $\applatit[-]{\lambda} \coloneqq \applatitk[-]{1}{\lambda}$ and we define $\tapplatit[-]{\lambda}$ by $\tapplatit[-]{\lambda}(s) \coloneqq \frac{1}{\sqrt n}\applatit[-]{\lambda}(s\sqrt n)$ for all $s \in \R$.
\end{definition}

An example of a map $\applatit[-]{\lambda}$ is given in Figure~\ref{figure:applatit-}.

\begin{figure}[h]
\begin{center}
\begin{tikzpicture}[scale=.6]
\draw[-stealth] (-8,0) -- (8,0) node[below]{$x$};
\draw[-stealth] (0,-1) -- (0,9) node[right]{$y$};
\foreach \i in {-7,-6,-5,-4,-3,-2,-1,1,2,3,4,5,6,7}
	{\draw (\i,-.1) --++ (0,.2);}
\foreach \j in {1,2,...,8}
	\draw (-.1,\j) --++ (.2,0);

\draw (-7,7) -- (-4,4) -- (-2,6) -- (0,4) -- (1,5) -- (2,4) -- (3,5) -- (4,4) -- (7,7);

\draw[green] (-7,3) node[left]{$\myLk{-}{1}$} -- (7,7.6666);
\draw[green] (-7,1.6666)node[left]{$\myLk{-}{2}$} -- (7,6.3333);
\draw[green] (-7,1) node[below left]{$\myLk{-}{3}$} -- (7,5.6666);

\draw[very thick,blue] (-7,7) -- (-4,4) -- (-1,5) -- (0,4) -- (1.5,4.5) -- (2,4) -- (3.5,4.5) -- (4,4) -- (7,7);
\draw[dotted, very thick, blue] (-7.7,7.7) -- (-7,7);
\draw[dotted, very thick, blue] (7,7) -- (7.7,7.7) node[below right]{$y = \applatit[-]{\lambda}(x)$};

\draw[dashed] (-4,4) -- (0,0) -- (4,4);
\draw[dashed] (-3,5) -- (0,2) -- (2,4) --++ (1,-1);
\draw[dashed] (-3,3) --++ (2,2);
\draw[dashed] (-2,2) --++ (2,2) --++ (2,-2);
\draw[dashed] (-1,1) --++ (1,1) --++ (1,-1);

\begin{scope}[densely dotted, thick]

\color{blue}
\foreach \i in {-4,0,2,4}
	{
	\draw (\i,4) -- (\i,0);
	\fill (\i,0) circle (\sizecorner);
	}
\draw (-4,0) node[below]{$\scriptstyle i_1$};
\draw (-.3,0) node[below]{$\scriptstyle i_2$};
\foreach \i in {3,4}
	\draw ({2*(\i -2)},0) node[below]{$\scriptstyle i_{\i}$};
	
\color{orange}
\draw (-1,5) -- (-1,0);
\draw (-.8,0) node[below left]{$\hm[1]$};
\fill (-1,0) circle (\sizecorner);

\draw (1.5,4.5) -- (1.5,0);
\draw (1.7,0) node[below left]{$\hm[2]$};
\fill (1.5,0) circle (\sizecorner);

\draw (3.5,4.5) -- (3.5,0);
\draw (3.7,0) node[below left]{$\hm[3]$};
\fill (3.5,0) circle (\sizecorner);
\end{scope}
\end{tikzpicture}
\end{center}
\caption{The map $\applatit[-]{\lambda}$ (in thick blue) for the $3$-regular partition $\lambda = (4,4,2,1)$}
\label{figure:applatit-}
\end{figure}
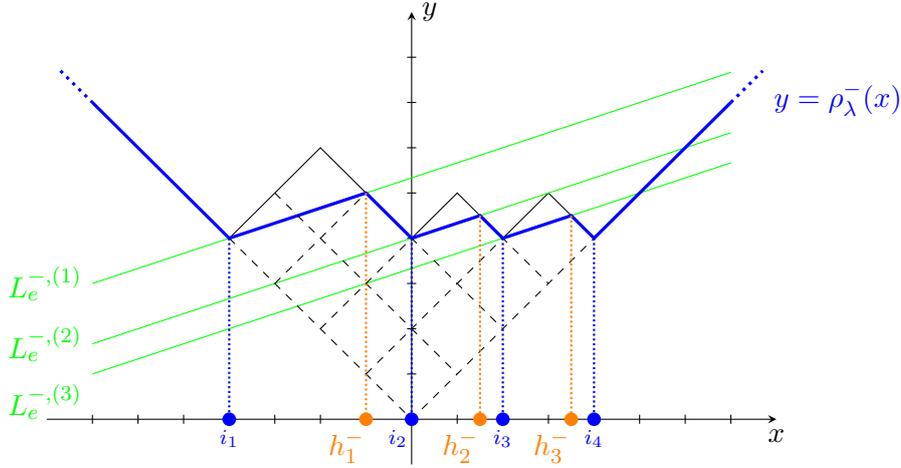

\begin{proposition}
\label{proposition:applatit-}
Recall that $\lambda$ is $e$-regular.
\begin{enumerate}
\item For all $s \in \R$ we have $-e \leq \applatit[-]{\lambda}(s) - \omega_\lambda(s) \leq 0$.
\item We have $\applatit[-]{\lambda}(s) \geq |s|$ for all $s \in \R$, with equality if and only if $s \notin (i_1,i_{r+1})$.
\item The graph $G_\lambda^- \coloneqq \gr\bigl(\applatit[-]{\lambda}\bigr)$ is stable under the shaking operation with slope $\dire$, that is, we have $\sh(G_\lambda^-) = G_\lambda^-$.
\end{enumerate}
\end{proposition}

\begin{proof}
We only prove that $\applatit[-]{\lambda}(s) = |s|$ if and only if $s \notin (i_1,i_{r+1})$, the remaining assertions being similar to the ones in Proposition~\ref{proposition:applatit+}. The statement is clear if $r = 0$ (that is, if $\lambda$ is empty), thus assume $r \geq 1$. By Proposition~\ref{proposition:applatit_inner_corners} it suffices to prove that $\applatit[-]{\lambda}(s) > |s|$ if $s \in (i_1,i_{r+1})$.

First note that, similarly to~\eqref{equation:applatitk_leq_ck}, by induction on $k \in \{1,\dots,r+1\}$ we have:
\begin{equation}
\label{equation:applatit-_omega}
\applatitk[-]{k}{\lambda}(s) = \omega_\lambda(s), \qquad \text{for all } s \leq i_k.
\end{equation}
Now let $s \in (i_1,i_{r+1})$ and let $k \in \{1,\dots,r\}$ be the unique integer such that $i_k \leq s < i_{k+1}$, so that by Proposition~\ref{proposition:applatit_inner_corners}   we have:
\begin{equation}
\label{equation:applatit_applatitk}
\applatit[-]{\lambda}(s) = \applatitk[-]{k}{\lambda}(s).
\end{equation} 
 If $s \in [\hm,i_{k+1})$ then:
\begin{align*}
\applatit[-]{\lambda}(s) &= \applatitk[-]{k}{\lambda}(s)
\\
&= \applatitk[-]{k+1}{\lambda}(s)
&&\text{since } s \geq \hm
\\
&=  \omega_\lambda(s)
&&\text{by~\eqref{equation:applatit-_omega}}
\\
&> |s|
&&\text{by~\eqref{subequations:omega_abs}},
\end{align*}
as desired.  We now prove this inequality on $[i_k,\hm]$. We have just proved that $\applatit[-]{\lambda}(\hm) > |\hm|$, moreover by~\eqref{equation:applatit-_omega} and~\eqref{equation:applatit_applatitk}  we have:
\[\applatit[-]{\lambda}(i_k) = \applatitk[-]{k}{\lambda}(i_k) = \omega_\lambda(i_k).
\]
Hence, if $k \geq 2$ we have $\applatit[-]{\lambda}(t) > |t|$ for all $t \in \{i_k,\hm\}$. By~\eqref{equation:applatit_applatitk}, we know that  the curve of $\applatit[-]{\lambda}(s)$ on $[i_k,\hm]$ is a line, thus by convexity of (the interior of the epigraph of) the absolute value we deduce that $\applatit[-]{\lambda}(t) > |t|$ for all $t \in [i_k,\hm]$, as desired. If $k = 1$ the same proof works on $(i_1,\hm[1]]$ and this concludes the proof.
\end{proof}

\section{Limit shape for regularisation of large partitions}
\label{section:limit_shape_regularisation}

We are now ready to gather all the previous results, to study the limit shape for the $e$-regularisation of large partitions taken under the Plancherel measure. 
 Recall from Definition~\ref{definition:Omegae} the definition of $\Omegae$ for $e \geq 2$.
 
\begin{theorem}
\label{theorem:limit_shape_regularisation}
Let $e \geq 2$. Under the Plancherel measure $\Pl$, the function $\tomega_{\reg(\lambda)}$ converges uniformly in probability to $\Omegae$ as $n \to +\infty$. In other words, for any $\varepsilon > 0$ we have
\[
\Pl\left(\sup_{\R} \,\bigl\lvert\tomega_{\reg(\lambda)}-\Omegae\bigr\rvert > \varepsilon\right) \xrightarrow{n\to+\infty} 0.
\]
\end{theorem}

\begin{proof}
Let $\varepsilon > 0$. By Proposition~\ref{proposition:diff_fepse_fe}, we can find $\eta \in (0,1)$ such that:
\begin{equation}
\label{equation:choice_eta}
\|\Omegaepse{\pm} - \Omegae\|_\infty \leq \frac{\varepsilon}{4},
\end{equation}
where:
\begin{equation}
\label{equation:Omegaepse}
\Omegaepse{\pm} \coloneqq \falphapsprim{(\Omegae)}{\pm} = \sh\bigl(\falphapsprim{\Omega}{\pm}\bigr)
\end{equation}
 (cf. Proposition~\ref{proposition:falphaeps}).
Define:
\begin{equation}
\label{equation:m^pm}
m^{\pm} \coloneqq \pm \left[\Omegaeps{\pm}\bigl(2\pm\tfrac{\myepsilon}{2}\bigr) - \Omega\bigl(2\pm\tfrac{\myepsilon}{2}\bigr)\right].
\end{equation}
By Lemma~\ref{lemma:feps}, we have $m^{\pm} > 0$. Now let $n \in \mathbb{Z}_{\geq 1}$ and let $M_n$ be the set of partitions $\lambda \in \partsn$ such that:
\begin{subequations}
\label{subequations:choice_Mn}
\begin{align}
\|\tomega_\lambda - \Omega\|_{\infty} &\leq \min\bigl(m^+,m^-\bigr),
\label{equation:norm_infini_tomega_Omega}
\\
\inf\{s \in \R : \tomega_\lambda(s) \neq |s|\} &\geq -2-\tfrac{\myepsilon}{2},
\label{equation:inf_tomega}
\\
\sup\{s \in \R : \tomega_\lambda(s) \neq |s|\} &\leq 2 + \tfrac{\myepsilon}{2}.
\label{equation:sup_tomega}
\end{align}
\end{subequations}
Note that by Theorem~\ref{theorem:LLN_partitions} we have:
\begin{equation}
\label{equation:cv_Pl_Mn}
\Pl(M_n) \to 1 \text{ as } n \to +\infty.
\end{equation}
Let $\lambda \in M_n$ and $s \in \R$. We will first prove that:
\begin{equation}
\label{equation:ineq_tomega}
\Omegaeps{-}(s) \leq \tomega_\lambda(s) \leq \Omegaeps{+}(s).
\end{equation} 
By~\eqref{equation:inf_tomega} and~\eqref{equation:sup_tomega}, we know that if $|s| \geq 2+\tfrac{\myepsilon}{2}$ then $\tomega_\lambda(s) = |s|$, thus~\eqref{equation:ineq_tomega} is satisfied by~\eqref{equation:f(s)=|s|} and Lemma~\ref{lemma:feps}. Now if $|s| \leq 2+\tfrac{\myepsilon}{2}$, by Lemma~\ref{lemma:feps}\ref{item:f+_-_f} and~\eqref{equation:norm_infini_tomega_Omega} we have:
\begin{align*}
\Omegaeps{+}(s) - \tomega_\lambda(s)
&=
\Omegaeps{+}(s) - \Omega(s) + \Omega(s) - \tomega_\lambda(s)
\\
&\geq
m^+ - \|\Omega - \tomega_\lambda\|_{\infty}
\\
&\geq
0,
\end{align*}
thus we have proved the second inequality of~\eqref{equation:ineq_tomega}. The proof of the first inequality is similar.

We now prove the uniform convergence in probability. Applying Lemma~\ref{lemma:inclusion_gr} to~\eqref{equation:ineq_tomega} gives:
\[
\gr\bigl(\Omegaeps{-}\bigr) \subseteq \gr\bigl(\tomega_\lambda\bigr) \subseteq \gr\bigl(\Omegaeps{+}\bigr).
\]
In the sequel, we write $\lambdareg \coloneqq \reg(\lambda)$  to simplify the expressions.
By Proposition~\ref{proposition:shaking_subset}, Theorem~\ref{theorem:sh_gr_f} and  Lemmas~\ref{lemma:salphaeta} and~\ref{lemma:sh_gr_reg}, we deduce that:
\begin{equation}
\label{equation:encastrement_Omegaeps}
\gr\bigl(\Omegaepse{-}\bigr) \subseteq \sh\bigl(\gr(\tomega_{\lambdareg})\bigr) \subseteq \gr\bigl(\Omegaepse{+}\bigr),
\end{equation}
recalling from~\eqref{equation:Omegaepse} that $\Omegaepse{\pm} = \falphapsprim{(\Omegae)}{\pm} = \sh\bigl(\falphapsprim{\Omega}{\pm}\bigr)$.
In particular, note that we can apply Theorem~\ref{theorem:sh_gr_f} for $\Omegaeps{\pm}$ indeed by Lemma~\ref{lemma:feps}.

Besides, by Propositions~\ref{proposition:applatit+} and~\ref{proposition:applatit-}, by~\eqref{equation:rescaling} and by Definitions~\ref{definition:applatit+} and~\ref{definition:applatit-}, for all $s \in \R$ we have the following inequalities:
\begin{subequations}
\label{subequations:tapplat_tomega}
\begin{gather}
\label{equation:tapplat-_leq_tomega}
\tapplat{-}(s) \leq \tomega_{\lambdareg}(s),
\\
\label{equation:tomega_leq_tapplat+}
\tomega_{\lambdareg}(s) \leq \tapplat{+}(s),
\end{gather}
\end{subequations}
together with:
\begin{subequations}
\label{subequations:tapplat_e_tomega}
\begin{gather}
\label{equation:tomega_leq_tapplat-}
\tomega_{\lambdareg}(s) \leq \tapplat{-}(s) + \frac{e}{\sqrt n},
\\
\label{equation:tapplat+_leq_tomega}
\tapplat{+}(s) \leq \tomega_{\lambdareg}(s) + \frac{e}{\sqrt n}.
\end{gather}
\end{subequations}
By~\eqref{subequations:tapplat_tomega} and Lemma~\ref{lemma:inclusion_gr}, we obtain:
\[
\gr\bigl(\tapplat{-}\bigr) \subseteq \gr(\tomega_{\lambdareg}) \subseteq \gr\bigl(\tapplat{+}\bigr).
\]
By Propositions~\ref{proposition:applatit+} and~\ref{proposition:applatit-} again and by Proposition~\ref{proposition:shaking_subset}, we deduce that:
\begin{equation}
\label{equation:encadrement_tapplatit}
\gr\bigl(\tapplat{-}\bigr) \subseteq \sh\bigl(\gr(\tomega_{\lambdareg})\bigr) \subseteq \gr\bigl(\tapplat{+}\bigr).
\end{equation}
 By~\eqref{equation:encastrement_Omegaeps} and~\eqref{equation:encadrement_tapplatit}, we deduce that:
\[
\gr\bigl(\Omegaepse{-}\bigr)  \subseteq \gr\bigl(\tapplat{+}\bigr), \qquad \text{and} \qquad
\gr\bigl(\tapplat{-}\bigr) \subseteq \gr\bigl(\Omegaeps{+}\bigr).
\]
By Lemma~\ref{lemma:inclusion_gr}, we obtain that for any $s \in \R$ we have:
\begin{equation}
\label{equation:Omegaespe_tapplat}
\Omegaepse{-}(s) \leq \tapplat{+}(s) \qquad \text{and} \qquad \tapplat{-}(s) \leq \Omegaepse{+}(s).
\end{equation}
We deduce that for any $s \in \R$ we have:
\begin{align*}
\tapplat{-}(s)
&\leq 
\Omegaepse{+}(s)
\\
&\leq \Omegaepse{-}(s) + \|\Omegaepse{-}-\Omegaepse{+}\|_{\infty}
\\
&\leq \tapplat{+}(s) +  \|\Omegaepse{-}-\Omegaepse{+}\|_{\infty}
\\
&\leq \tapplat{-}(s) + \frac{2e}{\sqrt n} + \|\Omegaepse{-}-\Omegaepse{+}\|_{\infty},
\end{align*}
where the last inequality follows from~\eqref{subequations:tapplat_e_tomega}. We deduce that for any $s \in \R$ we have:
\[
\bigl|\tapplat{-}(s) - \Omegaepse{+}(s)\bigr| \leq  \frac{2e}{\sqrt n} + \bigl\| \Omegaepse{-} - \Omegaepse{+}\bigr\|_\infty.
\]
Thus, by~\eqref{equation:tapplat-_leq_tomega} and~\eqref{equation:tomega_leq_tapplat-}, for any $s \in \R$ we have:
\begin{align*}
\bigl|\tomega_{\lambdareg}(s) - \Omegae(s)\bigr|
&
\begin{multlined}[t]
\leq
\bigl|\tomega_{\lambdareg}(s) - \tapplat{-}(s)\bigr|
+ \bigl|\tapplat{-}(s) - \Omegaepse{+}(s)\bigr|
\\
+ \bigl|\Omegaepse{+}(s) - \Omegae(s)\bigr|
\end{multlined}
\\
&\leq \frac{e}{\sqrt n} + \left(\frac{2e}{\sqrt n} +  \|\Omegaepse{-}-\Omegaepse{+}\|_{\infty}\right) + \|\Omegaepse{+}-\Omegae\|_\infty,
\end{align*}
whence:
\[
\|\tomega_{\lambdareg} - \Omegae\|_\infty \leq \frac{3e}{\sqrt n} + 2\|\Omegaepse{+} - \Omegae\|_\infty + \|\Omegaepse{-}-\Omegae\|_\infty.
\]
With $n \gg 0$ such that $\dfrac{3e}{\sqrt n} \leq \dfrac{\varepsilon}{4}$ and recalling~\eqref{equation:choice_eta} we thus obtain:
\[
\|\tomega_{\lambdareg} - \Omegae\|_\infty\leq \varepsilon.
\]
Finally, for $n \gg 0$ we have, recalling~\eqref{equation:cv_Pl_Mn},
\[
\Pl\bigl(\|\tomega_{\lambdareg} - \Omegae\|_\infty \leq \varepsilon\bigr) \geq  \Pl(M_n) \xrightarrow{n \to +\infty} 1,
\]
whence the desired convergence.
\end{proof}

The convergence of Theorem~\ref{theorem:limit_shape_regularisation} is illustrated in Figure~\ref{figure:limit_shape_2_diagram} for $e = 2$  and $n = 1000$, and in Figure~\ref{figure:limit_shape_3_diagram} for $e = 3$ and $n = 5000$.  We also give an example for $e = 4$ and $n = 30000$ in Figure~\ref{figure:limit_shape_4_rim}, where we focus on the positive abscissa part.

\begin{figure}
\begin{center}
\includegraphics[width=\mywidth]{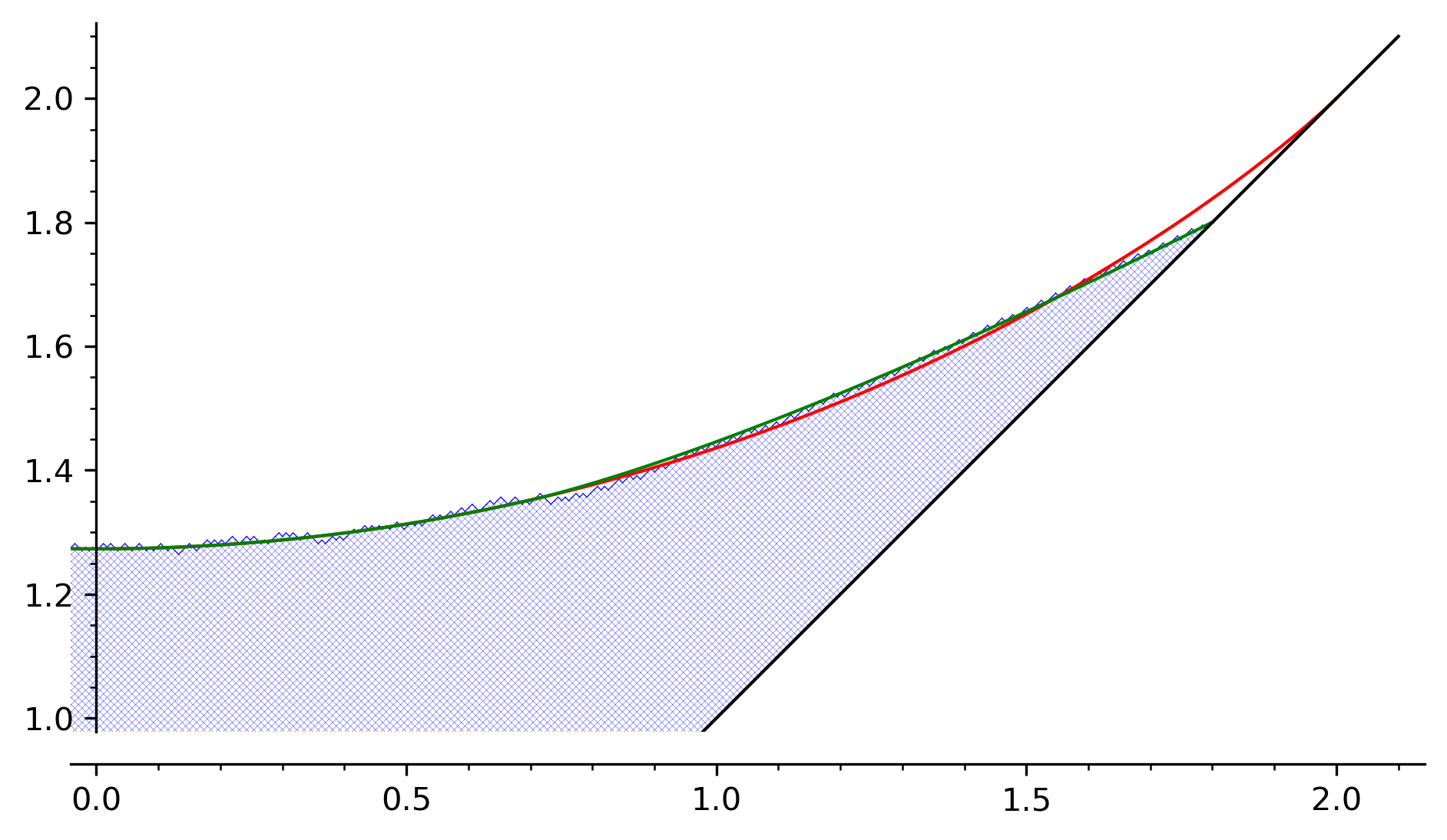}
\end{center}
\caption{Example of a (part of a) $4$-regularisation of a large partition taken under $\Pl[30000]$, with in green the limit shape $\Omegae[4]$ of Theorem~\ref{theorem:limit_shape_regularisation} and in red the limit shape $\Omega$ of Kerov-Vershik-Logan-Shepp}
\label{figure:limit_shape_4_rim}
\end{figure}

\begin{theorem}
\label{theorem:cv_supports}
The convergence of Theorem~\ref{theorem:limit_shape_regularisation} also holds for the supports, that is:
\begin{align*}
\inf\bigl\{s \in \R : \tomega_{\reg(\lambda)}(s) \neq |s|\bigr\} &\longrightarrow -2,
\\
\intertext{and:}
\sup\bigl\{s \in \mathbb R : \tomega_{\reg(\lambda)}(s) \neq |s|\bigr\} &\longrightarrow \suppfalpha[\dire](\Omega),
\end{align*}
in probability under $\Pl$ as $n \to +\infty$,  
\end{theorem}

\begin{proof}
Write $\suppfalpha$ instead of $\suppfalpha[\dire](\Omega)$. Note that by~\eqref{equation:suppfalpha_ineq} we have $\suppfalpha > 0$. Let $\varepsilon > 0$ and let $\myepsilon \in (0,1)$ such that:
\begin{equation}
\label{equation:eta_support}
\myepsilon \leq \min\left(\frac{\varepsilon}{2},\frac{2\varepsilon}{\suppfalpha}\right).
\end{equation}
Let $n \in \mathbb{Z}_{\geq 1}$ and let $M_n$ as in the proof of Theorem~\ref{theorem:limit_shape_regularisation}, that is, satisfying~\eqref{subequations:choice_Mn} where $m^{\pm}$ is defined in~\eqref{equation:m^pm}. Let $\lambda \in M_n$ and define $\mu \coloneqq \reg(\lambda)$. Recall from~\eqref{equation:Omegaespe_tapplat} that for all $s \in \R$ we have:
\[
\tapplatit[-]{\mu}(s) \leq \Omegaepse{+}(s).
\]
Thus, if $\{i_1 < \dots < i_{r+1}\}$ are the inner corners of $\mu$, we deduce from Proposition~\ref{proposition:applatit-} together with Lemma~\ref{lemma:feps} and Corollary~\ref{corollary:support_fe} that:
\[
-2-\myepsilon \leq \frac{i_1}{\sqrt n} \qquad \text{and} \qquad \frac{i_{r+1}}{\sqrt n} \leq \suppfalpha[\alpha_e]\bigl(\Omegaeps{+}\bigr).
\]
By Lemma~\ref{lemma:salphaeta} we deduce that:
\[
-2-\myepsilon - \frac{1}{\sqrt n} \leq \frac{i_1-1}{\sqrt n} \qquad \text{and} \qquad \frac{i_{r+1}}{\sqrt n} \leq \bigl(1+\tfrac{\myepsilon}{2}\bigr)\suppfalpha.
\]
By Proposition~\ref{proposition:applatit+}, we thus have:
\[
\tapplat{+}(s) = |s|, \qquad \text{if } s \leq -2 - \myepsilon-\frac{1}{\sqrt n}
\text{ or } s \geq \bigl(1+\tfrac{\myepsilon}{2}\bigr)\suppfalpha.
\]
Recalling~\eqref{equation:tomega_leq_tapplat+}, for all $s \in \R$ we have $|s| \leq \tomega_\mu(s) \leq \tapplat{+}(s)$ thus we deduce that:
\[
\tomega_\mu(s) = |s|, \qquad \text{if } s \leq -2 - \myepsilon-\frac{1}{\sqrt n}
\text{ or } s \geq \bigl(1+\tfrac{\myepsilon}{2}\bigr)\suppfalpha.
\]
Finally, we have proved that:
\begin{subequations}
\label{subequations:inf_geq}
\begin{align}
\inf\bigl\{s \in \R : \tomega_\mu(s) \neq |s|\bigr\} &\geq -2 - \myepsilon-\frac{1}{\sqrt n},
\\
\intertext{and:}
\sup\bigl\{s \in \R : \tomega_\mu(s) \neq |s|\bigr\} &\leq  \bigl(1+\tfrac{\myepsilon}{2}\bigr)\suppfalpha.
\end{align}
\end{subequations}

Similarly, from $\Omegaepse{-}(s) \leq \tapplat{-}(s)$ for all $s \in \R$ we deduce that:
\[
\frac{i_1-1}{\sqrt n} \leq -2+\myepsilon \qquad \text{and} \qquad \bigl(1-\tfrac{\myepsilon}{2}\bigr) \suppfalpha \leq \frac{i_{r+1}}{\sqrt n},
\]
hence recalling from~\eqref{equation:tapplat-_leq_tomega} the inequality $\tapplat{-}(s) \leq \tomega_\mu(s)$ for all $s \in \R$ and Proposition~\ref{proposition:applatit-} we know that if $s \in \bigl(-2+\myepsilon + \frac{1}{\sqrt n}, \bigl(1-\tfrac{\myepsilon}{2}\bigr) \suppfalpha \bigr)$ we have $|s| < \tapplat{-}(s)$ thus $|s| < \tomega_\mu(s)$ as well. We deduce that:
\begin{subequations}
\label{subequations:inf_leq}
\begin{align}
\inf\{s \in \R : \tomega_\mu(s) \neq |s|\} &\leq -2 + \myepsilon+\frac{1}{\sqrt n},
\\
\intertext{and:}
\sup\{s \in \R : \tomega_\mu(s) \neq |s|\} &\geq  \bigl(1-\tfrac{\myepsilon}{2}\bigr)\suppfalpha.
\end{align}
\end{subequations}
Hence, from~\eqref{subequations:inf_geq} and~\eqref{subequations:inf_leq} we obtain:
\begin{align*}
\bigl|\inf\{s \in \R : \tomega_\mu(s) \neq |s|\} + 2\bigr| &\leq \myepsilon + \frac{1}{\sqrt n},
\\
\intertext{and:}
\bigl|\sup\{s \in \R : \tomega_\mu(s) \neq |s|\} - \suppfalpha\bigr| &\leq \tfrac{\myepsilon}{2}\suppfalpha.
\end{align*}
With $n \gg 0$ such that $\frac{1}{\sqrt n} \leq \frac{\varepsilon}{2}$ and recalling~\eqref{equation:eta_support} we have $\myepsilon + \frac{1}{\sqrt n} \leq \varepsilon$ and $\frac{\myepsilon}{2}\suppfalpha \leq \varepsilon$, thus we obtain the desired convergences since $\Pl(M_n) \to 1$ as $n \to \infty$.
\end{proof}

Finally, we deduce the asymptotic behaviour of the first part and the first column.

\begin{corollary}
\label{corollary:length_first_column}
Let $e \geq 2$. Under the Plancherel measure $\Pl$:
\begin{enumerate}
\item  the rescaled size $\frac{1}{\sqrt n}\reg(\lambda)_1$ of the first row of $\reg(\lambda)$ converges as $n \to +\infty$ in probability to $2$;
\item the rescaled size $\frac{1}{\sqrt n}\reg(\lambda)'_1$ of the first column of $\reg(\lambda)$ converges as $n \to +\infty$ in probability to $\frac{2e}{\pi}\sin\frac{\pi}{e}$.
\end{enumerate}
\end{corollary}

\begin{proof}
The first point is clear by Theorem~\ref{theorem:cv_supports}. For the second one, by the same theorem it suffices to prove that the announced value is equal to $\suppfalpha = \suppfalpha(\Omega)$, with $\alpha = \dire = 1-2e^{-1}$. By Corollary~\ref{corollary:support_fe} applied for $f = \Omega$ we have:
\[
\suppfalpha = (1-\alpha)^{-1}\Bigl[\Omega\bigl(\Omega'^{-1}(\alpha)\bigr)-\alpha \Omega'^{-1}(\alpha)\Bigr].
\]
Recall from Lemma~\ref{lemma:derivative} that $\Omega'(s) = \frac{2}{\pi}\arcsin(\frac{s}{2})$ for $s \in (-2,2)$. Hence, for $s \in (-2,2)$ and $t \in (-1,1)$ we have $\Omega'(s) = t \iff \arcsin(\frac{s}{2}) = \frac{\pi t}{2} \iff s = 2\sin \frac{\pi t}{2}$ thus $\Omega'^{-1}(t) = 2\sin\frac{\pi t}{2}$. We thus have:
\begin{align*}
\Omega'^{-1}(\alpha) &= 2\sin\frac{\pi \alpha}{2}
\\
&=
2\sin\frac{\pi (1-2e^{-1})}{2}
\\
&=
2\sin\left(\frac{\pi}{2} - \frac{\pi}{e}\right)
\\
&=
2 \cos\frac{\pi}{e}.
\end{align*}
We obtain:
\begin{align*}
\suppfalpha
&=
\frac{e}{2}\Bigl[\Omega\bigl(\Omega'^{-1}(\alpha)\bigr) - \alpha \Omega'^{-1}(\alpha)\Bigr]
\\
&=
\frac{e}{2}\Bigl[\Omega\bigl(2\cos\tfrac{\pi}{e}\bigr) - 2\alpha \cos\tfrac{\pi}{e}\Bigr].
\end{align*}
Now we have, recalling the identity $\arcsin(\cos x) = \frac{\pi}{2} - x$ for $x \in [0,\pi]$:
\begin{align*}
\Omega\left(2\cos\frac{\pi}{e}\right)
&=
\frac{2}{\pi}\left[\arcsin\left(\cos\frac{\pi}{e}\right)2\cos\frac{\pi}{e} + \sqrt{4 - 4\cos^2\frac{\pi}{e}}\right]
\\
&=
\frac{2}{\pi}\left[\left(\frac{\pi}{2} - \frac{\pi}{e}\right)2\cos\frac{\pi}{e} + 2\sin\frac{\pi}{e}\right]
\\
&=
2\left(1 - \frac{2}{e}\right) \cos\frac{\pi}{e} + \frac{4}{\pi}\sin\frac{\pi}{e}
\\
&=
2\alpha\cos\frac{\pi}{e}+\frac{4}{\pi}\sin\frac{\pi}{e},
\end{align*}
thus we finally obtain:
\begin{equation}
\label{equation:suppfalpha_Omega}
\suppfalpha = \frac{2e}{\pi}\sin\frac{\pi}{e}.
\end{equation}
\end{proof}

\begin{example}
Here are approximations of the first values of the limit for $\frac{1}{\sqrt n}\reg(\lambda)'_1$:
\[
\begin{array}{c|c}
e & \tfrac{2e}{\pi}\sin\tfrac{\pi}{e}
\\[.5ex]
\hline
2 & 1.27
\\
3 & 1.65
\\
4 & 1.80
\end{array}
\]
The reader can check that they match the corresponding values of Figures~\ref{figure:limit_shape_2_diagram}, \ref{figure:limit_shape_3_diagram} and~\ref{figure:limit_shape_4_rim}.
\end{example}

\end{document}